\documentclass[11pt]{article}

\usepackage[a4paper,margin=1.1in]{geometry}
\usepackage{amsmath,amssymb,amsthm,mathtools}
\usepackage{mathrsfs}
\usepackage{enumitem}
\usepackage{aliascnt}
\usepackage[hidelinks]{hyperref}
\usepackage{cleveref}

\numberwithin{equation}{section}

\newtheorem{theorem}{Theorem}[section]

\newaliascnt{proposition}{theorem}
\newtheorem{proposition}[proposition]{Proposition}
\aliascntresetthe{proposition}

\newaliascnt{lemma}{theorem}
\newtheorem{lemma}[lemma]{Lemma}
\aliascntresetthe{lemma}

\newaliascnt{corollary}{theorem}
\newtheorem{corollary}[corollary]{Corollary}
\aliascntresetthe{corollary}

\theoremstyle{definition}
\newaliascnt{definition}{theorem}
\newtheorem{definition}[definition]{Definition}
\aliascntresetthe{definition}

\newaliascnt{assumption}{theorem}
\newtheorem{assumption}[assumption]{Assumption}
\aliascntresetthe{assumption}

\newaliascnt{convention}{theorem}
\newtheorem{convention}[convention]{Convention}
\aliascntresetthe{convention}

\newaliascnt{example}{theorem}

\aliascntresetthe{example}

\theoremstyle{remark}
\newaliascnt{remark}{theorem}
\newtheorem{remark}[remark]{Remark}
\aliascntresetthe{remark}

\crefname{assumption}{Assumption}{Assumptions}
\Crefname{assumption}{Assumption}{Assumptions}
\crefname{convention}{Convention}{Conventions}
\Crefname{convention}{Convention}{Conventions}
\crefname{definition}{Definition}{Definitions}
\Crefname{definition}{Definition}{Definitions}
\crefname{lemma}{Lemma}{Lemmas}
\Crefname{lemma}{Lemma}{Lemmas}
\crefname{proposition}{Proposition}{Propositions}
\Crefname{proposition}{Proposition}{Propositions}
\crefname{corollary}{Corollary}{Corollaries}
\Crefname{corollary}{Corollary}{Corollaries}
\crefname{theorem}{Theorem}{Theorems}
\Crefname{theorem}{Theorem}{Theorems}
\crefname{remark}{Remark}{Remarks}
\Crefname{remark}{Remark}{Remarks}

\newcommand{\R}{\mathbb{R}}

\newcommand{\cl}{\mathrm{cl}}
\newcommand{\loc}{\mathrm{loc}}

\newcommand{\Dist}{\operatorname{dist}}
\newcommand{\Image}{\operatorname{Im}}

\newcommand{\Rep}{\operatorname{Rep}}
\newcommand{\Tax}{\operatorname{Tax}}

\newcommand{\Sector}{\operatorname{Sector}}
\newcommand{\Err}{\mathsf{Err}}
\newcommand{\calC}{\mathcal C}
\newcommand{\calD}{\mathcal D}
\newcommand{\calK}{\mathcal K}
\newcommand{\calM}{\mathcal M}
\newcommand{\calP}{\mathcal P}
\newcommand{\calQ}{\mathcal Q}
\newcommand{\calR}{\mathcal R}
\newcommand{\calS}{\mathcal S}
\newcommand{\calU}{\mathcal U}
\newcommand{\calX}{\mathcal X}
\newcommand{\calY}{\mathcal Y}
\newcommand{\calZ}{\mathcal Z}

\title{Finite-Window Computational Anti-Phantom Theorems for Scale-Critical Navier--Stokes Defects}
\author{Runlong Yu\\
The University of Alabama, Tuscaloosa, AL, USA\\
\texttt{ryu5@ua.edu}}
\date{}

\begin{document}

\maketitle

\begin{abstract}
We prove a finite-window anti-phantom principle for scale-critical
Navier--Stokes defect packages and develop a conditional localized transfer
framework around it.  In a fixed clean quotient, the main compactness theorem
shows that if the only defect that is simultaneously invisible, exactly
reproducible, and tax-free is a gauge artifact, then observation,
reproduction failure, and tax control the clean quotient distance by a positive
finite-window gap.  We then isolate the additional localized inputs needed to
use this clean gap for Navier--Stokes packages: pressure-source observability,
enhanced pressure-tail geometry, chart visibility, and residual sub-budgets
for localization, reproduction, and gate/tax mismatch.  The localized results
are conditional finite-window reductions with explicit error constants,
including a comparison theorem between the enhanced-tail geometry and the
original intrinsic geometry under stated projection and harmonic tail
approximation assumptions.  The paper should be read as a rigorous
finite-window obstruction framework, not as a proof of Navier--Stokes
regularity, a construction of a singular solution, or a scale-uniform
regularity criterion.
\end{abstract}

\tableofcontents

\section{Introduction}

\subsection{The local regularity obstruction}

The local regularity theory for the three-dimensional incompressible
Navier--Stokes equations turns smallness of scale-critical quantities into
regularity.  This viewpoint goes back to the Leray--Hopf weak-solution
framework and to the partial-regularity theory of Scheffer and
Caffarelli--Kohn--Nirenberg, with later refinements and expositions such as
Lin's proof and Seregin's lecture notes
\cite{Leray1934,Hopf1951,Scheffer1976,Scheffer1977,CKN1982,Lin1998,SereginLectureNotes}.
A potential singular mechanism must therefore do more than make a single
critical norm large.  It must persist through a sequence of scales while
evading the observations and ledger terms that normally expose pressure,
flux, energy, dissipation, and trace defects.

There are two levels of that obstruction.  The first is clean: after
quotienting out specified gauges, can a non-gauge defect remain invisible to
the finite-window detectors?  The second is localized and PDE-facing: what
pressure, tail, chart, and residual estimates would transfer the clean gap
back to a localized Navier--Stokes package?

The terminology of flux, budget, and defect is consistent with the local
energy-transfer viewpoint for weak Euler and Navier--Stokes flows
\cite{ConstantinETiti1994,Eyink1994,DuchonRobert2000}.  The pressure terms are
kept explicit because local pressure decomposition and pressure regularity are
central in Navier--Stokes partial regularity
\cite{SohrWahl1986,SereginSverak2002,SereginLectureNotes}.  Critical-space and
backward-uniqueness mechanisms provide another perspective on how
scale-critical information interacts with possible singular behavior
\cite{ESS2003}.

The word \emph{clean} means that cutoff leakage, harmonic-pressure artifacts,
periodization errors, and localization residuals have been removed from the
algebraic model or placed into explicit transfer hypotheses.  The word
\emph{finite-window} means that all spaces are finite-dimensional and only a
fixed dyadic range is considered.  The main clean theorem is therefore an
algebraic compactness theorem.  The later localized statements are conditional
reductions that name the remaining Navier--Stokes inputs.

The framework is motivated by two neighboring lines of work.  One line studies
one-component and anisotropic regularity criteria for the three-dimensional
Navier--Stokes equations
\cite{KukavicaZiane2006,CaoTiti2011,CheminZhang2016,CheminZhangZhang2017,HanLeiLiZhao2019,KangNguyen2023}.
Another line asks for quantitative regularity information from spatial
concentration or weak--strong uniqueness mechanisms
\cite{BarkerPrange2021,AlbrittonBarkerPrange2023}.  The present manuscript is
also a finite-window algebraic companion to the preceding harmonic-pressure,
strict-shadow, Schur-visibility, and defect-cascade formulations
\cite{YuOneComponent2026,YuStrict2026,YuSchur2026,YuInvisible2026}.

\subsection{Main contribution}

The main contribution of this paper is twofold.  First, we prove a clean
finite-window anti-phantom theorem.  We define a constrained clean defect space
\(\calK_\Lambda^{\cl}\), a gauge subspace
\(\Gamma_\Lambda^{\cl}\), and a quotient distance
\[
    \Dist_{\cl}(d,\Gamma_\Lambda^{\cl})
    =
    \inf_{\gamma\in\Gamma_\Lambda^{\cl}}\|d-\gamma\|_{\cl}.
\]
We then define a computational detector consisting of three parts:
\[
    \text{observation}
    \quad+\quad
    \text{reproduction residual}
    \quad+\quad
    \text{tax}.
\]
The theorem is stated for a gauge-compatible clean detector.  It says that if
the zero set of the induced quotient detector contains no nonzero quotient
class, then the detector controls the quotient distance by a positive
finite-window constant.

Second, we build a conditional localized transfer layer around the clean gap.
The localized part introduces pressure-source quotient observability,
enhanced pressure-tail geometry, a chart-visibility decomposition, and four
explicit residual sub-budgets: chart mismatch, localization leakage,
reproduction drift, and gate/tax mismatch.  These pieces are assembled into a
finite-window enhanced-tail residual budget.  We then prove a conditional
comparison between the enhanced-tail distance and the original intrinsic
distance, assuming projection and harmonic tail approximation on a common
intrinsic representative.

The innovation is not the claim that these localized hypotheses are automatic.
It is the organization of the obstruction: the manuscript separates the
finite-window algebra that is proved from the PDE-facing compatibility
estimates that remain open.

\subsection{Main theorem narrative and theorem status}

The theorem should not be read as a proof of global regularity for
Navier--Stokes.  It also does not construct a singular solution, does not
claim that singularities are computers, and does not assert any form of
undecidability.  Its role is narrower: it proves a fixed clean finite-window
gap and records conditional mechanisms by which that gap could be transferred
to localized defect packages.

The proof narrative has four layers.  First, the clean finite-window theorem uses a detector that descends to the
quotient by gauge directions and then applies finite-dimensional compactness to
obtain a positive detector gap.  Second, a conditional
local-to-clean transfer theorem records how such a clean gap would imply a
localized lower bound in the presence of explicit comparison and residual
budget hypotheses.  Third, the pressure-source and enhanced-tail sections
decompose those hypotheses into concrete finite-window components.  Fourth,
the final comparison theorem explains when the enhanced-tail geometry is
controlled by the original intrinsic geometry, up to explicit projection and
harmonic tail errors.

Every localized conclusion in this manuscript is conditional on named
structural inputs.  No decay of projection tails as \(N\to\infty\), no decay of
harmonic tails as \(M\to\infty\), no pressure/tax coercivity theorem, and no
scale-uniform moving-window estimate is proved.

\subsection{Organization of the paper}

\Cref{sec:clean-spaces} defines clean finite-window defect spaces, constraints,
gauge directions, and quotient distance.  \Cref{sec:quotient-geometry} records
the compactness of the quotient unit sphere.  \Cref{sec:detectors} defines
observation, reproduction, tax, and the computational anti-phantom constant.
\Cref{sec:main-theorem} proves the kernel-free characterization and the clean
finite-window computational anti-phantom inequality.
\Cref{sec:finite-state-sector} derives a finite-state sector interpretation of
the gap.  \Cref{sec:localized-transfer-corollary} records the conditional
localized transfer corollary obtained by using
\(\mu_\Lambda^{\rm comp}\) as the clean gap constant.
\Cref{sec:failure-modes} classifies the ways in which the clean mechanism can
fail to produce a localized conclusion.
\Cref{sec:pressure-source-observability} develops the pressure-source,
enhanced-tail, residual-budget, and intrinsic-comparison components needed for
the localized transfer branch.  \Cref{sec:roadmap} records the next
PDE-facing derivation targets and the limitations of the present manuscript.

\section{Clean Finite-Window Defect Spaces}
\label{sec:clean-spaces}

\subsection{Finite dyadic windows}

Fix a finite dyadic window
\[
    \Lambda=\{k_0,k_0+1,\ldots,k_0+L\},
    \qquad L<\infty,
\]
and write
\[
    \Lambda_{\rm adj}
    :=
    \{k\in\Lambda:\ k+1\in\Lambda\}.
\]
All constants in this paper may depend on the fixed window \(\Lambda\), the
chosen finite-dimensional spaces, and the chosen norms.  No scale-uniform
statement in \(L\), \(k_0\), or the dyadic radius is asserted.

\subsection{Clean defect coordinates}

\begin{definition}[Clean finite-window defect space]
\label{def:clean-defect-space}
For each \(k\in\Lambda\), let \(\calD_k^{\cl}\) be a finite-dimensional real
normed vector space.  A clean defect at scale \(k\) is written schematically as
\[
    d_k=(U_k,P_k,R_k,\Pi_k,\Phi_k,\tau_k)\in\calD_k^{\cl},
\]
where the coordinates represent velocity, pressure, covariance or Reynolds
stress, interscale flux, energy/trace, and tax or ledger-slack data.  The
clean finite-window defect space is
\[
    \calD_\Lambda^{\cl}
    :=
    \prod_{k\in\Lambda}\calD_k^{\cl}.
\]
It is equipped with a fixed norm \(\|\cdot\|_{\cl}\).
\end{definition}

\begin{definition}[Clean constraint space]
\label{def:clean-constraint-space}
The clean constraint space
\[
    \calK_\Lambda^{\cl}\subset\calD_\Lambda^{\cl}
\]
is a finite-dimensional linear subspace.  Its role is to encode the clean
finite-dimensional constraints that survive after localization artifacts have
been removed, such as projected divergence constraints, projected pressure
compatibility, projected momentum compatibility, and projected finite-window
energy-flux identities.
\end{definition}

\begin{remark}[Linear clean model first]
The present paper proves the main theorem in the linear clean model, where
\(\calK_\Lambda^{\cl}\) is a vector subspace.  Positivity constraints,
semialgebraic admissibility conditions, or Reynolds-stress cones can be added
later, but they are not needed for the compactness-based gap theorem below.
\end{remark}

\subsection{Gauge directions and quotient distance}

\begin{definition}[Clean gauge space]
\label{def:clean-gauge-space}
Let \(\calC_\Lambda^{\cl}\) be a finite-dimensional real normed vector space
and let
\[
    G_\Lambda^{\cl}:\calC_\Lambda^{\cl}\to\calD_\Lambda^{\cl}
\]
be a linear map.  The clean gauge subspace inside the constrained clean space is
\[
    \Gamma_\Lambda^{\cl}
    :=
    \calK_\Lambda^{\cl}\cap \Image G_\Lambda^{\cl}.
\]
Elements of \(\Gamma_\Lambda^{\cl}\) are the clean removable directions, such
as pressure mean gauges, selected harmonic-pressure gauges, finite projection
artifacts, and clean periodization artifacts represented at the finite-window
level.
\end{definition}

\begin{definition}[Clean quotient and quotient distance]
\label{def:clean-quotient-distance}
The clean quotient space is
\[
    \calQ_\Lambda^{\cl}
    :=
    \calK_\Lambda^{\cl}/\Gamma_\Lambda^{\cl}.
\]
For \(d\in\calK_\Lambda^{\cl}\), define
\[
    \Dist_{\cl}(d,\Gamma_\Lambda^{\cl})
    :=
    \inf_{\gamma\in\Gamma_\Lambda^{\cl}}\|d-\gamma\|_{\cl}.
\]
Equivalently,
\[
    \|[d]\|_{\calQ_\Lambda^{\cl}}
    :=
    \Dist_{\cl}(d,\Gamma_\Lambda^{\cl})
\]
is the quotient norm of the class \([d]\in\calQ_\Lambda^{\cl}\).
\end{definition}

\begin{remark}[Gauge status]
The quotient removes only the finite-dimensional clean gauge directions that
have been explicitly placed in \(\Gamma_\Lambda^{\cl}\).  It does not assert
that every localized Navier--Stokes artifact is gauge, and it does not prove a
localized gauge-cleaning theorem.
\end{remark}

\begin{convention}[Nontrivial quotient regime]
\label{conv:nontrivial-quotient}
The main gap statements are formulated in the case
\(\calQ_\Lambda^{\cl}\ne\{0\}\).  If
\(\calQ_\Lambda^{\cl}=\{0\}\), then every constrained clean defect is already a
clean gauge direction, so there is no non-gauge clean finite-window obstruction
to detect.  In that degenerate case the anti-phantom statement is vacuous
rather than a positive-gap theorem.
\end{convention}

\section{Quotient Geometry}
\label{sec:quotient-geometry}

\begin{lemma}[Quotient compactness]
\label{lem:quotient-compactness}
The quotient \(\calQ_\Lambda^{\cl}\) is finite-dimensional.  Its unit sphere
\[
    S_\Lambda^{\cl}
    :=
    \{[d]\in\calQ_\Lambda^{\cl}:\ \|[d]\|_{\calQ_\Lambda^{\cl}}=1\}
\]
is compact.
\end{lemma}

\begin{proof}
Since \(\calK_\Lambda^{\cl}\) is finite-dimensional and
\(\Gamma_\Lambda^{\cl}\subset\calK_\Lambda^{\cl}\) is a linear subspace, the
quotient \(\calQ_\Lambda^{\cl}\) is finite-dimensional.  Every norm on a
finite-dimensional vector space induces the usual finite-dimensional topology.
The unit sphere of a finite-dimensional normed space is closed and bounded,
and hence compact.
\end{proof}

\begin{lemma}[Distance normalization]
\label{lem:distance-normalization}
If \(d\in\calK_\Lambda^{\cl}\) and
\(\Dist_{\cl}(d,\Gamma_\Lambda^{\cl})>0\), then
\[
    e:=\frac{d}{\Dist_{\cl}(d,\Gamma_\Lambda^{\cl})}
\]
satisfies
\[
    \Dist_{\cl}(e,\Gamma_\Lambda^{\cl})=1.
\]
\end{lemma}

\begin{proof}
Because \(\Gamma_\Lambda^{\cl}\) is a linear subspace, quotient distance is
positively homogeneous:
\[
    \Dist_{\cl}(\lambda d,\Gamma_\Lambda^{\cl})
    =
    |\lambda|\Dist_{\cl}(d,\Gamma_\Lambda^{\cl})
\]
for every scalar \(\lambda\).  Taking
\(\lambda=\Dist_{\cl}(d,\Gamma_\Lambda^{\cl})^{-1}\) proves the claim.
\end{proof}

\section{Observation, Reproduction, and Tax Detectors}
\label{sec:detectors}

\subsection{Observation channels}

\begin{definition}[Clean observation detector]
\label{def:clean-observation-detector}
Let \(\calY_\Lambda^{P}\), \(\calY_\Lambda^{F}\),
\(\calY_\Lambda^{E}\), and \(\calY_\Lambda^{T}\) be finite-dimensional normed
observation spaces.  The clean observation detector is a linear map
\[
    O_\Lambda^{\cl}
    =
    (O_\Lambda^P,O_\Lambda^F,O_\Lambda^E,O_\Lambda^T):
    \calK_\Lambda^{\cl}\to
    \calY_\Lambda^P\times\calY_\Lambda^F
    \times\calY_\Lambda^E\times\calY_\Lambda^T.
\]
The four channels represent active pressure, interscale flux, energy or
dissipation, and selected trace or adjoint-trace observations.
\end{definition}

\subsection{Reproduction residual}

\begin{definition}[Clean reproduction residual]
\label{def:clean-reproduction-residual}
For each \(k\in\Lambda_{\rm adj}\), fix a linear map
\[
    R_k^{\cl}:\calD_k^{\cl}\to\calD_{k+1}^{\cl}.
\]
For \(d=(d_k)_{k\in\Lambda}\in\calK_\Lambda^{\cl}\), define
\[
    \Rep_\Lambda^{\cl}(d)
    :=
    \left(
    \sum_{k\in\Lambda_{\rm adj}}
    \|d_{k+1}-R_k^{\cl}d_k\|_{\cl,k+1}^2
    \right)^{1/2},
\]
where \(\|\cdot\|_{\cl,k+1}\) is a fixed norm on
\(\calD_{k+1}^{\cl}\).
\end{definition}

\begin{remark}[Meaning of reproduction]
The residual \(\Rep_\Lambda^{\cl}\) measures whether a clean defect is
compatible with the chosen adjacent-scale reproduction maps.  It does not say
that Navier--Stokes actually generates these maps, and it does not prove
scale-uniform reproduction.
\end{remark}

\subsection{Tax functional}

\begin{definition}[Clean tax functional]
\label{def:clean-tax-functional}
A clean tax functional is a continuous map
\[
    \Tax_\Lambda^{\cl}:\calK_\Lambda^{\cl}\to[0,\infty)
\]
which is positively homogeneous:
\[
    \Tax_\Lambda^{\cl}(\lambda d)=|\lambda|\Tax_\Lambda^{\cl}(d),
    \qquad \lambda\in\R.
\]
It records the finite-window cost assigned to dissipation, flux, pressure, or
ledger-slack channels.  In this first paper it is an abstract detector, not a
proved pressure-tax lower bound.
\end{definition}

\subsection{Computational detector and gap constant}

\begin{definition}[Computational detector size]
\label{def:computational-detector-size}
Fix constants \(C_R>0\) and \(C_T>0\).  Define
\[
    \calM_\Lambda^{\rm comp}(d)
    :=
    \|O_\Lambda^{\cl}d\|
    +
    C_R\Rep_\Lambda^{\cl}(d)
    +
    C_T\Tax_\Lambda^{\cl}(d),
    \qquad d\in\calK_\Lambda^{\cl}.
\]
\end{definition}

\begin{assumption}[Clean detector gauge compatibility]
\label{ass:clean-detector-gauge-compatibility}
The clean detector is constant on clean gauge cosets.  More precisely, for
every \(d\in\calK_\Lambda^{\cl}\) and every
\(\gamma\in\Gamma_\Lambda^{\cl}\),
\[
    O_\Lambda^{\cl}(d+\gamma)=O_\Lambda^{\cl}d,
    \qquad
    \Rep_\Lambda^{\cl}(d+\gamma)=\Rep_\Lambda^{\cl}(d),
    \qquad
    \Tax_\Lambda^{\cl}(d+\gamma)=\Tax_\Lambda^{\cl}(d).
\]
Equivalently, \(\calM_\Lambda^{\rm comp}\) is constant on every affine coset
\(d+\Gamma_\Lambda^{\cl}\).
\end{assumption}

\begin{remark}[Why gauge compatibility is included]
Continuity and positive homogeneity of \(\calM_\Lambda^{\rm comp}\), together
with the qualitative kernel-free condition, do not by themselves force a
positive gap on the set
\(\{d:\Dist_{\cl}(d,\Gamma_\Lambda^{\cl})=1\}\).  That set is not compact in
\(\calK_\Lambda^{\cl}\) when gauge directions are nontrivial, and a sequence may
escape along \(\Gamma_\Lambda^{\cl}\) while keeping unit quotient distance.  The
compactness proof below is therefore made on the quotient detector, which is
well defined only under \Cref{ass:clean-detector-gauge-compatibility} or after
choosing an equivalent canonical gauge-slice formulation.
\end{remark}

\begin{lemma}[Continuity and homogeneity of the quotient detector]
\label{lem:detector-continuity-homogeneity}
The map \(\calM_\Lambda^{\rm comp}\) is continuous and positively homogeneous:
\[
    \calM_\Lambda^{\rm comp}(\lambda d)
    =
    |\lambda|\calM_\Lambda^{\rm comp}(d).
\]
Under \Cref{ass:clean-detector-gauge-compatibility}, it descends to a
well-defined continuous positively homogeneous map
\[
    \overline{\calM}_\Lambda^{\rm comp}:
    \calQ_\Lambda^{\cl}\to[0,\infty),
    \qquad
    \overline{\calM}_\Lambda^{\rm comp}([d])
    :=
    \calM_\Lambda^{\rm comp}(d).
\]
\end{lemma}

\begin{proof}
The observation map is linear between finite-dimensional normed spaces and is
therefore continuous.  The reproduction residual is a finite sum of norms of
linear expressions in \(d\), hence is continuous and positively homogeneous.
The tax functional is continuous and positively homogeneous by definition.
The weighted sum of these three terms has the same properties.

If \(d'\) is another representative of \([d]\), then
\(d'=d+\gamma\) for some \(\gamma\in\Gamma_\Lambda^{\cl}\).  By
\Cref{ass:clean-detector-gauge-compatibility},
\(\calM_\Lambda^{\rm comp}(d')=\calM_\Lambda^{\rm comp}(d)\).  Hence the
quotient detector is well defined.  Choose any linear section
\(s:\calQ_\Lambda^{\cl}\to\calK_\Lambda^{\cl}\) of the quotient map.  Since the
spaces are finite-dimensional,
\(\overline{\calM}_\Lambda^{\rm comp}=\calM_\Lambda^{\rm comp}\circ s\) is
continuous.  Homogeneity is inherited from the homogeneity of
\(\calM_\Lambda^{\rm comp}\).
\end{proof}

\begin{definition}[Computational anti-phantom constant]
\label{def:computational-antiphantom-constant}
Assume the nontrivial quotient regime of \Cref{conv:nontrivial-quotient} and
\Cref{ass:clean-detector-gauge-compatibility}.  The clean finite-window
computational anti-phantom constant is
\[
    \mu_\Lambda^{\rm comp}
    :=
    \inf_{q\in S_\Lambda^{\cl}}
    \overline{\calM}_\Lambda^{\rm comp}(q).
\]
Equivalently, because the detector is constant on clean gauge cosets,
\[
    \mu_\Lambda^{\rm comp}
    =
    \inf_{\Dist_{\cl}(d,\Gamma_\Lambda^{\cl})=1}
    \calM_\Lambda^{\rm comp}(d).
\]
\end{definition}

\section{Clean Computational Anti-Phantom Theorem}
\label{sec:main-theorem}

\subsection{Kernel-free condition}

\begin{assumption}[Kernel-free computational detector modulo gauge]
\label{ass:kernel-free-computational-detector}
For every \(d\in\calK_\Lambda^{\cl}\),
\[
    O_\Lambda^{\cl}d=0,
    \qquad
    \Rep_\Lambda^{\cl}(d)=0,
    \qquad
    \Tax_\Lambda^{\cl}(d)=0
    \quad
    \Longrightarrow
    \quad
    d\in\Gamma_\Lambda^{\cl}.
\]
\end{assumption}

\begin{remark}[Status of the kernel-free assumption]
\Cref{ass:kernel-free-computational-detector} is the substantive clean
finite-window hypothesis.  It says that there is no non-gauge clean defect
which is simultaneously invisible, exactly reproducible, and tax-free.  The
compactness theorem below proves that this qualitative statement is equivalent
to a quantitative finite-window gap.
\end{remark}

\subsection{Kernel-free characterization}

\begin{theorem}[Kernel-free characterization of the computational gap]
\label{thm:kernel-free-characterization}
Assume \(\calQ_\Lambda^{\cl}\ne\{0\}\) and
\Cref{ass:clean-detector-gauge-compatibility}.
The following are equivalent.
\begin{enumerate}[label=(\roman*),leftmargin=2em]
    \item \(\mu_\Lambda^{\rm comp}>0\).
    \item The kernel-free condition in
    \Cref{ass:kernel-free-computational-detector} holds.
\end{enumerate}
\end{theorem}

\begin{proof}
First assume \(\mu_\Lambda^{\rm comp}>0\).  Let
\(d\in\calK_\Lambda^{\cl}\) satisfy
\[
    O_\Lambda^{\cl}d=0,
    \qquad
    \Rep_\Lambda^{\cl}(d)=0,
    \qquad
    \Tax_\Lambda^{\cl}(d)=0.
\]
Then \(\overline{\calM}_\Lambda^{\rm comp}([d])=0\).  If
\([d]\ne0\) in \(\calQ_\Lambda^{\cl}\), set
\[
    q=\frac{[d]}{\|[d]\|_{\calQ_\Lambda^{\cl}}}.
\]
By \Cref{lem:detector-continuity-homogeneity}, the quotient detector is
positively homogeneous, so
\(\overline{\calM}_\Lambda^{\rm comp}(q)=0\).  This contradicts the definition
of \(\mu_\Lambda^{\rm comp}\) as the infimum of the quotient detector on
\(S_\Lambda^{\cl}\).  Hence \([d]=0\), equivalently
\(d\in\Gamma_\Lambda^{\cl}\).

Conversely, assume \Cref{ass:kernel-free-computational-detector}.  Since
\(S_\Lambda^{\cl}\) is compact by \Cref{lem:quotient-compactness} and
\(\overline{\calM}_\Lambda^{\rm comp}\) is continuous by
\Cref{lem:detector-continuity-homogeneity}, the infimum defining
\(\mu_\Lambda^{\rm comp}\) is attained at some
\(q_\ast\in S_\Lambda^{\cl}\).  If \(\mu_\Lambda^{\rm comp}=0\), choose a
representative \(d_\ast\in\calK_\Lambda^{\cl}\) of \(q_\ast\).  Then
\[
    \calM_\Lambda^{\rm comp}(d_\ast)=
    \overline{\calM}_\Lambda^{\rm comp}(q_\ast)=0.
\]
Because all three detector terms are nonnegative, this gives
\[
    O_\Lambda^{\cl}d_\ast=0,
    \qquad
    \Rep_\Lambda^{\cl}(d_\ast)=0,
    \qquad
    \Tax_\Lambda^{\cl}(d_\ast)=0.
\]
By \Cref{ass:kernel-free-computational-detector},
\(d_\ast\in\Gamma_\Lambda^{\cl}\), so \(q_\ast=0\), contradicting
\(q_\ast\in S_\Lambda^{\cl}\).  Therefore
\(\mu_\Lambda^{\rm comp}>0\).
\end{proof}

\subsection{Quantitative anti-phantom inequality}

\begin{theorem}[Clean finite-window computational anti-phantom inequality]
\label{thm:clean-computational-antiphantom}
Assume \(\calQ_\Lambda^{\cl}\ne\{0\}\),
\Cref{ass:clean-detector-gauge-compatibility}, and
\Cref{ass:kernel-free-computational-detector}.  Then
\(\mu_\Lambda^{\rm comp}>0\), and every
\(d\in\calK_\Lambda^{\cl}\) satisfies
\[
    \|O_\Lambda^{\cl}d\|
    +
    C_R\Rep_\Lambda^{\cl}(d)
    +
    C_T\Tax_\Lambda^{\cl}(d)
    \ge
    \mu_\Lambda^{\rm comp}
    \Dist_{\cl}(d,\Gamma_\Lambda^{\cl}).
\]
\end{theorem}

\begin{proof}
The positivity of \(\mu_\Lambda^{\rm comp}\) follows from
\Cref{thm:kernel-free-characterization}.  If
\(\Dist_{\cl}(d,\Gamma_\Lambda^{\cl})=0\), then the right-hand side is zero
and the inequality follows from nonnegativity of the detector terms.  If
\(\Dist_{\cl}(d,\Gamma_\Lambda^{\cl})>0\), set
\[
    q=\frac{[d]}{\Dist_{\cl}(d,\Gamma_\Lambda^{\cl})}
    \in S_\Lambda^{\cl}.
\]
By definition of \(\mu_\Lambda^{\rm comp}\),
\[
    \overline{\calM}_\Lambda^{\rm comp}(q)
    \ge
    \mu_\Lambda^{\rm comp}.
\]
Using the positive homogeneity of the quotient detector and the identity
\(\overline{\calM}_\Lambda^{\rm comp}([d])=\calM_\Lambda^{\rm comp}(d)\) gives
\[
    \calM_\Lambda^{\rm comp}(d)
    \ge
    \mu_\Lambda^{\rm comp}
    \Dist_{\cl}(d,\Gamma_\Lambda^{\cl}),
\]
which is exactly the displayed inequality.
\end{proof}

\begin{remark}[Interpretation]
\Cref{thm:clean-computational-antiphantom} says that in a fixed clean quotient,
a non-gauge defect must pay in at least one of three ways: it is seen by the
combined observation channels, it fails to reproduce across the finite window,
or it pays positive tax.  This is a finite-dimensional anti-phantom theorem,
not a scale-uniform PDE theorem.
\end{remark}

\section{Finite-State Sector Interpretation}
\label{sec:finite-state-sector}

\subsection{Detector sectors}

The clean gap can be read as a finite-state alternative once the detector is
split into its component channels.  This section is only a bookkeeping
interpretation of the finite-dimensional theorem.  It does not assert universal
computation, undecidability, or any dynamics of Navier--Stokes itself.

\begin{definition}[Finite detector-sector set]
\label{def:detector-sector-set}
Let
\[
    \mathfrak S_\Lambda
    :=
    \{\mathsf P,\mathsf F,\mathsf E,\mathsf T,\mathsf R,\mathsf{Tax}\}.
\]
The sectors correspond respectively to pressure observation, flux
observation, energy or dissipation observation, trace observation,
reproduction failure, and tax.
\end{definition}

\begin{lemma}[Observation norm resolved by channels]
\label{lem:observation-channel-resolution}
There is a finite constant \(A_{\rm obs,\Lambda}<\infty\), depending only on
the fixed observation spaces and their chosen product norm, such that for
every \(d\in\calK_\Lambda^{\cl}\),
\[
    \|O_\Lambda^{\cl}d\|
    \le
    A_{\rm obs,\Lambda}
    \bigl(
        \|O_\Lambda^P d\|
        +
        \|O_\Lambda^F d\|
        +
        \|O_\Lambda^E d\|
        +
        \|O_\Lambda^T d\|
    \bigr).
\]
\end{lemma}

\begin{proof}
The target of \(O_\Lambda^{\cl}\) is the finite product
\[
    \calY_\Lambda^P\times\calY_\Lambda^F
    \times\calY_\Lambda^E\times\calY_\Lambda^T.
\]
On this finite-dimensional product, the chosen product norm is bounded by the
sum norm
\[
    \|(y_P,y_F,y_E,y_T)\|_1
    :=
    \|y_P\|+\|y_F\|+\|y_E\|+\|y_T\|.
\]
Thus there exists \(A_{\rm obs,\Lambda}<\infty\) such that
\[
    \|(y_P,y_F,y_E,y_T)\|
    \le
    A_{\rm obs,\Lambda}\|(y_P,y_F,y_E,y_T)\|_1
\]
for every element of the product.  Applying this to
\[
    (y_P,y_F,y_E,y_T)
    =
    (O_\Lambda^P d,O_\Lambda^F d,
     O_\Lambda^E d,O_\Lambda^T d)
\]
gives the displayed estimate.
\end{proof}

\begin{definition}[Sector detector amplitudes]
\label{def:sector-detector-amplitudes}
For \(d\in\calK_\Lambda^{\cl}\), define
\[
\begin{aligned}
    D_{\mathsf P}(d)
        &:= A_{\rm obs,\Lambda}\|O_\Lambda^P d\|,\\
    D_{\mathsf F}(d)
        &:= A_{\rm obs,\Lambda}\|O_\Lambda^F d\|,\\
    D_{\mathsf E}(d)
        &:= A_{\rm obs,\Lambda}\|O_\Lambda^E d\|,\\
    D_{\mathsf T}(d)
        &:= A_{\rm obs,\Lambda}\|O_\Lambda^T d\|,\\
    D_{\mathsf R}(d)
        &:= C_R\Rep_\Lambda^{\cl}(d),\\
    D_{\mathsf{Tax}}(d)
        &:= C_T\Tax_\Lambda^{\cl}(d).
\end{aligned}
\]
For a non-gauge defect, meaning
\(\Dist_{\cl}(d,\Gamma_\Lambda^{\cl})>0\), define
\(\Sector_\Lambda(d)\) to be the first sector in the fixed order
\[
    \mathsf P\prec \mathsf F\prec \mathsf E
    \prec \mathsf T\prec \mathsf R\prec \mathsf{Tax}
\]
at which the maximum of
\(\{D_\sigma(d):\sigma\in\mathfrak S_\Lambda\}\) is attained.
\end{definition}

\subsection{Finite-state alternative}

\begin{proposition}[Finite-state detector alternative]
\label{prop:finite-state-detector-alternative}
Assume \(\calQ_\Lambda^{\cl}\ne\{0\}\),
\Cref{ass:clean-detector-gauge-compatibility}, and
\Cref{ass:kernel-free-computational-detector}.  If
\(d\in\calK_\Lambda^{\cl}\) and
\(\Dist_{\cl}(d,\Gamma_\Lambda^{\cl})>0\), then
\[
    D_{\Sector_\Lambda(d)}(d)
    \ge
    \frac{\mu_\Lambda^{\rm comp}}{6}
    \Dist_{\cl}(d,\Gamma_\Lambda^{\cl}).
\]
\end{proposition}

\begin{proof}
By \Cref{lem:observation-channel-resolution},
\[
\begin{aligned}
    \calM_\Lambda^{\rm comp}(d)
    &=
    \|O_\Lambda^{\cl}d\|
    +
    C_R\Rep_\Lambda^{\cl}(d)
    +
    C_T\Tax_\Lambda^{\cl}(d)\\
    &\le
    \sum_{\sigma\in\mathfrak S_\Lambda}D_\sigma(d).
\end{aligned}
\]
By \Cref{thm:clean-computational-antiphantom},
\[
    \mu_\Lambda^{\rm comp}
    \Dist_{\cl}(d,\Gamma_\Lambda^{\cl})
    \le
    \calM_\Lambda^{\rm comp}(d).
\]
Combining the two inequalities gives
\[
    \mu_\Lambda^{\rm comp}
    \Dist_{\cl}(d,\Gamma_\Lambda^{\cl})
    \le
    \sum_{\sigma\in\mathfrak S_\Lambda}D_\sigma(d).
\]
Since \(\mathfrak S_\Lambda\) has six elements, at least one sector satisfies
\[
    D_\sigma(d)
    \ge
    \frac{\mu_\Lambda^{\rm comp}}{6}
    \Dist_{\cl}(d,\Gamma_\Lambda^{\cl}).
\]
The tie-breaking rule in \Cref{def:sector-detector-amplitudes} selects one
such maximizing sector, and the claim follows.
\end{proof}

\begin{corollary}[Sector cover of normalized quotient classes]
\label{cor:sector-cover}
Under the hypotheses of \Cref{prop:finite-state-detector-alternative}, the
quotient unit sphere is covered by the six sector sets
\[
\begin{aligned}
    \calS_\sigma
    :=
    \{[d]\in S_\Lambda^{\cl}:&
    \text{ there is a representative }d
    \text{ with }
    \Dist_{\cl}(d,\Gamma_\Lambda^{\cl})=1\\
    &\text{ and }
    D_\sigma(d)\ge \mu_\Lambda^{\rm comp}/6\},
    \qquad
    \sigma\in\mathfrak S_\Lambda .
\end{aligned}
 \]
\end{corollary}

\begin{proof}
Let \([d]\in S_\Lambda^{\cl}\), so
\(\Dist_{\cl}(d,\Gamma_\Lambda^{\cl})=1\).  Applying
\Cref{prop:finite-state-detector-alternative} gives a sector
\(\sigma=\Sector_\Lambda(d)\) such that
\[
    D_\sigma(d)\ge\frac{\mu_\Lambda^{\rm comp}}{6}.
\]
Thus \([d]\in\calS_\sigma\), proving the cover.
\end{proof}

\begin{remark}[Representative dependence]
The sector assignment is made for a chosen clean representative \(d\), not for
an abstract quotient class alone.  This is intentional: the present theorem is
a finite-window detector statement before any canonical gauge slice has been
selected.  A later gauge-slice theorem may turn the sector map into a
well-defined map on quotient classes, but that is not claimed here.
\end{remark}

\begin{remark}[No computational overinterpretation]
The word \emph{sector} means only that, in a fixed finite-dimensional window,
one of finitely many detector channels must account for a definite fraction of
the clean quotient gap.  The result does not encode a Turing machine, does not
produce a symbolic dynamics, and does not imply any undecidability statement.
\end{remark}

\section{Conditional Localized Transfer Corollary}
\label{sec:localized-transfer-corollary}

The clean gap can be inserted into a localized transfer theorem only after the
localized defect package has been compared with the clean quotient.  This
section records the algebraic consequence of such a comparison.  It does not
prove quotient lifting, residual absorption, or any localized Navier--Stokes
estimate.

\subsection{Localized transfer data}

\begin{definition}[Abstract localized detector package]
\label{def:abstract-localized-detector-package}
Let \(\calK_\Lambda^{\loc}\) be a finite-dimensional localized defect space
with a localized gauge subspace \(\Gamma_\Lambda^{\loc}\).  Define
\[
    \Dist_{\loc}(\mathfrak D,\Gamma_\Lambda^{\loc})
    :=
    \inf_{\gamma\in\Gamma_\Lambda^{\loc}}
    \|\mathfrak D-\gamma\|_{\loc}.
\]
Let
\[
    \Theta_\Lambda:\calK_\Lambda^{\loc}\to\calK_\Lambda^{\cl}
\]
be a local-to-clean chart, and let
\[
    \calM_\Lambda^{\loc}:\calK_\Lambda^{\loc}\to[0,\infty)
\]
be the localized detector size.  The quantity \(\calM_\Lambda^{\loc}\) is
intended to collect the localized pressure, flux, energy, trace, reproduction,
and tax observations, but no concrete PDE formula is assumed in this
definition.
\end{definition}

\begin{assumption}[Local-to-clean transfer comparison]
\label{ass:local-to-clean-transfer-comparison}
There are constants
\[
    0\le\varepsilon_G<1,
    \qquad
    \delta_G\ge0,
\]
and a nonnegative residual functional
\[
    \Err_\Lambda^{\loc}:\calK_\Lambda^{\loc}\to[0,\infty)
\]
such that every \(\mathfrak D\in\calK_\Lambda^{\loc}\) satisfies
\[
    \Dist_{\cl}(\Theta_\Lambda\mathfrak D,\Gamma_\Lambda^{\cl})
    \ge
    (1-\varepsilon_G)
    \Dist_{\loc}(\mathfrak D,\Gamma_\Lambda^{\loc})
    -
    \delta_G
\]
and
\[
    \calM_\Lambda^{\loc}(\mathfrak D)
    +
    \Err_\Lambda^{\loc}(\mathfrak D)
    \ge
    \calM_\Lambda^{\rm comp}(\Theta_\Lambda\mathfrak D).
\]
\end{assumption}

\begin{assumption}[Localized residual budget]
\label{ass:localized-residual-budget}
There are constants \(\eta_\Lambda\ge0\) and
\(\Delta_\Lambda\ge0\) such that every
\(\mathfrak D\in\calK_\Lambda^{\loc}\) satisfies
\[
    \Err_\Lambda^{\loc}(\mathfrak D)
    \le
    \eta_\Lambda
    \Dist_{\loc}(\mathfrak D,\Gamma_\Lambda^{\loc})
    +
    \Delta_\Lambda.
\]
\end{assumption}

\subsection{Transfer consequence}

\begin{corollary}[Conditional localized transfer from the computational gap]
\label{cor:localized-transfer-from-computational-gap}
Assume \(\calQ_\Lambda^{\cl}\ne\{0\}\),
\Cref{ass:clean-detector-gauge-compatibility},
\Cref{ass:kernel-free-computational-detector},
\Cref{ass:local-to-clean-transfer-comparison}, and
\Cref{ass:localized-residual-budget}.  Then every
\(\mathfrak D\in\calK_\Lambda^{\loc}\) satisfies
\[
\begin{aligned}
    \calM_\Lambda^{\loc}(\mathfrak D)
    \ge&
    \bigl(
        \mu_\Lambda^{\rm comp}(1-\varepsilon_G)
        -
        \eta_\Lambda
    \bigr)
    \Dist_{\loc}(\mathfrak D,\Gamma_\Lambda^{\loc})\\
    &-
    \mu_\Lambda^{\rm comp}\delta_G
    -
    \Delta_\Lambda.
\end{aligned}
\]
In particular, if
\[
    \eta_\Lambda
    <
    \mu_\Lambda^{\rm comp}(1-\varepsilon_G),
\]
then the localized detector controls the localized quotient distance up to the
explicit additive loss
\(\mu_\Lambda^{\rm comp}\delta_G+\Delta_\Lambda\).
\end{corollary}

\begin{proof}
By \Cref{thm:clean-computational-antiphantom} applied to
\(\Theta_\Lambda\mathfrak D\),
\[
    \calM_\Lambda^{\rm comp}(\Theta_\Lambda\mathfrak D)
    \ge
    \mu_\Lambda^{\rm comp}
    \Dist_{\cl}(\Theta_\Lambda\mathfrak D,\Gamma_\Lambda^{\cl}).
\]
Using the quotient-distance comparison in
\Cref{ass:local-to-clean-transfer-comparison} gives
\[
    \calM_\Lambda^{\rm comp}(\Theta_\Lambda\mathfrak D)
    \ge
    \mu_\Lambda^{\rm comp}
    \bigl(
        (1-\varepsilon_G)
        \Dist_{\loc}(\mathfrak D,\Gamma_\Lambda^{\loc})
        -
        \delta_G
    \bigr).
\]
The detector comparison in
\Cref{ass:local-to-clean-transfer-comparison} therefore implies
\[
    \calM_\Lambda^{\loc}(\mathfrak D)
    \ge
    \mu_\Lambda^{\rm comp}(1-\varepsilon_G)
    \Dist_{\loc}(\mathfrak D,\Gamma_\Lambda^{\loc})
    -
    \mu_\Lambda^{\rm comp}\delta_G
    -
    \Err_\Lambda^{\loc}(\mathfrak D).
\]
Finally apply \Cref{ass:localized-residual-budget} to bound
\(\Err_\Lambda^{\loc}(\mathfrak D)\) from above.  This gives the displayed
inequality.  The final statement follows by positivity of the coefficient of
the localized quotient distance.
\end{proof}

\begin{remark}[Status of the localized corollary]
\Cref{cor:localized-transfer-from-computational-gap} is only a conditional
algebraic transfer statement.  It uses
\(c_\Lambda^{\cl}=\mu_\Lambda^{\rm comp}\), but it does not prove the
local-to-clean chart estimate, the quotient-distance comparison, the residual
budget, or the threshold inequality
\(\eta_\Lambda<\mu_\Lambda^{\rm comp}(1-\varepsilon_G)\).
\end{remark}

\section{Failure Modes}
\label{sec:failure-modes}

The preceding sections separate the clean finite-window obstruction from the
localized transfer obstruction.  This makes the possible failures explicit.

\begin{definition}[Fixed-window computational phantom]
\label{def:fixed-window-computational-phantom}
A fixed-window computational phantom is an element
\[
    d\in\calK_\Lambda^{\cl}
\]
such that
\[
    \Dist_{\cl}(d,\Gamma_\Lambda^{\cl})>0,
    \qquad
    O_\Lambda^{\cl}d=0,
    \qquad
    \Rep_\Lambda^{\cl}(d)=0,
    \qquad
    \Tax_\Lambda^{\cl}(d)=0.
\]
Equivalently, it is a non-gauge clean defect that is invisible, exactly
reproducible, and tax-free in the fixed finite window.
\end{definition}

\begin{definition}[Moving-window computational gap collapse]
\label{def:moving-window-gap-collapse}
A moving-window computational gap collapse is a sequence of finite windows
\(\Lambda_n\) and clean detector packages satisfying the fixed-window clean
detector gauge compatibility and kernel-free condition for each \(n\), but for
which
\[
    \mu_{\Lambda_n}^{\rm comp}\to0.
\]
This is not a contradiction to the fixed-window theorem.  It means only that
the positive constants obtained by compactness are not uniform along the
chosen sequence of windows.
\end{definition}

\begin{definition}[Asymptotic invisible reproducible tax-free sequence]
\label{def:asymptotic-invisible-sequence}
An asymptotic invisible reproducible tax-free sequence is a sequence
\((\Lambda_n,d_n)\) such that
\[
    \Dist_{\cl,n}(d_n,\Gamma_{\Lambda_n}^{\cl})=1
\]
and
\[
    \|O_{\Lambda_n}^{\cl}d_n\|
    +
    C_R\Rep_{\Lambda_n}^{\cl}(d_n)
    +
    C_T\Tax_{\Lambda_n}^{\cl}(d_n)
    \to0.
\]
Such a sequence is an asymptotic version of a clean phantom.  It is compatible
with every fixed-window theorem unless a window-uniform lower bound is proved.
\end{definition}

\begin{definition}[Localized transfer-threshold failure]
\label{def:localized-transfer-threshold-failure}
For a localized package satisfying the comparison hypotheses of
\Cref{ass:local-to-clean-transfer-comparison}, transfer-threshold failure
means that at least one of the following prevents
\Cref{cor:localized-transfer-from-computational-gap} from giving a useful
localized lower bound:
\[
    \eta_\Lambda
    \ge
    \mu_\Lambda^{\rm comp}(1-\varepsilon_G),
\]
or the additive loss
\[
    \mu_\Lambda^{\rm comp}\delta_G+\Delta_\Lambda
\]
is not controlled in the regime under consideration.
\end{definition}

\begin{proposition}[Exhaustion of finite-window failure modes]
\label{prop:failure-mode-exhaustion}
Fix a nontrivial clean quotient.  Suppose a localized lower bound of the form
\[
    \calM_\Lambda^{\loc}(\mathfrak D)
    \ge
    c_{\loc,\Lambda}
    \Dist_{\loc}(\mathfrak D,\Gamma_\Lambda^{\loc})
    -
    \Delta_{\loc,\Lambda}
\]
is not obtained from the clean computational anti-phantom mechanism for the
window \(\Lambda\).  Then at least one of the following is responsible:
\begin{enumerate}[label=(\roman*),leftmargin=2em]
    \item the clean detector gauge compatibility in
    \Cref{ass:clean-detector-gauge-compatibility} is unavailable, so the
    detector has not descended to the clean quotient;
    \item the fixed-window kernel-free condition fails, equivalently a
    fixed-window computational phantom exists;
    \item the clean gap is positive but too small for the desired window
    family, indicating moving-window gap collapse as a possible obstruction;
    \item the local-to-clean quotient or detector comparison in
    \Cref{ass:local-to-clean-transfer-comparison} is unavailable;
    \item the residual budget in \Cref{ass:localized-residual-budget} is
    unavailable;
    \item the transfer threshold or additive-loss control in
    \Cref{def:localized-transfer-threshold-failure} fails.
\end{enumerate}
\end{proposition}

\begin{proof}
If \Cref{ass:clean-detector-gauge-compatibility} is unavailable, the clean
detector has not been shown to be a quotient detector, so the first alternative
is responsible.  If gauge compatibility holds but the fixed-window kernel-free
condition fails, then \Cref{def:fixed-window-computational-phantom} gives the
second alternative.  If both hold, then \Cref{thm:kernel-free-characterization}
gives \(\mu_\Lambda^{\rm comp}>0\) for the fixed window.  To pass from this clean
gap to a localized lower bound, the argument uses exactly the quotient and
detector comparison in \Cref{ass:local-to-clean-transfer-comparison}, the
residual budget in \Cref{ass:localized-residual-budget}, and the threshold
condition in \Cref{cor:localized-transfer-from-computational-gap}.  If all
these inputs hold with useful constants, the displayed localized lower bound
follows from that corollary.  Therefore, if the lower bound is not obtained,
one of the listed inputs or threshold requirements must fail.  Along a family
of windows, the additional possibility is that the fixed-window constants are
positive but degenerate, which is precisely the moving-window gap-collapse
alternative.
\end{proof}

\begin{remark}[Role of the taxonomy]
\Cref{prop:failure-mode-exhaustion} is a bookkeeping result.  It does not say
which failure mode actually occurs for Navier--Stokes.  Its purpose is to keep
future work honest: after the clean theorem is proved, the remaining
mathematics is either uniform clean-gap control, local-to-clean comparison,
residual absorption, or additive-loss control.
\end{remark}

\section{Pressure-Source Quotient Observability}
\label{sec:pressure-source-observability}

We now isolate the first PDE-facing compatibility input.  The purpose of this
section is not to prove a pressure estimate.  It is to state precisely what
must be controlled if the localized quotient distance is to see the active
pressure-source mismatch.  The model is deliberately aligned with the standard
Navier--Stokes pressure identity and with the pressure-splitting techniques
used in local regularity theory
\cite{SohrWahl1986,SereginSverak2002,SereginLectureNotes}.

\subsection{Pressure-source coordinate}

\begin{definition}[Localized pressure-source datum]
\label{def:localized-pressure-source-datum}
A localized pressure-source datum consists, for each \(k\in\Lambda\), of
finite-dimensional normed spaces
\[
    \calX_{{\rm src},k},
    \qquad
    \calZ_{{\rm src},k},
    \qquad
    \calU_k,
    \qquad
    \calP_k^{\rm act},
    \qquad
    \calR_k,
\]
coordinate maps
\[
    U_k:\calK_\Lambda^{\loc}\to\calU_k,
    \qquad
    P_k^{\rm act}:\calK_\Lambda^{\loc}\to\calP_k^{\rm act},
    \qquad
    R_k:\calK_\Lambda^{\loc}\to\calR_k,
\]
an assembly map
\[
    \mathscr A_{{\rm src},k}:
    \calP_k^{\rm act}\times\calU_k\times\calR_k
    \to
    \calZ_{{\rm src},k},
\]
and a fixed finite-dimensional source projection
\[
    \Pi_{{\rm src},k}:\calZ_{{\rm src},k}
    \to \calX_{{\rm src},k}.
\]
The associated pressure-source mismatch is
\[
    S_k^{\rm prs}(\mathfrak D)
    :=
    \Pi_{{\rm src},k}
    \mathscr A_{{\rm src},k}
    \left(
        P_k^{\rm act}(\mathfrak D),
        U_k(\mathfrak D),
        R_k(\mathfrak D)
    \right),
\]
where \(\mathscr A_{{\rm src},k}\) represents the finite-dimensional form of
\[
        -\Delta P_k^{\rm act}(\mathfrak D)
        -
        \partial_i\partial_j
        \bigl(
            U_{k,i}(\mathfrak D)U_{k,j}(\mathfrak D)
            +
            R_{k,ij}(\mathfrak D)
        \bigr)
    .
\]
\end{definition}

\begin{remark}[Status of the source coordinate]
\Cref{def:localized-pressure-source-datum} is a finite-window model datum.
The notation with \(-\Delta\) and \(\partial_i\partial_j\) records the intended
Navier--Stokes pressure identity, but the definition does not prove that the
coordinates are produced by a suitable weak solution or that cutoff,
projection, harmonic, or truncation errors are small.
\end{remark}

\begin{definition}[Pressure-source residual]
\label{def:pressure-source-residual}
Fix weights \(w_k>0\).  Define
\[
    \Err_{\rm src}^{\rm prs}(\mathfrak D)
    :=
    \left(
        \sum_{k\in\Lambda}
        w_k
        \|S_k^{\rm prs}(\mathfrak D)\|_{\calX_{{\rm src},k}}^2
    \right)^{1/2}.
\]
The quotient-compatible pressure-source residual is
\[
    \Err_{\rm src,q}^{\rm prs}(\mathfrak D)
    :=
    \inf_{\gamma\in\Gamma_\Lambda^{\loc}}
    \Err_{\rm src}^{\rm prs}(\mathfrak D-\gamma).
\]
\end{definition}

\begin{remark}[Gauge convention]
The quotient residual is used because this manuscript has not proved that
\(S_k^{\rm prs}\) is invariant under every localized gauge direction.  If a
later gauge-cleaning theorem proves such invariance, the raw residual
\(\Err_{\rm src}^{\rm prs}\) may replace \(\Err_{\rm src,q}^{\rm prs}\).
\end{remark}

\subsection{Structural observability model}

\begin{assumption}[Pressure-source quotient observability]
\label{ass:pressure-source-quotient-observability}
There exist constants
\[
    C_{\rm src}<\infty,
    \qquad
    \Delta_{\rm src}\ge0,
\]
such that every localized package \(\mathfrak D\in\calK_\Lambda^{\loc}\)
satisfies
\[
    \Err_{\rm src,q}^{\rm prs}(\mathfrak D)
    \le
    C_{\rm src}
    \Dist_{\loc}(\mathfrak D,\Gamma_\Lambda^{\loc})
    +
    \Delta_{\rm src}.
\]
\end{assumption}

\begin{proposition}[Conditional pressure-source quotient observability]
\label{prop:conditional-pressure-source-observability}
Under \Cref{ass:pressure-source-quotient-observability}, the active
pressure-source residual is controlled by the localized quotient distance up
to the additive error \(\Delta_{\rm src}\):
\[
    \Err_{\rm src,q}^{\rm prs}(\mathfrak D)
    \le
    C_{\rm src}
    \Dist_{\loc}(\mathfrak D,\Gamma_\Lambda^{\loc})
    +
    \Delta_{\rm src}
\]
for every \(\mathfrak D\in\calK_\Lambda^{\loc}\).
\end{proposition}

\begin{proof}
This is exactly \Cref{ass:pressure-source-quotient-observability}.  No
additional pressure estimate is being used.
\end{proof}

\begin{remark}[Effect on the transfer theorem]
This structural-observability model does not change the form of
\Cref{cor:localized-transfer-from-computational-gap}; it keeps the original
localized quotient distance \(\Dist_{\loc}\).  Its cost is that
\Cref{ass:pressure-source-quotient-observability} becomes an explicit
component hypothesis that must be proved or assumed before pressure-source
terms can be absorbed into the localized residual budget.
\end{remark}

\subsection{What remains a PDE estimate}

To turn \Cref{ass:pressure-source-quotient-observability} into a theorem, one
must control the following components in the source norm
\(\calX_{{\rm src},k}\):
\begin{enumerate}[label=(\roman*),leftmargin=2em]
    \item the active pressure Poisson mismatch generated by the localized
    pressure splitting;
    \item cutoff commutator leakage from localizing
    \(-\Delta p=\partial_i\partial_j(u_i u_j)\);
    \item harmonic pressure gauge terms not removed by
    \(\Gamma_\Lambda^{\loc}\);
    \item finite projection and truncation errors in
    \(\Pi_{{\rm src},k}\);
    \item stress or covariance mismatch between the projected quadratic
    velocity term and the \(R_k\)-coordinate.
\end{enumerate}
None of these estimates is proved in this section.

\begin{proposition}[Pressure-source observability decision point]
\label{prop:pressure-source-decision-point}
The present manuscript establishes only the following conditional implication:
if \Cref{ass:pressure-source-quotient-observability} is supplied, then
\Cref{prop:conditional-pressure-source-observability} follows.  Removing the
assumption requires either:
\begin{enumerate}[label=(\roman*),leftmargin=2em]
    \item an enhanced quotient distance that includes
    \(\Err_{\rm src,q}^{\rm prs}\);
    \item a proof of \Cref{ass:pressure-source-quotient-observability} from
    localized Navier--Stokes pressure geometry; or
    \item a sharpened localized residual norm that includes
    \(\Err_{\rm src,q}^{\rm prs}\) and a residual-budget estimate for that
    sharpened norm.
\end{enumerate}
\end{proposition}

\begin{proof}
The three alternatives are precisely the three ways to make the source
residual appear in the existing transfer architecture.  The first changes the
quotient geometry, the second keeps the quotient geometry and proves the
structural hypothesis, and the third changes the residual budget.  Without at
least one of these inputs, the displayed pressure-source estimate is not a
consequence of the clean finite-window gap or of the abstract localized
transfer corollary.
\end{proof}

\subsection{Enhanced localized quotient distance with pressure-source observability}
\label{subsec:enhanced-localized-quotient-pressure-source}

We now record the enhanced-distance route in a concrete normalized local
model.  This subsection changes the localized quotient geometry by placing the
pressure-source residual directly into the defect distance.

\begin{definition}[Normalized pressure-source geometry]
\label{def:normalized-pressure-source-geometry}
Let
\[
    Q_1:=B_1\times(-1,0),
    \qquad
    I:=(-1,0),
\]
with
\[
    B_{1/2}\subset B_{3/4}\subset B_1.
\]
Fix a cutoff \(\eta\in C_c^\infty(B_1)\) such that
\[
    \eta\equiv1\quad\text{on }B_{3/4}.
\]
Set
\[
    X_{\rm src}
    :=
    L^{3/2}\bigl(I;L^{3/2}(B_1)\bigr)^{3\times3},
    \qquad
    Y_{\rm prs}
    :=
    L^{3/2}\bigl(I;L^{3/2}(B_{1/2})\bigr).
\]
For a localized velocity \(u\), define
\[
    f_{ij}:=u_i u_j.
\]
For a clean finite-window package, define the clean source tensor
\[
    F_{ij}^{\cl}:=U_iU_j+R_{ij}.
\]
The active localized pressure and clean active pressure are
\[
    p^{\rm act}:=R_iR_j(\eta f_{ij}),
    \qquad
    p^{\cl}:=P_{\rm prs}^{\cl}R_iR_j(F_{ij}^{\cl}),
\]
where \(P_{\rm prs}^{\cl}:Y_{\rm prs}\to Y_{\rm prs}\) is a fixed bounded
clean pressure projection.  All pressures in this subsection are restricted to
\(B_{1/2}\) when measured in \(Y_{\rm prs}\).  All source tensors are
extended by zero outside \(B_1\) before applying the whole-space Riesz
transforms.
\end{definition}

\begin{definition}[Concrete pressure-source mismatch]
\label{def:concrete-pressure-source-mismatch}
Let \(h^{\rm harm}\in Y_{\rm prs}\) denote the selected harmonic-gauge
residual, with \(h^{\rm harm}=0\) if no harmonic gauge residual is present.
Define the concrete pressure-source mismatch by
\[
    \mathfrak P_{\rm src}
    :=
    p^{\rm act}-p^{\cl}-h^{\rm harm}.
\]
Define
\[
    \Err_{\rm src}^{\rm prs}
    :=
    \|\mathfrak P_{\rm src}\|_{Y_{\rm prs}}.
\]
The quotient-compatible version is
\[
    \Err_{\rm src,q}^{\rm prs}(\mathfrak D)
    :=
    \inf_{\gamma\in\Gamma_\Lambda^{\loc}}
    \Err_{\rm src}^{\rm prs}(\mathfrak D-\gamma),
\]
which agrees with \Cref{def:pressure-source-residual} in this concrete
one-window pressure model.
\end{definition}

\begin{lemma}[Pressure-source mismatch decomposition]
\label{lem:pressure-source-mismatch-decomposition}
In the normalized model of
\Cref{def:normalized-pressure-source-geometry}, the pressure-source mismatch
admits the decomposition
\[
    \mathfrak P_{\rm src}
    =
    C_\eta(f)
    +
    E_{\rm act}^{\rm src}
    +
    E_{\rm proj}^{\cl}
    +
    E_{\rm harm},
\]
where
\[
    C_\eta(f)
    :=
    R_iR_j(\eta f_{ij})-\eta R_iR_j(f_{ij}),
\]
\[
    E_{\rm act}^{\rm src}
    :=
    \eta R_iR_j(f_{ij})-R_iR_j(F_{ij}^{\cl}),
\]
\[
    E_{\rm proj}^{\cl}
    :=
    (I-P_{\rm prs}^{\cl})R_iR_j(F_{ij}^{\cl}),
    \qquad
    E_{\rm harm}:=-h^{\rm harm}.
\]
Consequently,
\[
    \Err_{\rm src}^{\rm prs}
    \le
    \|C_\eta(f)\|_{Y_{\rm prs}}
    +
    \|E_{\rm act}^{\rm src}\|_{Y_{\rm prs}}
    +
    \|E_{\rm proj}^{\cl}\|_{Y_{\rm prs}}
    +
    \|E_{\rm harm}\|_{Y_{\rm prs}}.
\]
\end{lemma}

\begin{proof}
Using the definitions of \(p^{\rm act}\) and \(p^{\cl}\),
\[
\begin{aligned}
    \mathfrak P_{\rm src}
    &=
    R_iR_j(\eta f_{ij})
    -
    P_{\rm prs}^{\cl}R_iR_j(F_{ij}^{\cl})
    -
    h^{\rm harm}\\
    &=
    \bigl(R_iR_j(\eta f_{ij})-\eta R_iR_j(f_{ij})\bigr)\\
    &\quad+
    \bigl(\eta R_iR_j(f_{ij})-R_iR_j(F_{ij}^{\cl})\bigr)\\
    &\quad+
    \bigl(R_iR_j(F_{ij}^{\cl})
      -P_{\rm prs}^{\cl}R_iR_j(F_{ij}^{\cl})\bigr)
    -
    h^{\rm harm}.
\end{aligned}
\]
This is the stated decomposition.  The norm bound follows from the triangle
inequality in \(Y_{\rm prs}\).
\end{proof}

\begin{lemma}[Fixed-geometry commutator bound]
\label{lem:fixed-geometry-commutator-bound}
There is a constant \(C_\eta<\infty\), depending only on the fixed cutoff and
the normalized balls, such that
\[
    \|C_\eta(f)\|_{Y_{\rm prs}}
    \le
    C_\eta
    \|(1-\eta)f\|_{L^{3/2}(I;L^{3/2}(B_1\setminus B_{3/4}))^{3\times3}}.
\]
This is a fixed-geometry estimate and contains no scale-uniform assertion.
\end{lemma}

\begin{proof}
On \(B_{1/2}\), \(\eta=1\).  Hence
\[
    C_\eta(f)
    =
    R_iR_j(\eta f_{ij})-R_iR_j(f_{ij})
    =
    -R_iR_j((1-\eta)f_{ij}).
\]
The source \((1-\eta)f\) is supported in
\[
    B_1\setminus B_{3/4},
\]
which is separated from \(B_{1/2}\).  Therefore the kernel of
\(R_iR_j\) is smooth and uniformly bounded as a map from the source annulus to
\(B_{1/2}\).  For each fixed time,
\[
    \|C_\eta(f)(t)\|_{L^{3/2}(B_{1/2})}
    \le
    C_\eta
    \|(1-\eta)f(t)\|_{L^{3/2}(B_1\setminus B_{3/4})^{3\times3}}.
\]
Taking the \(L^{3/2}\)-norm in time gives the claim.
\end{proof}

\begin{remark}[Fixed estimates versus structural terms]
\Cref{lem:fixed-geometry-commutator-bound} is the only estimate in this
subsection that comes from the fixed separated geometry.  The active source
residual \(E_{\rm act}^{\rm src}\), the clean projection residual
\(E_{\rm proj}^{\cl}\), and the harmonic gauge residual \(E_{\rm harm}\) are
not shown to be small or controlled by the original localized quotient
distance.
\end{remark}

\begin{definition}[Enhanced localized quotient distance]
\label{def:enhanced-localized-quotient-distance}
Fix \(\alpha_{\rm src}>0\).  Define
\[
    \Dist_{\loc}^{\sharp}
    (\mathfrak D,\Gamma_\Lambda^{\loc})
    :=
    \Dist_{\loc}(\mathfrak D,\Gamma_\Lambda^{\loc})
    +
    \alpha_{\rm src}
    \Err_{\rm src,q}^{\rm prs}(\mathfrak D).
\]
\end{definition}

\begin{lemma}[Enhanced-distance pressure-source observability]
\label{lem:enhanced-distance-pressure-source-observability}
For every localized package \(\mathfrak D\),
\[
    \Err_{\rm src,q}^{\rm prs}(\mathfrak D)
    \le
    \alpha_{\rm src}^{-1}
    \Dist_{\loc}^{\sharp}
    (\mathfrak D,\Gamma_\Lambda^{\loc}).
\]
\end{lemma}

\begin{proof}
By \Cref{def:enhanced-localized-quotient-distance},
\[
    \Dist_{\loc}^{\sharp}
    (\mathfrak D,\Gamma_\Lambda^{\loc})
    \ge
    \alpha_{\rm src}
    \Err_{\rm src,q}^{\rm prs}(\mathfrak D).
\]
Dividing by \(\alpha_{\rm src}>0\) gives the result.
\end{proof}

\begin{assumption}[Enhanced local-to-clean transfer comparison]
\label{ass:enhanced-local-to-clean-transfer-comparison}
There are constants \(0\le\varepsilon_G<1\) and \(\delta_G\ge0\) such that
every \(\mathfrak D\in\calK_\Lambda^{\loc}\) satisfies
\[
    \Dist_{\cl}(\Theta_\Lambda\mathfrak D,\Gamma_\Lambda^{\cl})
    \ge
    (1-\varepsilon_G)
    \Dist_{\loc}^{\sharp}(\mathfrak D,\Gamma_\Lambda^{\loc})
    -
    \delta_G
\]
and
\[
    \calM_\Lambda^{\loc}(\mathfrak D)
    +
    \Err_\Lambda^{\loc}(\mathfrak D)
    \ge
    \calM_\Lambda^{\rm comp}(\Theta_\Lambda\mathfrak D).
\]
\end{assumption}

\begin{assumption}[Enhanced localized residual budget]
\label{ass:enhanced-localized-residual-budget}
There are constants \(\eta_\Lambda\ge0\) and \(\Delta_\Lambda\ge0\) such that
every \(\mathfrak D\in\calK_\Lambda^{\loc}\) satisfies
\[
    \Err_\Lambda^{\loc}(\mathfrak D)
    \le
    \eta_\Lambda
    \Dist_{\loc}^{\sharp}(\mathfrak D,\Gamma_\Lambda^{\loc})
    +
    \Delta_\Lambda.
\]
\end{assumption}

\begin{theorem}[Conditional enhanced-distance localized transfer]
\label{thm:enhanced-distance-localized-transfer}
Assume \(\calQ_\Lambda^{\cl}\ne\{0\}\),
\Cref{ass:clean-detector-gauge-compatibility},
\Cref{ass:kernel-free-computational-detector},
\Cref{ass:enhanced-local-to-clean-transfer-comparison}, and
\Cref{ass:enhanced-localized-residual-budget}.  Then every
\(\mathfrak D\in\calK_\Lambda^{\loc}\) satisfies
\[
\begin{aligned}
    \calM_\Lambda^{\loc}(\mathfrak D)
    \ge&
    \bigl(
        \mu_\Lambda^{\rm comp}(1-\varepsilon_G)
        -
        \eta_\Lambda
    \bigr)
    \Dist_{\loc}^{\sharp}
    (\mathfrak D,\Gamma_\Lambda^{\loc})\\
    &-
    \mu_\Lambda^{\rm comp}\delta_G
    -
    \Delta_\Lambda.
\end{aligned}
\]
\end{theorem}

\begin{proof}
The proof is the same algebraic transfer argument as in
\Cref{cor:localized-transfer-from-computational-gap}, with
\(\Dist_{\loc}\) replaced by \(\Dist_{\loc}^{\sharp}\).  Applying
\Cref{thm:clean-computational-antiphantom} to
\(\Theta_\Lambda\mathfrak D\) gives
\[
    \calM_\Lambda^{\rm comp}(\Theta_\Lambda\mathfrak D)
    \ge
    \mu_\Lambda^{\rm comp}
    \Dist_{\cl}(\Theta_\Lambda\mathfrak D,\Gamma_\Lambda^{\cl}).
\]
The enhanced quotient comparison in
\Cref{ass:enhanced-local-to-clean-transfer-comparison} implies
\[
    \calM_\Lambda^{\rm comp}(\Theta_\Lambda\mathfrak D)
    \ge
    \mu_\Lambda^{\rm comp}
    \bigl(
        (1-\varepsilon_G)
        \Dist_{\loc}^{\sharp}(\mathfrak D,\Gamma_\Lambda^{\loc})
        -
        \delta_G
    \bigr).
\]
Using the detector comparison from the same assumption and then subtracting
the residual term gives
\[
    \calM_\Lambda^{\loc}(\mathfrak D)
    \ge
    \mu_\Lambda^{\rm comp}(1-\varepsilon_G)
    \Dist_{\loc}^{\sharp}(\mathfrak D,\Gamma_\Lambda^{\loc})
    -
    \mu_\Lambda^{\rm comp}\delta_G
    -
    \Err_\Lambda^{\loc}(\mathfrak D).
\]
Finally apply \Cref{ass:enhanced-localized-residual-budget}.  This proves the
displayed estimate.
\end{proof}

\begin{remark}[Meaning of enhanced pressure-source observability]
\Cref{lem:enhanced-distance-pressure-source-observability} proves
pressure-source observability only relative to the enhanced quotient geometry.
It does not show that the original localized quotient distance
\(\Dist_{\loc}\) controls \(\Err_{\rm src,q}^{\rm prs}\).  It also makes no
scale-uniform claim and gives no Navier--Stokes regularity conclusion.
\end{remark}

\begin{remark}[Tradeoff]
The advantage of \(\Dist_{\loc}^{\sharp}\) is that pressure-source
observability is built into the quotient geometry.  The cost is that the
localized quotient geometry has changed.  A later theorem must either compare
\(\Dist_{\loc}^{\sharp}\) with the original localized quotient distance or
justify the enhanced distance as the natural defect distance for localized
Navier--Stokes packages.
\end{remark}

\subsection{One-sided comparison with the original localized distance}
\label{subsec:enhanced-original-distance-comparison}

The next question is whether the enhanced distance is merely a convenient
renorming of the original localized quotient distance, or whether it has added
a genuinely new pressure-source coordinate.  The present subsection proves
only a sufficient condition for one-sided comparison.  It does not prove that
condition from Navier--Stokes pressure geometry.

\begin{definition}[Quotient pressure-source component budget]
\label{def:quotient-pressure-source-component-budget}
For \(\gamma\in\Gamma_\Lambda^{\loc}\), let
\[
    C_\eta^\gamma,\qquad
    E_{{\rm act},\gamma}^{\rm src},\qquad
    E_{{\rm proj},\gamma}^{\cl},\qquad
    E_{{\rm harm},\gamma}
\]
denote the four pressure-source components in
\Cref{lem:pressure-source-mismatch-decomposition} evaluated on
\(\mathfrak D-\gamma\).  Define the quotient component budget by
\[
\begin{aligned}
    B_{\rm src,q}^{\rm prs}(\mathfrak D)
    :=
    \inf_{\gamma\in\Gamma_\Lambda^{\loc}}
    \bigl(
        &\|C_\eta^\gamma\|_{Y_{\rm prs}}
        +
        \|E_{{\rm act},\gamma}^{\rm src}\|_{Y_{\rm prs}}\\
        &+
        \|E_{{\rm proj},\gamma}^{\cl}\|_{Y_{\rm prs}}
        +
        \|E_{{\rm harm},\gamma}\|_{Y_{\rm prs}}
    \bigr).
\end{aligned}
\]
\end{definition}

\begin{lemma}[Component budget controls the quotient pressure residual]
\label{lem:component-budget-controls-pressure-residual}
For every localized package \(\mathfrak D\),
\[
    \Err_{\rm src,q}^{\rm prs}(\mathfrak D)
    \le
    B_{\rm src,q}^{\rm prs}(\mathfrak D).
\]
\end{lemma}

\begin{proof}
Fix \(\gamma\in\Gamma_\Lambda^{\loc}\).  Applying
\Cref{lem:pressure-source-mismatch-decomposition} to
\(\mathfrak D-\gamma\) gives
\[
\begin{aligned}
    \Err_{\rm src}^{\rm prs}(\mathfrak D-\gamma)
    \le&
        \|C_\eta^\gamma\|_{Y_{\rm prs}}
        +
        \|E_{{\rm act},\gamma}^{\rm src}\|_{Y_{\rm prs}}\\
        &+
        \|E_{{\rm proj},\gamma}^{\cl}\|_{Y_{\rm prs}}
        +
        \|E_{{\rm harm},\gamma}\|_{Y_{\rm prs}}.
\end{aligned}
\]
Taking the infimum over \(\gamma\in\Gamma_\Lambda^{\loc}\) gives the
displayed inequality.
\end{proof}

\begin{assumption}[Original-distance pressure-source component budget]
\label{ass:original-distance-pressure-source-component-budget}
There are constants
\[
    C_{\rm cmp}<\infty,
    \qquad
    \Delta_{\rm cmp}\ge0,
\]
such that every localized package \(\mathfrak D\in\calK_\Lambda^{\loc}\)
satisfies
\[
    B_{\rm src,q}^{\rm prs}(\mathfrak D)
    \le
    C_{\rm cmp}
    \Dist_{\loc}(\mathfrak D,\Gamma_\Lambda^{\loc})
    +
    \Delta_{\rm cmp}.
\]
\end{assumption}

\begin{proposition}[Conditional one-sided comparison of localized distances]
\label{prop:conditional-enhanced-original-distance-comparison}
Assume \Cref{ass:original-distance-pressure-source-component-budget}.  Then
every localized package \(\mathfrak D\in\calK_\Lambda^{\loc}\) satisfies
\[
    \Dist_{\loc}(\mathfrak D,\Gamma_\Lambda^{\loc})
    \le
    \Dist_{\loc}^{\sharp}(\mathfrak D,\Gamma_\Lambda^{\loc})
\]
and
\[
    \Dist_{\loc}^{\sharp}(\mathfrak D,\Gamma_\Lambda^{\loc})
    \le
    \bigl(1+\alpha_{\rm src}C_{\rm cmp}\bigr)
    \Dist_{\loc}(\mathfrak D,\Gamma_\Lambda^{\loc})
    +
    \alpha_{\rm src}\Delta_{\rm cmp}.
\]
In particular, if \(\Delta_{\rm cmp}=0\), then the enhanced distance is
bounded above and below by fixed multiples of the original distance in this
finite-window model.
\end{proposition}

\begin{proof}
The lower bound follows directly from
\Cref{def:enhanced-localized-quotient-distance}, since
\(\Err_{\rm src,q}^{\rm prs}\ge0\).  For the upper bound, combine
\Cref{def:enhanced-localized-quotient-distance},
\Cref{lem:component-budget-controls-pressure-residual}, and
\Cref{ass:original-distance-pressure-source-component-budget}:
\[
\begin{aligned}
    \Dist_{\loc}^{\sharp}(\mathfrak D,\Gamma_\Lambda^{\loc})
    &=
    \Dist_{\loc}(\mathfrak D,\Gamma_\Lambda^{\loc})
    +
    \alpha_{\rm src}
    \Err_{\rm src,q}^{\rm prs}(\mathfrak D)\\
    &\le
    \Dist_{\loc}(\mathfrak D,\Gamma_\Lambda^{\loc})
    +
    \alpha_{\rm src}
    B_{\rm src,q}^{\rm prs}(\mathfrak D)\\
    &\le
    \bigl(1+\alpha_{\rm src}C_{\rm cmp}\bigr)
    \Dist_{\loc}(\mathfrak D,\Gamma_\Lambda^{\loc})
    +
    \alpha_{\rm src}\Delta_{\rm cmp}.
\end{aligned}
\]
If \(\Delta_{\rm cmp}=0\), this gives the stated two-sided multiplicative
comparison.
\end{proof}

\begin{remark}[Status of the comparison condition]
\Cref{prop:conditional-enhanced-original-distance-comparison} does not prove
that the original localized quotient distance controls the pressure-source
residual.  It identifies the exact componentwise pressure estimate needed for
that conclusion.  If \Cref{ass:original-distance-pressure-source-component-budget}
fails, then the pressure-source mismatch is not merely a bookkeeping term for
the original distance; it is evidence that the enhanced pressure-source
geometry may be a genuinely different localized defect geometry.
\end{remark}

\subsection{Concrete localized quotient norm: Model A}
\label{subsec:model-a-localized-quotient-norm}

We now make one explicit finite-window norm choice.  This is not presented as
the unique or canonical localized norm.  It is a model geometry whose purpose
is to make clear what changes when pressure-source observability is built
directly into the localized quotient distance.

\begin{definition}[Model A: weighted finite-dimensional localized coordinate norm]
\label{def:model-a-localized-coordinate-norm}
For each \(k\in\Lambda\), fix finite-dimensional normed spaces
\[
    V_{U,k},\quad
    V_{P,k},\quad
    V_{R,k},\quad
    V_{\Pi,k},\quad
    V_{\Phi,k},\quad
    V_{T,k},\quad
    V_{s,k},
\]
with norms
\[
    \|\cdot\|_{U,k},\quad
    \|\cdot\|_{P,k},\quad
    \|\cdot\|_{R,k},\quad
    \|\cdot\|_{\Pi,k},\quad
    \|\cdot\|_{\Phi,k},\quad
    \|\cdot\|_{T,k},\quad
    \|\cdot\|_{s,k}.
\]
A Model A localized package is a finite-window coordinate vector
\[
    D
    =
    (U_k,P_k^{\rm act},R_k,\Pi_k,\Phi_k,T_k,s_k)_{k\in\Lambda},
\]
where
\[
\begin{gathered}
    U_k\in V_{U,k},\qquad
    P_k^{\rm act}\in V_{P,k},\qquad
    R_k\in V_{R,k},\\
    \Pi_k\in V_{\Pi,k},\qquad
    \Phi_k\in V_{\Phi,k},\qquad
    T_k\in V_{T,k},\qquad
    s_k\in V_{s,k}.
\end{gathered}
\]
Here \(U_k\) is the localized velocity coordinate,
\(P_k^{\rm act}\) is the active pressure coordinate, \(R_k\) is the
Reynolds or covariance coordinate, \(\Pi_k\) is the flux coordinate,
\(\Phi_k\) is the energy or trace coordinate, \(T_k\) is the selected trace
coordinate, and \(s_k\) is the ledger slack coordinate.  Fix positive weights
\[
    w_U,w_P,w_R,w_\Pi,w_\Phi,w_T,w_s>0.
\]
Define
\[
\begin{aligned}
    \|D\|_{\loc,0}^2
    :=
    \sum_{k\in\Lambda}
    \bigl(
        &w_U\|U_k\|_{U,k}^2
        +
        w_P\|P_k^{\rm act}\|_{P,k}^2
        +
        w_R\|R_k\|_{R,k}^2\\
        &+
        w_\Pi\|\Pi_k\|_{\Pi,k}^2
        +
        w_\Phi\|\Phi_k\|_{\Phi,k}^2
        +
        w_T\|T_k\|_{T,k}^2
        +
        w_s\|s_k\|_{s,k}^2
    \bigr).
\end{aligned}
\]
Let \(\Gamma_\Lambda^{\loc}\) be a linear localized gauge subspace of this
finite-dimensional coordinate space.  The Model A baseline localized quotient
distance is
\[
    \Dist_{\loc,0}(D,\Gamma_\Lambda^{\loc})
    :=
    \inf_{\gamma\in\Gamma_\Lambda^{\loc}}
    \|D-\gamma\|_{\loc,0}.
\]
\end{definition}

\begin{definition}[Model A pressure-source quotient residual]
\label{def:model-a-pressure-source-quotient-residual}
Fix a nonnegative positively homogeneous pressure-source residual
\[
    \Err_{\rm src}^{\rm prs}:D\mapsto[0,\infty).
\]
In applications this is intended to be the pressure-source residual associated
with the concrete pressure mismatch of
\Cref{def:concrete-pressure-source-mismatch}.  In Model A, its positive
homogeneity is part of the finite-dimensional datum; it is not derived here
from the nonlinear Navier--Stokes source \(u_i u_j\).  Define
\[
    \Err_{\rm src,q}^{\rm prs}(D)
    :=
    \inf_{\gamma\in\Gamma_\Lambda^{\loc}}
    \Err_{\rm src}^{\rm prs}(D-\gamma).
\]
\end{definition}

\begin{definition}[Model A enhanced localized quotient distance]
\label{def:model-a-enhanced-localized-distance}
Fix \(\alpha_{\rm src}>0\).  Define
\[
    \Dist_{\loc,\alpha}^{\sharp}
    (D,\Gamma_\Lambda^{\loc})
    :=
    \Dist_{\loc,0}(D,\Gamma_\Lambda^{\loc})
    +
    \alpha_{\rm src}
    \Err_{\rm src,q}^{\rm prs}(D).
\]
\end{definition}

\begin{lemma}[Model A pressure-source observability]
\label{lem:model-a-pressure-source-observability}
For every Model A localized package \(D\),
\[
    \Err_{\rm src,q}^{\rm prs}(D)
    \le
    \alpha_{\rm src}^{-1}
    \Dist_{\loc,\alpha}^{\sharp}
    (D,\Gamma_\Lambda^{\loc}).
\]
\end{lemma}

\begin{proof}
By \Cref{def:model-a-enhanced-localized-distance},
\[
    \Dist_{\loc,\alpha}^{\sharp}
    (D,\Gamma_\Lambda^{\loc})
    \ge
    \alpha_{\rm src}
    \Err_{\rm src,q}^{\rm prs}(D).
\]
Divide by \(\alpha_{\rm src}>0\).
\end{proof}

\begin{lemma}[Homogeneity and gauge vanishing in Model A]
\label{lem:model-a-homogeneity-gauge-vanishing}
For every \(\lambda\ge0\) and every Model A localized package \(D\),
\[
    \Dist_{\loc,\alpha}^{\sharp}
    (\lambda D,\Gamma_\Lambda^{\loc})
    =
    \lambda
    \Dist_{\loc,\alpha}^{\sharp}
    (D,\Gamma_\Lambda^{\loc}).
\]
Moreover, if \(D\in\Gamma_\Lambda^{\loc}\), then
\[
    \Dist_{\loc,\alpha}^{\sharp}(D,\Gamma_\Lambda^{\loc})=0.
\]
\end{lemma}

\begin{proof}
Because \(\Gamma_\Lambda^{\loc}\) is a linear subspace and
\(\|\cdot\|_{\loc,0}\) is a norm, the baseline quotient distance is
positively homogeneous:
\[
    \Dist_{\loc,0}(\lambda D,\Gamma_\Lambda^{\loc})
    =
    \lambda
    \Dist_{\loc,0}(D,\Gamma_\Lambda^{\loc})
    \qquad(\lambda\ge0).
\]
The same argument applies to the quotient pressure-source residual.  For
\(\lambda>0\), use the change of variables
\(\gamma=\lambda\gamma'\) in the infimum and the positive homogeneity of
\(\Err_{\rm src}^{\rm prs}\):
\[
\begin{aligned}
    \Err_{\rm src,q}^{\rm prs}(\lambda D)
    &=
    \inf_{\gamma\in\Gamma_\Lambda^{\loc}}
    \Err_{\rm src}^{\rm prs}(\lambda D-\gamma)\\
    &=
    \lambda
    \inf_{\gamma'\in\Gamma_\Lambda^{\loc}}
    \Err_{\rm src}^{\rm prs}(D-\gamma')
    =
    \lambda\Err_{\rm src,q}^{\rm prs}(D).
\end{aligned}
\]
The case \(\lambda=0\) follows from positive homogeneity, which gives
\(\Err_{\rm src}^{\rm prs}(0)=0\), and from \(\|0\|_{\loc,0}=0\).
Adding the two homogeneous terms proves the first claim.

If \(D\in\Gamma_\Lambda^{\loc}\), then choosing \(\gamma=D\) gives
\(\Dist_{\loc,0}(D,\Gamma_\Lambda^{\loc})=0\) and
\(\Err_{\rm src,q}^{\rm prs}(D)=0\).  Hence
\(\Dist_{\loc,\alpha}^{\sharp}(D,\Gamma_\Lambda^{\loc})=0\).
\end{proof}

\begin{remark}[Seminorm and gauge status]
\Cref{lem:model-a-homogeneity-gauge-vanishing} proves the formal
homogeneity and gauge-vanishing properties of Model A.  The object
\(\Dist_{\loc,\alpha}^{\sharp}\) should nevertheless be read as a
finite-window quotient gauge distance.  It is a genuine norm on the intended
physical quotient only after the pressure-source residual and the localized
gauge subspace are verified to be compatible with the same physical gauge and
to separate non-gauge classes.  That compatibility is not proved here.
\end{remark}

\begin{assumption}[Model A enhanced transfer comparison]
\label{ass:model-a-enhanced-transfer-comparison}
There are constants \(0\le\varepsilon_G<1\) and \(\delta_G\ge0\) such that
every Model A localized package \(D\) satisfies
\[
    \Dist_{\cl}(\Theta_\Lambda D,\Gamma_\Lambda^{\cl})
    \ge
    (1-\varepsilon_G)
    \Dist_{\loc,\alpha}^{\sharp}(D,\Gamma_\Lambda^{\loc})
    -
    \delta_G
\]
and
\[
    \calM_\Lambda^{\loc}(D)
    +
    \Err_\Lambda^{\loc}(D)
    \ge
    \calM_\Lambda^{\rm comp}(\Theta_\Lambda D).
\]
\end{assumption}

\begin{assumption}[Model A enhanced residual budget]
\label{ass:model-a-enhanced-residual-budget}
There are constants \(\eta_\Lambda\ge0\) and \(\Delta_\Lambda\ge0\) such that
every Model A localized package \(D\) satisfies
\[
    \Err_\Lambda^{\loc}(D)
    \le
    \eta_\Lambda
    \Dist_{\loc,\alpha}^{\sharp}(D,\Gamma_\Lambda^{\loc})
    +
    \Delta_\Lambda.
\]
\end{assumption}

\begin{theorem}[Conditional Model A enhanced-distance localized transfer]
\label{thm:model-a-enhanced-distance-localized-transfer}
Assume \(\calQ_\Lambda^{\cl}\ne\{0\}\),
\Cref{ass:clean-detector-gauge-compatibility},
\Cref{ass:kernel-free-computational-detector},
\Cref{ass:model-a-enhanced-transfer-comparison}, and
\Cref{ass:model-a-enhanced-residual-budget}.  Then every Model A localized
package \(D\) satisfies
\[
\begin{aligned}
    \calM_\Lambda^{\loc}(D)
    \ge&
    \bigl(
        \mu_\Lambda^{\rm comp}(1-\varepsilon_G)
        -
        \eta_\Lambda
    \bigr)
    \Dist_{\loc,\alpha}^{\sharp}
    (D,\Gamma_\Lambda^{\loc})\\
    &-
    \mu_\Lambda^{\rm comp}\delta_G
    -
    \Delta_\Lambda.
\end{aligned}
\]
\end{theorem}

\begin{proof}
This is the localized transfer argument with the Model A distance
\(\Dist_{\loc,\alpha}^{\sharp}\) in place of the abstract localized distance.
By \Cref{thm:clean-computational-antiphantom},
\[
    \calM_\Lambda^{\rm comp}(\Theta_\Lambda D)
    \ge
    \mu_\Lambda^{\rm comp}
    \Dist_{\cl}(\Theta_\Lambda D,\Gamma_\Lambda^{\cl}).
\]
Using \Cref{ass:model-a-enhanced-transfer-comparison} gives
\[
    \calM_\Lambda^{\rm comp}(\Theta_\Lambda D)
    \ge
    \mu_\Lambda^{\rm comp}
    \bigl(
        (1-\varepsilon_G)
        \Dist_{\loc,\alpha}^{\sharp}(D,\Gamma_\Lambda^{\loc})
        -
        \delta_G
    \bigr).
\]
The detector comparison in
\Cref{ass:model-a-enhanced-transfer-comparison} then implies
\[
    \calM_\Lambda^{\loc}(D)
    \ge
    \mu_\Lambda^{\rm comp}(1-\varepsilon_G)
    \Dist_{\loc,\alpha}^{\sharp}(D,\Gamma_\Lambda^{\loc})
    -
    \mu_\Lambda^{\rm comp}\delta_G
    -
    \Err_\Lambda^{\loc}(D).
\]
Applying \Cref{ass:model-a-enhanced-residual-budget} gives the stated
estimate.
\end{proof}

\begin{remark}[Model A tradeoff]
Model A is a finite-window norm choice.  It makes pressure-source
observability part of the localized quotient geometry.  The remaining PDE
problem is to justify that this norm is natural for localized
Navier--Stokes packages, or to compare it with a more intrinsic quotient norm
derived from pressure splitting and local energy estimates.  No uniqueness or
canonicality of Model A, no scale-uniformity, and no Navier--Stokes regularity
conclusion is claimed.
\end{remark}

\subsection{Intrinsic localized norm candidate and Model A comparison}
\label{subsec:intrinsic-norm-candidate-model-a-comparison}

We next record the minimal comparison statement that would make Model A a
controlled finite-window coordinate norm.  The intrinsic norm below is still a
model datum.  Its role is to separate a purely finite-dimensional coordinate
comparison from the PDE problem of deriving such a norm from suitable weak
solutions, pressure splitting, and the local energy inequality.

\begin{definition}[Intrinsic localized package norm candidate]
\label{def:intrinsic-localized-package-norm-candidate}
For each \(k\in\Lambda\), let
\[
    \mathcal U_k,\quad
    \mathcal P_k,\quad
    \mathcal R_k,\quad
    \mathcal F_k,\quad
    \mathcal E_k,\quad
    \mathcal T_k,\quad
    \mathcal S_k
\]
be normed spaces for intrinsic localized velocity, active pressure,
Reynolds/covariance, flux, energy or trace, selected trace, and slack data.
The intended PDE examples are
\[
    \mathcal U_k=L^3(Q_k)^3,\qquad
    \mathcal P_k=L^{3/2}(Q_k),\qquad
    \mathcal R_k=L^{3/2}(Q_k)^{3\times3},
\]
with the remaining factors given by the corresponding finite-window flux,
energy, trace, and slack observables.  No scale-uniform assertion is attached
to these choices.

An intrinsic localized package is
\[
    \mathcal D
    =
    (u_k,p_k^{\rm act},r_k,\pi_k,\phi_k,\tau_k,\sigma_k)_{k\in\Lambda},
\]
with components in the seven spaces above.  Fix positive intrinsic weights
\[
    a_U,a_P,a_R,a_\Pi,a_\Phi,a_T,a_s>0,
\]
and define
\[
\begin{aligned}
    \|\mathcal D\|_{\loc,{\rm int}}^2
    :=
    \sum_{k\in\Lambda}
    \bigl(
        &a_U\|u_k\|_{\mathcal U_k}^2
        +
        a_P\|p_k^{\rm act}\|_{\mathcal P_k}^2
        +
        a_R\|r_k\|_{\mathcal R_k}^2\\
        &+
        a_\Pi\|\pi_k\|_{\mathcal F_k}^2
        +
        a_\Phi\|\phi_k\|_{\mathcal E_k}^2
        +
        a_T\|\tau_k\|_{\mathcal T_k}^2
        +
        a_s\|\sigma_k\|_{\mathcal S_k}^2
    \bigr).
\end{aligned}
\]
Let \(\Gamma_{\Lambda}^{\rm int}\) be a linear intrinsic gauge subspace and
define
\[
    \Dist_{\loc,{\rm int}}(\mathcal D,\Gamma_{\Lambda}^{\rm int})
    :=
    \inf_{\zeta\in\Gamma_{\Lambda}^{\rm int}}
    \|\mathcal D-\zeta\|_{\loc,{\rm int}}.
\]
\end{definition}

\begin{assumption}[Bounded coordinate extraction from the intrinsic norm]
\label{ass:bounded-coordinate-extraction}
There is a linear coordinate extraction map
\[
    \mathcal L_A:\mathfrak I_\Lambda^{\loc}\to\mathcal K_{\Lambda,A}^{\loc}
\]
from intrinsic packages to Model A coordinate packages, written
\[
    \mathcal L_A\mathcal D
    =
    (U_k,P_k^{\rm act},R_k,\Pi_k,\Phi_k,T_k,s_k)_{k\in\Lambda},
\]
and constants
\[
    \beta_{U,k},\beta_{P,k},\beta_{R,k},
    \beta_{\Pi,k},\beta_{\Phi,k},\beta_{T,k},\beta_{s,k}<\infty
\]
such that, for each \(k\in\Lambda\),
\[
\begin{gathered}
    \|U_k\|_{U,k}\le \beta_{U,k}\|u_k\|_{\mathcal U_k},\qquad
    \|P_k^{\rm act}\|_{P,k}\le
        \beta_{P,k}\|p_k^{\rm act}\|_{\mathcal P_k},\\
    \|R_k\|_{R,k}\le \beta_{R,k}\|r_k\|_{\mathcal R_k},\qquad
    \|\Pi_k\|_{\Pi,k}\le \beta_{\Pi,k}\|\pi_k\|_{\mathcal F_k},\\
    \|\Phi_k\|_{\Phi,k}\le
        \beta_{\Phi,k}\|\phi_k\|_{\mathcal E_k},\qquad
    \|T_k\|_{T,k}\le \beta_{T,k}\|\tau_k\|_{\mathcal T_k},\\
    \|s_k\|_{s,k}\le \beta_{s,k}\|\sigma_k\|_{\mathcal S_k}.
\end{gathered}
\]
Assume also that intrinsic gauges map into Model A gauges:
\[
    \mathcal L_A\Gamma_{\Lambda}^{\rm int}
    \subset
    \Gamma_\Lambda^{\loc}.
\]
\end{assumption}

\begin{lemma}[Intrinsic norm controls the Model A baseline distance]
\label{lem:intrinsic-controls-model-a-baseline}
Under \Cref{ass:bounded-coordinate-extraction}, there is a finite constant
\[
    C_{A\leftarrow{\rm int}}<\infty
\]
depending only on the finite window, the Model A weights, the intrinsic
weights, and the coordinate-extraction constants, such that
\[
    \|\mathcal L_A\mathcal D\|_{\loc,0}
    \le
    C_{A\leftarrow{\rm int}}
    \|\mathcal D\|_{\loc,{\rm int}}
\]
and
\[
    \Dist_{\loc,0}(\mathcal L_A\mathcal D,\Gamma_\Lambda^{\loc})
    \le
    C_{A\leftarrow{\rm int}}
    \Dist_{\loc,{\rm int}}(\mathcal D,\Gamma_{\Lambda}^{\rm int})
\]
for every intrinsic package \(\mathcal D\).
\end{lemma}

\begin{proof}
Set
\[
    C_{A\leftarrow{\rm int}}^2
    :=
    \max_{k\in\Lambda}
    \max\left\{
    \frac{w_U\beta_{U,k}^2}{a_U},
    \frac{w_P\beta_{P,k}^2}{a_P},
    \frac{w_R\beta_{R,k}^2}{a_R},
    \frac{w_\Pi\beta_{\Pi,k}^2}{a_\Pi},
    \frac{w_\Phi\beta_{\Phi,k}^2}{a_\Phi},
    \frac{w_T\beta_{T,k}^2}{a_T},
    \frac{w_s\beta_{s,k}^2}{a_s}
    \right\}.
\]
This constant is finite because the window is finite and all listed constants
are finite.  The coordinate bounds in
\Cref{ass:bounded-coordinate-extraction} imply, term by term, that
\[
    \|\mathcal L_A\mathcal D\|_{\loc,0}^2
    \le
    C_{A\leftarrow{\rm int}}^2
    \|\mathcal D\|_{\loc,{\rm int}}^2.
\]
This proves the norm estimate.

For the quotient estimate, fix
\(\zeta\in\Gamma_{\Lambda}^{\rm int}\).  Since
\(\mathcal L_A\zeta\in\Gamma_\Lambda^{\loc}\),
\[
\begin{aligned}
    \Dist_{\loc,0}(\mathcal L_A\mathcal D,\Gamma_\Lambda^{\loc})
    &\le
    \|\mathcal L_A\mathcal D-\mathcal L_A\zeta\|_{\loc,0}\\
    &=
    \|\mathcal L_A(\mathcal D-\zeta)\|_{\loc,0}\\
    &\le
    C_{A\leftarrow{\rm int}}
    \|\mathcal D-\zeta\|_{\loc,{\rm int}}.
\end{aligned}
\]
Taking the infimum over \(\zeta\in\Gamma_{\Lambda}^{\rm int}\) gives the
claim.
\end{proof}

\begin{assumption}[Intrinsic pressure-source control]
\label{ass:intrinsic-pressure-source-control}
There are constants
\[
    C_{{\rm src},{\rm int}}<\infty,
    \qquad
    \Delta_{{\rm src},{\rm int}}\ge0,
\]
such that every intrinsic localized package \(\mathcal D\) satisfies
\[
    \Err_{\rm src,q}^{\rm prs}(\mathcal L_A\mathcal D)
    \le
    C_{{\rm src},{\rm int}}
    \Dist_{\loc,{\rm int}}(\mathcal D,\Gamma_{\Lambda}^{\rm int})
    +
    \Delta_{{\rm src},{\rm int}}.
\]
\end{assumption}

\begin{proposition}[Conditional intrinsic control of the Model A enhanced distance]
\label{prop:conditional-intrinsic-control-model-a-enhanced-distance}
Assume \Cref{ass:bounded-coordinate-extraction} and
\Cref{ass:intrinsic-pressure-source-control}.  Then every intrinsic localized
package \(\mathcal D\) satisfies
\[
\begin{aligned}
    \Dist_{\loc,\alpha}^{\sharp}
    (\mathcal L_A\mathcal D,\Gamma_\Lambda^{\loc})
    \le&
    \bigl(
        C_{A\leftarrow{\rm int}}
        +
        \alpha_{\rm src}C_{{\rm src},{\rm int}}
    \bigr)
    \Dist_{\loc,{\rm int}}(\mathcal D,\Gamma_{\Lambda}^{\rm int})\\
    &+
    \alpha_{\rm src}\Delta_{{\rm src},{\rm int}}.
\end{aligned}
\]
\end{proposition}

\begin{proof}
By \Cref{def:model-a-enhanced-localized-distance},
\[
    \Dist_{\loc,\alpha}^{\sharp}
    (\mathcal L_A\mathcal D,\Gamma_\Lambda^{\loc})
    =
    \Dist_{\loc,0}(\mathcal L_A\mathcal D,\Gamma_\Lambda^{\loc})
    +
    \alpha_{\rm src}
    \Err_{\rm src,q}^{\rm prs}(\mathcal L_A\mathcal D).
\]
Apply \Cref{lem:intrinsic-controls-model-a-baseline} to the first term and
\Cref{ass:intrinsic-pressure-source-control} to the second.
\end{proof}

\begin{remark}[What remains PDE-facing]
\Cref{prop:conditional-intrinsic-control-model-a-enhanced-distance} proves a
finite-window implication from two explicit inputs: bounded extraction of
Model A coordinates from intrinsic localized data, and intrinsic control of
the pressure-source residual.  It does not derive either input from the
Navier--Stokes equations.  The PDE-facing work is to construct the intrinsic
package from a suitable weak solution, verify the coordinate bounds from the
chosen pressure splitting and local energy quantities, and prove or disprove
\Cref{ass:intrinsic-pressure-source-control}.
\end{remark}

\subsection{Cutoff commutator as an intrinsic pressure-source component}
\label{subsec:cutoff-commutator-intrinsic-component}

We now test the first component in the intrinsic pressure-source control
assumption.  The cutoff--Riesz commutator is the most favorable term because
the source is separated from the observation ball.  Even here, the natural
estimate is quadratic in the intrinsic velocity size, reflecting the nonlinear
source \(u_i u_j\).

\begin{definition}[Annular intrinsic velocity quotient]
\label{def:annular-intrinsic-velocity-quotient}
In the normalized pressure-source geometry, set
\[
    A_{3/4,1}:=B_1\setminus B_{3/4}.
\]
For an intrinsic package \(\mathcal D\), let \(u\) denote its velocity
component in this normalized window.  For
\(\zeta\in\Gamma_\Lambda^{\rm int}\), let \(\zeta_U\) denote the velocity
component of \(\zeta\).  Define
\[
    d_{U,{\rm ann}}(\mathcal D)
    :=
    \inf_{\zeta\in\Gamma_\Lambda^{\rm int}}
    \|u-\zeta_U\|_{L^3(I;L^3(A_{3/4,1}))}.
\]
\end{definition}

\begin{lemma}[Quadratic cutoff--Riesz commutator bound]
\label{lem:quadratic-cutoff-riesz-commutator-bound}
There is a constant \(C_\eta<\infty\), depending only on the fixed normalized
cutoff and balls, such that for every intrinsic package \(\mathcal D\),
\[
    \inf_{\zeta\in\Gamma_\Lambda^{\rm int}}
    \left\|
        C_\eta\bigl((u-\zeta_U)\otimes(u-\zeta_U)\bigr)
    \right\|_{Y_{\rm prs}}
    \le
    C_\eta
    d_{U,{\rm ann}}(\mathcal D)^2.
\]
\end{lemma}

\begin{proof}
Fix \(\zeta\in\Gamma_\Lambda^{\rm int}\) and set
\[
    v:=u-\zeta_U,
    \qquad
    f_{ij}:=v_i v_j.
\]
By the fixed separated-support estimate in
\Cref{lem:fixed-geometry-commutator-bound},
\[
    \|C_\eta(f)\|_{Y_{\rm prs}}
    \le
    C_\eta
    \|(1-\eta)f\|_{L^{3/2}(I;L^{3/2}(A_{3/4,1}))^{3\times3}}.
\]
Since \(\eta\equiv1\) on \(B_{3/4}\), the factor \(1-\eta\) is supported in
the annulus.  Hölder's inequality gives
\[
    \|(1-\eta)v_i v_j\|_{L^{3/2}(I;L^{3/2}(A_{3/4,1}))}
    \le
    C_\eta
    \|v\|_{L^3(I;L^3(A_{3/4,1}))}^2.
\]
Summing over the finitely many tensor components changes only the constant.
Thus
\[
    \left\|C_\eta(v\otimes v)\right\|_{Y_{\rm prs}}
    \le
    C_\eta
    \|u-\zeta_U\|_{L^3(I;L^3(A_{3/4,1}))}^2.
\]
Taking the infimum over \(\zeta\in\Gamma_\Lambda^{\rm int}\) proves the
claim.
\end{proof}

\begin{corollary}[Finite-amplitude linear commutator control]
\label{cor:finite-amplitude-linear-commutator-control}
Assume the intrinsic norm dominates the annular velocity quotient in the
sense that
\[
    d_{U,{\rm ann}}(\mathcal D)
    \le
    C_{U,{\rm ann}}
    \Dist_{\loc,{\rm int}}(\mathcal D,\Gamma_\Lambda^{\rm int})
\]
for a finite constant \(C_{U,{\rm ann}}\).  If, in addition,
\[
    d_{U,{\rm ann}}(\mathcal D)\le M,
\]
then
\[
    \inf_{\zeta\in\Gamma_\Lambda^{\rm int}}
    \left\|
        C_\eta\bigl((u-\zeta_U)\otimes(u-\zeta_U)\bigr)
    \right\|_{Y_{\rm prs}}
    \le
    C_\eta M C_{U,{\rm ann}}
    \Dist_{\loc,{\rm int}}(\mathcal D,\Gamma_\Lambda^{\rm int}).
\]
\end{corollary}

\begin{proof}
By \Cref{lem:quadratic-cutoff-riesz-commutator-bound},
\[
    \inf_{\zeta\in\Gamma_\Lambda^{\rm int}}
    \left\|
        C_\eta\bigl((u-\zeta_U)\otimes(u-\zeta_U)\bigr)
    \right\|_{Y_{\rm prs}}
    \le
    C_\eta d_{U,{\rm ann}}(\mathcal D)^2.
\]
If \(d_{U,{\rm ann}}(\mathcal D)\le M\), then
\[
    d_{U,{\rm ann}}(\mathcal D)^2
    \le
    M d_{U,{\rm ann}}(\mathcal D)
    \le
    M C_{U,{\rm ann}}
    \Dist_{\loc,{\rm int}}(\mathcal D,\Gamma_\Lambda^{\rm int}).
\]
This proves the estimate.
\end{proof}

\begin{remark}[Quadratic obstruction]
\Cref{lem:quadratic-cutoff-riesz-commutator-bound} is a genuine fixed-scale
pressure estimate, but it is quadratic in the annular velocity quotient.  It
does not by itself prove the linear intrinsic pressure-source control
assumption in \Cref{ass:intrinsic-pressure-source-control}.  The linear form
in \Cref{cor:finite-amplitude-linear-commutator-control} requires a
finite-amplitude restriction and is not a scale-uniform smallness statement.
The active source residual, clean projection residual, and harmonic gauge
residual remain separate components.
\end{remark}

\subsection{Active source residual and covariance mismatch}
\label{subsec:active-source-covariance-mismatch}

We next isolate the active source residual.  The point is to avoid hiding a
nonlinear covariance mismatch inside the pressure operator.  Once the cutoff
commutator has been separated, the remaining active source term is controlled
by a source-level mismatch through the \(L^{3/2}\)-boundedness of the Riesz
transforms.

\begin{definition}[Active covariance mismatch]
\label{def:active-covariance-mismatch}
In the normalized pressure-source geometry, define the active source mismatch
by
\[
    M_{ij}^{\rm act}
    :=
    \eta f_{ij}-F_{ij}^{\cl}
    =
    \eta u_i u_j-(U_iU_j+R_{ij}),
\]
with all source tensors extended by zero outside \(B_1\).  For an intrinsic
package \(\mathcal D\), define the quotient source mismatch size
\[
    d_{\rm act,src}(\mathcal D)
    :=
    \inf_{\zeta\in\Gamma_\Lambda^{\rm int}}
    \|M^{\rm act}(\mathcal D-\zeta)\|_
    {L^{3/2}(I;L^{3/2}(B_1))^{3\times3}}.
\]
\end{definition}

\begin{lemma}[Active residual source decomposition]
\label{lem:active-residual-source-decomposition}
In the normalized model,
\[
    E_{\rm act}^{\rm src}
    =
    R_iR_j(M_{ij}^{\rm act})
    -
    C_\eta(f).
\]
Consequently, there is a finite constant \(C_{\rm Riesz}\), depending only on
the Riesz-transform bound at exponent \(3/2\), such that
\[
    \|E_{\rm act}^{\rm src}\|_{Y_{\rm prs}}
    \le
    C_{\rm Riesz}
    \|M^{\rm act}\|_{L^{3/2}(I;L^{3/2}(B_1))^{3\times3}}
    +
    \|C_\eta(f)\|_{Y_{\rm prs}}.
\]
\end{lemma}

\begin{proof}
By definition,
\[
    C_\eta(f)=R_iR_j(\eta f_{ij})-\eta R_iR_j(f_{ij}).
\]
Thus
\[
    \eta R_iR_j(f_{ij})
    =
    R_iR_j(\eta f_{ij})-C_\eta(f).
\]
Using the definition of \(E_{\rm act}^{\rm src}\),
\[
\begin{aligned}
    E_{\rm act}^{\rm src}
    &=
    \eta R_iR_j(f_{ij})-R_iR_j(F_{ij}^{\cl})\\
    &=
    R_iR_j(\eta f_{ij})-R_iR_j(F_{ij}^{\cl})-C_\eta(f)\\
    &=
    R_iR_j(M_{ij}^{\rm act})-C_\eta(f).
\end{aligned}
\]
Taking the \(Y_{\rm prs}\)-norm, using the triangle inequality, and applying
the \(L^{3/2}\)-boundedness of \(R_iR_j\) gives the estimate.
\end{proof}

\begin{proposition}[Common-gauge active-source residual control]
\label{prop:conditional-active-source-residual-control}
For every intrinsic package \(\mathcal D\),
\[
\begin{aligned}
    &\inf_{\zeta\in\Gamma_\Lambda^{\rm int}}
    \|E_{\rm act}^{\rm src}(\mathcal D-\zeta)\|_{Y_{\rm prs}}\\
    &\quad\le
    \inf_{\zeta\in\Gamma_\Lambda^{\rm int}}
    \bigl(
    C_{\rm Riesz}
    \|M^{\rm act}(\mathcal D-\zeta)\|_
    {L^{3/2}(I;L^{3/2}(B_1))^{3\times3}}
    +
    C_\eta
    \|u_\zeta\|_{L^3(I;L^3(A_{3/4,1}))}^2
    \bigr),
\end{aligned}
\]
where \(u_\zeta\) is the velocity component of \(\mathcal D-\zeta\).  In
particular, if there is a common representative \(\zeta_\diamond\) and constants
\(M_U,C_U,\eta_{\rm act},\Delta_{\rm act}\ge0\) such that, with
\(\rho_{\rm int}(\mathcal D)=\Dist_{\loc,{\rm int}}(\mathcal D,\Gamma_\Lambda^{\rm int})\),
\[
    \|u_{\zeta_\diamond}\|_{L^3(I;L^3(A_{3/4,1}))}\le M_U,
    \qquad
    \|u_{\zeta_\diamond}\|_{L^3(I;L^3(A_{3/4,1}))}\le C_U\rho_{\rm int}(\mathcal D),
\]
and
\[
    C_{\rm Riesz}
    \|M^{\rm act}(\mathcal D-\zeta_\diamond)\|_
    {L^{3/2}(I;L^{3/2}(B_1))^{3\times3}}
    \le
    \eta_{\rm act}\rho_{\rm int}(\mathcal D)+\Delta_{\rm act},
\]
then
\[
\begin{aligned}
    \inf_{\zeta\in\Gamma_\Lambda^{\rm int}}
    \|E_{\rm act}^{\rm src}(\mathcal D-\zeta)\|_{Y_{\rm prs}}
    \le
    (\eta_{\rm act}+C_\eta M_U C_U)
    \rho_{\rm int}(\mathcal D)+\Delta_{\rm act}.
\end{aligned}
\]
\end{proposition}

\begin{proof}
Apply \Cref{lem:active-residual-source-decomposition} to the same representative
\(\mathcal D-\zeta\) and then use the fixed separated-support estimate for the
commutator.  This gives, for every \(\zeta\in\Gamma_\Lambda^{\rm int}\),
\[
\begin{aligned}
    \|E_{\rm act}^{\rm src}(\mathcal D-\zeta)\|_{Y_{\rm prs}}
    \le&
    C_{\rm Riesz}
    \|M^{\rm act}(\mathcal D-\zeta)\|_
    {L^{3/2}(I;L^{3/2}(B_1))^{3\times3}}\\
    &+
    C_\eta
    \|u_\zeta\|_{L^3(I;L^3(A_{3/4,1}))}^2.
\end{aligned}
\]
Taking the infimum over the same \(\zeta\) proves the first estimate.  The
second estimate follows by evaluating this common-gauge upper bound at
\(\zeta_\diamond\) and using
\(\|u_{\zeta_\diamond}\|^2\le M_UC_U\rho_{\rm int}(\mathcal D)\).
\end{proof}

\begin{remark}[Active-source status]
\Cref{prop:conditional-active-source-residual-control} reduces the active
pressure residual to a source-level covariance mismatch and the already
identified cutoff commutator only on a common representative.  It deliberately
avoids taking separate infima for the covariance mismatch and the commutator,
because separate optimizing gauges would not control the sum.  Control by the
intrinsic quotient distance must therefore be supplied by a same-gauge
localized covariance coordinate or treated as an obstruction.
\end{remark}

\subsection{Clean projection residual as a pressure-source tail}
\label{subsec:clean-projection-residual-pressure-source-tail}

The clean projection residual is the part of the clean active pressure source
lost by the chosen finite-dimensional pressure projection.  This subsection
records only boundedness of that tail.  No decay in the projection dimension
is claimed.

\begin{assumption}[Bounded clean pressure projection]
\label{ass:bounded-clean-pressure-projection}
For a fixed finite-dimensional pressure space and an index \(N\), let
\[
    P_{{\rm prs},N}^{\cl}:Y_{\rm prs}\to Y_{\rm prs}
\]
be a bounded linear projection.  Its operator norm is denoted by
\[
    \|P_{{\rm prs},N}^{\cl}\|_{Y_{\rm prs}\to Y_{\rm prs}}.
\]
\end{assumption}

\begin{definition}[Clean projection residual]
\label{def:clean-projection-residual-tail}
Define
\[
    E_{{\rm proj},N}^{\cl}
    :=
    (I-P_{{\rm prs},N}^{\cl})R_iR_j(F_{ij}^{\cl}).
\]
For an intrinsic package \(D\), define the quotient clean projection residual
by
\[
    d_{{\rm proj},q}(D)
    :=
    \inf_{\zeta\in\Gamma_\Lambda^{\rm int}}
    \|E_{{\rm proj},N}^{\cl}(D-\zeta)\|_{Y_{\rm prs}}.
\]
\end{definition}

\begin{lemma}[Bounded projection-tail estimate]
\label{lem:bounded-projection-tail-estimate}
Under \Cref{ass:bounded-clean-pressure-projection},
\[
    \|E_{{\rm proj},N}^{\cl}\|_{Y_{\rm prs}}
    \le
    \bigl(
        1+\|P_{{\rm prs},N}^{\cl}\|_{Y_{\rm prs}\to Y_{\rm prs}}
    \bigr)
    \|R_iR_j(F_{ij}^{\cl})\|_{Y_{\rm prs}}.
\]
If, moreover,
\[
    F^{\cl}\in L^{3/2}(I;L^{3/2}(B_1))^{3\times3},
\]
with zero extension outside \(B_1\), then
\[
    \|E_{{\rm proj},N}^{\cl}\|_{Y_{\rm prs}}
    \le
    \bigl(
        1+\|P_{{\rm prs},N}^{\cl}\|_{Y_{\rm prs}\to Y_{\rm prs}}
    \bigr)
    C_{\rm CZ}
    \|F^{\cl}\|_{L^{3/2}(I;L^{3/2}(B_1))^{3\times3}},
\]
where \(C_{\rm CZ}\) is the Calderon--Zygmund constant for \(R_iR_j\) at
exponent \(3/2\), followed by restriction to \(B_{1/2}\).
\end{lemma}

\begin{proof}
By definition,
\[
    E_{{\rm proj},N}^{\cl}
    =
    R_iR_j(F_{ij}^{\cl})
    -
    P_{{\rm prs},N}^{\cl}R_iR_j(F_{ij}^{\cl}).
\]
The triangle inequality and boundedness of \(P_{{\rm prs},N}^{\cl}\) give
\[
\begin{aligned}
    \|E_{{\rm proj},N}^{\cl}\|_{Y_{\rm prs}}
    &\le
    \|R_iR_j(F_{ij}^{\cl})\|_{Y_{\rm prs}}
    +
    \|P_{{\rm prs},N}^{\cl}R_iR_j(F_{ij}^{\cl})\|_{Y_{\rm prs}}\\
    &\le
    \bigl(
        1+\|P_{{\rm prs},N}^{\cl}\|_{Y_{\rm prs}\to Y_{\rm prs}}
    \bigr)
    \|R_iR_j(F_{ij}^{\cl})\|_{Y_{\rm prs}}.
\end{aligned}
\]
If \(F^{\cl}\in L^{3/2}(I;L^{3/2}(B_1))^{3\times3}\), then the
Calderon--Zygmund estimate for the zero extension of \(F^{\cl}\) gives
\[
    \|R_iR_j(F_{ij}^{\cl})\|_{Y_{\rm prs}}
    \le
    C_{\rm CZ}
    \|F^{\cl}\|_{L^{3/2}(I;L^{3/2}(B_1))^{3\times3}}.
\]
Substituting this into the first estimate proves the second.
\end{proof}

\begin{proposition}[Quotient projection-tail bound]
\label{prop:quotient-projection-tail-bound}
Under \Cref{ass:bounded-clean-pressure-projection}, every intrinsic package
\(D\) satisfies
\[
\begin{aligned}
    d_{{\rm proj},q}(D)
    \le&
    \bigl(
        1+\|P_{{\rm prs},N}^{\cl}\|_{Y_{\rm prs}\to Y_{\rm prs}}
    \bigr)
    \inf_{\zeta\in\Gamma_\Lambda^{\rm int}}
    \|R_iR_j(F_{ij}^{\cl}(D-\zeta))\|_{Y_{\rm prs}}.
\end{aligned}
\]
If \(F^{\cl}(D-\zeta)\in
L^{3/2}(I;L^{3/2}(B_1))^{3\times3}\) for the relevant representatives, then
\[
\begin{aligned}
    d_{{\rm proj},q}(D)
    \le&
    \bigl(
        1+\|P_{{\rm prs},N}^{\cl}\|_{Y_{\rm prs}\to Y_{\rm prs}}
    \bigr)
    C_{\rm CZ}\\
    &\times
    \inf_{\zeta\in\Gamma_\Lambda^{\rm int}}
    \|F^{\cl}(D-\zeta)\|_{L^{3/2}(I;L^{3/2}(B_1))^{3\times3}}.
\end{aligned}
\]
\end{proposition}

\begin{proof}
Apply \Cref{lem:bounded-projection-tail-estimate} to \(D-\zeta\) for an
arbitrary \(\zeta\in\Gamma_\Lambda^{\rm int}\), and then take the infimum
over \(\zeta\).  The Calderon--Zygmund form follows in the same way from the
second estimate in \Cref{lem:bounded-projection-tail-estimate}.
\end{proof}

\begin{remark}[Projection tail]
The clean projection residual is a finite-window tail component.  Boundedness
of the clean projection only proves that the tail is a well-defined controlled
residual.  Any decay in \(N\), compactness of the clean pressure source, or
absorption into the original localized quotient distance is a separate theorem
target.
\end{remark}

\subsection{Harmonic gauge residual as a pressure-source component}
\label{subsec:harmonic-gauge-residual-pressure-source-component}

The final term in the pressure-source decomposition is the harmonic gauge
tail.  We keep the discussion at a fixed finite-window level.  The purpose is
only to record that the retained harmonic tail is a well-defined pressure
source component once a harmonic gauge projection and observation map have
been fixed.

\begin{assumption}[Fixed harmonic pressure observation datum]
\label{ass:fixed-harmonic-pressure-observation-datum}
Let \(Y_{\rm harm}\) be a fixed finite-window normed space for harmonic
pressure components.  Let
\[
    \mathcal H_M\subset Y_{\rm harm}
\]
be a chosen finite-dimensional harmonic gauge space, and let
\[
    \Pi_{\rm harm,M}:Y_{\rm harm}\to \mathcal H_M
\]
be a bounded projection.  If the harmonic pressure norm is not already
realized inside the pressure observation space \(Y_{\rm prs}\), we also fix a
bounded observation map
\[
    J_{\rm harm}:Y_{\rm harm}\to Y_{\rm prs}.
\]
The operator norms are denoted by
\[
    \|\Pi_{\rm harm,M}\|_{Y_{\rm harm}\to Y_{\rm harm}},
    \qquad
    \|J_{\rm harm}\|_{Y_{\rm harm}\to Y_{\rm prs}}.
\]
\end{assumption}

\begin{definition}[Harmonic gauge residual]
\label{def:harmonic-gauge-residual-pressure-source}
For a harmonic pressure component \(p_{\rm harm}(D)\in Y_{\rm harm}\), define
the retained harmonic tail by
\[
    h_{\rm harm}(D)
    :=
    (I-\Pi_{\rm harm,M})p_{\rm harm}(D).
\]
In the direct-observation convention, where the tail is already an element of
\(Y_{\rm prs}\), define
\[
    E_{\rm harm}(D):=-h_{\rm harm}(D).
\]
In the observation-map convention of
\Cref{ass:fixed-harmonic-pressure-observation-datum}, the same notation means
\[
    E_{\rm harm}(D):=-J_{\rm harm}h_{\rm harm}(D).
\]
This is the harmonic term \(E_{\rm harm}=-h^{\rm harm}\) appearing in
\Cref{lem:pressure-source-mismatch-decomposition}.
\end{definition}

\begin{lemma}[Fixed-window harmonic residual bound]
\label{lem:fixed-window-harmonic-residual-bound}
In the direct-observation convention,
\[
    \|E_{\rm harm}(D)\|_{Y_{\rm prs}}
    =
    \|h_{\rm harm}(D)\|_{Y_{\rm prs}}.
\]
Under the observation-map convention of
\Cref{ass:fixed-harmonic-pressure-observation-datum},
\[
    \|E_{\rm harm}(D)\|_{Y_{\rm prs}}
    \le
    \|J_{\rm harm}\|_{Y_{\rm harm}\to Y_{\rm prs}}
    \|(I-\Pi_{\rm harm,M})p_{\rm harm}(D)\|_{Y_{\rm harm}},
\]
and hence
\[
    \|E_{\rm harm}(D)\|_{Y_{\rm prs}}
    \le
    \|J_{\rm harm}\|_{Y_{\rm harm}\to Y_{\rm prs}}
    \bigl(
        1+\|\Pi_{\rm harm,M}\|_{Y_{\rm harm}\to Y_{\rm harm}}
    \bigr)
    \|p_{\rm harm}(D)\|_{Y_{\rm harm}}.
\]
\end{lemma}

\begin{proof}
In the direct-observation convention,
\[
    E_{\rm harm}(D)=-h_{\rm harm}(D),
\]
so the first identity follows from homogeneity of the norm.  In the
observation-map convention,
\[
    E_{\rm harm}(D)
    =
    -J_{\rm harm}(I-\Pi_{\rm harm,M})p_{\rm harm}(D).
\]
Boundedness of \(J_{\rm harm}\) gives the first inequality.  The second
follows from
\[
    \|(I-\Pi_{\rm harm,M})p_{\rm harm}(D)\|_{Y_{\rm harm}}
    \le
    \bigl(
        1+\|\Pi_{\rm harm,M}\|_{Y_{\rm harm}\to Y_{\rm harm}}
    \bigr)
    \|p_{\rm harm}(D)\|_{Y_{\rm harm}}.
\]
\end{proof}

\begin{definition}[Quotient harmonic gauge residual]
\label{def:quotient-harmonic-gauge-residual}
Define
\[
    d_{{\rm harm},q}(D)
    :=
    \inf_{\zeta\in\Gamma_\Lambda^{\rm int}}
    \|E_{\rm harm}(D-\zeta)\|_{Y_{\rm prs}}.
\]
\end{definition}

\begin{proposition}[Quotient harmonic-tail bound]
\label{prop:quotient-harmonic-tail-bound}
Under \Cref{ass:fixed-harmonic-pressure-observation-datum},
\[
    d_{{\rm harm},q}(D)
    \le
    \|J_{\rm harm}\|_{Y_{\rm harm}\to Y_{\rm prs}}
    \inf_{\zeta\in\Gamma_\Lambda^{\rm int}}
    \|(I-\Pi_{\rm harm,M})p_{\rm harm}(D-\zeta)\|_{Y_{\rm harm}}.
\]
Consequently,
\[
    d_{{\rm harm},q}(D)
    \le
    \|J_{\rm harm}\|_{Y_{\rm harm}\to Y_{\rm prs}}
    \bigl(
        1+\|\Pi_{\rm harm,M}\|_{Y_{\rm harm}\to Y_{\rm harm}}
    \bigr)
    \inf_{\zeta\in\Gamma_\Lambda^{\rm int}}
    \|p_{\rm harm}(D-\zeta)\|_{Y_{\rm harm}}.
\]
\end{proposition}

\begin{proof}
Apply \Cref{lem:fixed-window-harmonic-residual-bound} to \(D-\zeta\) and
then take the infimum over \(\zeta\in\Gamma_\Lambda^{\rm int}\).  The second
estimate follows from the second bound in
\Cref{lem:fixed-window-harmonic-residual-bound} before taking the infimum.
\end{proof}

\begin{remark}[Harmonic tail]
The harmonic gauge residual is a fixed finite-window tail component.
Boundedness of the harmonic projection and of the harmonic observation map
only proves that the residual is a well-defined controlled component.  Any
decay in the harmonic degree \(M\), compatibility with the physical pressure
gauge, or absorption into the original localized quotient distance is a
separate theorem target.
\end{remark}

\subsection{Assembled intrinsic pressure-source component budget}
\label{subsec:assembled-intrinsic-pressure-source-component-budget}

We now collect the four pressure-source components in the intrinsic quotient
geometry.  The important point is that the quotient infimum must be taken over
a single common gauge representative after the component norms have been
summed.  Taking separate infima for the four components would use different
gauges and would not, by itself, control the pressure-source residual.

\begin{definition}[Intrinsic quotient pressure-source residual]
\label{def:intrinsic-quotient-pressure-source-residual}
For an intrinsic package \(\mathcal D\), define
\[
    \Err_{{\rm src},{\rm int}}^{\rm prs}(\mathcal D)
    :=
    \inf_{\zeta\in\Gamma_\Lambda^{\rm int}}
    \|\mathfrak P_{\rm src}(\mathcal D-\zeta)\|_{Y_{\rm prs}},
\]
where \(\mathfrak P_{\rm src}\) is the concrete pressure-source mismatch from
\Cref{def:concrete-pressure-source-mismatch}, evaluated through the
normalized intrinsic pressure-source coordinates.
\end{definition}

\begin{definition}[Common-gauge intrinsic component budget]
\label{def:common-gauge-intrinsic-component-budget}
For \(\zeta\in\Gamma_\Lambda^{\rm int}\), let
\[
    C_\eta^\zeta,\qquad
    E_{{\rm act},\zeta}^{\rm src},\qquad
    E_{{\rm proj},\zeta}^{\cl},\qquad
    E_{{\rm harm},\zeta}
\]
denote the cutoff commutator, active source residual, clean projection
residual, and harmonic gauge residual evaluated on \(\mathcal D-\zeta\).
Define
\[
\begin{aligned}
    \mathfrak B_{{\rm src},{\rm int}}^{\rm prs}(\mathcal D)
    :=
    \inf_{\zeta\in\Gamma_\Lambda^{\rm int}}
    \bigl(
        &\|C_\eta^\zeta\|_{Y_{\rm prs}}
        +
        \|E_{{\rm act},\zeta}^{\rm src}\|_{Y_{\rm prs}}\\
        &+
        \|E_{{\rm proj},\zeta}^{\cl}\|_{Y_{\rm prs}}
        +
        \|E_{{\rm harm},\zeta}\|_{Y_{\rm prs}}
    \bigr).
\end{aligned}
\]
\end{definition}

\begin{definition}[Resolved intrinsic source-size budget]
\label{def:resolved-intrinsic-source-size-budget}
For \(\zeta\in\Gamma_\Lambda^{\rm int}\), write \(u_\zeta\) for the velocity
component of \(\mathcal D-\zeta\), \(M_\zeta^{\rm act}\) for the active
covariance mismatch of \(\mathcal D-\zeta\), \(F_\zeta^{\cl}\) for the clean
source \(F^{\cl}\) associated with \(\mathcal D-\zeta\), and
\(p_{{\rm harm},\zeta}\) for its harmonic pressure component.  Define
\[
\begin{aligned}
    \mathfrak S_{{\rm src},{\rm int}}^{\rm prs}(\mathcal D)
    :=
    \inf_{\zeta\in\Gamma_\Lambda^{\rm int}}
    \bigl(&
        2C_\eta
        \|u_\zeta\|_{L^3(I;L^3(A_{3/4,1}))}^2\\
        &+
        C_{\rm Riesz}
        \|M_\zeta^{\rm act}\|
        _{L^{3/2}(I;L^{3/2}(B_1))^{3\times3}}\\
        &+
        \bigl(
            1+\|P_{{\rm prs},N}^{\cl}\|_{Y_{\rm prs}\to Y_{\rm prs}}
        \bigr)
        \|R_iR_j(F_{\zeta,ij}^{\cl})\|_{Y_{\rm prs}}\\
        &+
        \|J_{\rm harm}\|_{Y_{\rm harm}\to Y_{\rm prs}}
        \|(I-\Pi_{\rm harm,M})p_{{\rm harm},\zeta}\|_{Y_{\rm harm}}
    \bigr).
\end{aligned}
\]
If the clean source is measured directly in
\(L^{3/2}(I;L^{3/2}(B_1))^{3\times3}\), the third line may be replaced by
\[
    \bigl(
        1+\|P_{{\rm prs},N}^{\cl}\|_{Y_{\rm prs}\to Y_{\rm prs}}
    \bigr)
    C_{\rm CZ}
    \|F_\zeta^{\cl}\|_{L^{3/2}(I;L^{3/2}(B_1))^{3\times3}}.
\]
\end{definition}

\begin{proposition}[Intrinsic pressure-source component budget]
\label{prop:intrinsic-pressure-source-component-budget}
Under the fixed-window pressure-source assumptions above,
\[
    \Err_{{\rm src},{\rm int}}^{\rm prs}(\mathcal D)
    \le
    \mathfrak B_{{\rm src},{\rm int}}^{\rm prs}(\mathcal D)
    \le
    \mathfrak S_{{\rm src},{\rm int}}^{\rm prs}(\mathcal D).
\]
The same conclusion holds with the Calderon--Zygmund version of
\(\mathfrak S_{{\rm src},{\rm int}}^{\rm prs}\) when the clean sources have
the stated \(L^{3/2}\) integrability.
\end{proposition}

\begin{proof}
Fix \(\zeta\in\Gamma_\Lambda^{\rm int}\).  Applying
\Cref{lem:pressure-source-mismatch-decomposition} to
\(\mathcal D-\zeta\) and then using the triangle inequality gives
\[
\begin{aligned}
    \|\mathfrak P_{\rm src}(\mathcal D-\zeta)\|_{Y_{\rm prs}}
    \le&
        \|C_\eta^\zeta\|_{Y_{\rm prs}}
        +
        \|E_{{\rm act},\zeta}^{\rm src}\|_{Y_{\rm prs}}\\
        &+
        \|E_{{\rm proj},\zeta}^{\cl}\|_{Y_{\rm prs}}
        +
        \|E_{{\rm harm},\zeta}\|_{Y_{\rm prs}}.
\end{aligned}
\]
Taking the infimum over the same gauge representative \(\zeta\) proves
\[
    \Err_{{\rm src},{\rm int}}^{\rm prs}(\mathcal D)
    \le
    \mathfrak B_{{\rm src},{\rm int}}^{\rm prs}(\mathcal D).
\]

For the second inequality, use the fixed component estimates on the same
representative \(\mathcal D-\zeta\).  The separated commutator estimate gives
\[
    \|C_\eta^\zeta\|_{Y_{\rm prs}}
    \le
    C_\eta
    \|u_\zeta\|_{L^3(I;L^3(A_{3/4,1}))}^2.
\]
\Cref{lem:active-residual-source-decomposition} gives
\[
    \|E_{{\rm act},\zeta}^{\rm src}\|_{Y_{\rm prs}}
    \le
    C_{\rm Riesz}
    \|M_\zeta^{\rm act}\|
    _{L^{3/2}(I;L^{3/2}(B_1))^{3\times3}}
    +
    \|C_\eta^\zeta\|_{Y_{\rm prs}}.
\]
Together these two estimates contribute the first two lines in
\(\mathfrak S_{{\rm src},{\rm int}}^{\rm prs}\).  The clean projection term is
bounded by \Cref{lem:bounded-projection-tail-estimate}, and the harmonic term
is bounded by \Cref{lem:fixed-window-harmonic-residual-bound}.  Summing these
four bounds and then taking the infimum over
\(\zeta\in\Gamma_\Lambda^{\rm int}\) proves
\[
    \mathfrak B_{{\rm src},{\rm int}}^{\rm prs}(\mathcal D)
    \le
    \mathfrak S_{{\rm src},{\rm int}}^{\rm prs}(\mathcal D).
\]
The Calderon--Zygmund version follows by using the second estimate in
\Cref{lem:bounded-projection-tail-estimate} for the clean projection term.
\end{proof}

\begin{assumption}[Intrinsic source-size budget compatibility]
\label{ass:intrinsic-source-size-budget-compatibility}
There are constants
\[
    \eta_{{\rm src},{\rm int}}\ge0,
    \qquad
    \Delta_{{\rm src},{\rm int}}\ge0,
\]
such that every intrinsic package \(\mathcal D\) satisfies
\[
    \mathfrak S_{{\rm src},{\rm int}}^{\rm prs}(\mathcal D)
    \le
    \eta_{{\rm src},{\rm int}}
    \Dist_{\loc,{\rm int}}(\mathcal D,\Gamma_\Lambda^{\rm int})
    +
    \Delta_{{\rm src},{\rm int}}.
\]
\end{assumption}

\begin{corollary}[Conditional intrinsic pressure-source control]
\label{cor:conditional-intrinsic-pressure-source-control}
Assume \Cref{ass:intrinsic-source-size-budget-compatibility}.  Then
\[
    \Err_{{\rm src},{\rm int}}^{\rm prs}(\mathcal D)
    \le
    \eta_{{\rm src},{\rm int}}
    \Dist_{\loc,{\rm int}}(\mathcal D,\Gamma_\Lambda^{\rm int})
    +
    \Delta_{{\rm src},{\rm int}}.
\]
\end{corollary}

\begin{proof}
Combine \Cref{prop:intrinsic-pressure-source-component-budget} with
\Cref{ass:intrinsic-source-size-budget-compatibility}.
\end{proof}

\begin{remark}[Status of the assembled budget]
\Cref{prop:intrinsic-pressure-source-component-budget} is a fixed-window
assembly statement.  It proves that the pressure-source residual is bounded by
the common-gauge sum of the four already isolated components, and then by the
resolved source-size budget.  It does not prove that this budget is small,
scale-uniform, or controlled by the original intrinsic quotient distance.
\Cref{ass:intrinsic-source-size-budget-compatibility} is the exact remaining
compatibility input needed for the conditional control in
\Cref{cor:conditional-intrinsic-pressure-source-control}.
\end{remark}

\subsection{Same-gauge compatibility criterion}
\label{subsec:same-gauge-compatibility-criterion}

The previous subsection reduces intrinsic pressure-source control to the
resolved source-size budget.  We now record a sufficient condition for that
budget to satisfy the compatibility assumption.  This criterion is deliberately
same-gauge: all component estimates are required on one representative
\(\mathcal D-\zeta_\ast\).  The cutoff commutator part follows from the
finite-amplitude annular velocity control; the other components remain
explicit inputs.

\begin{assumption}[Same-gauge component compatibility datum]
\label{ass:same-gauge-component-compatibility-datum}
There are constants
\[
    M_U,C_U<\infty,\qquad
    \eta_{\rm act},\eta_{\rm proj},\eta_{\rm harm}\ge0,
    \qquad
    \Delta_{\rm act},\Delta_{\rm proj},\Delta_{\rm harm}\ge0
\]
such that for every intrinsic package \(\mathcal D\) there exists a gauge
representative \(\zeta_\ast=\zeta_\ast(\mathcal D)\in
\Gamma_\Lambda^{\rm int}\) with the following properties.  Writing
\[
    \rho_{\rm int}(\mathcal D)
    :=
    \Dist_{\loc,{\rm int}}(\mathcal D,\Gamma_\Lambda^{\rm int}),
\]
the annular velocity component satisfies
\[
    \|u_{\zeta_\ast}\|_{L^3(I;L^3(A_{3/4,1}))}\le M_U,
    \qquad
    \|u_{\zeta_\ast}\|_{L^3(I;L^3(A_{3/4,1}))}
    \le
    C_U\rho_{\rm int}(\mathcal D).
\]
The active covariance, clean projection, and harmonic tail entries satisfy
\[
    C_{\rm Riesz}
    \|M_{\zeta_\ast}^{\rm act}\|
    _{L^{3/2}(I;L^{3/2}(B_1))^{3\times3}}
    \le
    \eta_{\rm act}\rho_{\rm int}(\mathcal D)+\Delta_{\rm act},
\]
\[
    \bigl(
        1+\|P_{{\rm prs},N}^{\cl}\|_{Y_{\rm prs}\to Y_{\rm prs}}
    \bigr)
    \|R_iR_j(F_{\zeta_\ast,ij}^{\cl})\|_{Y_{\rm prs}}
    \le
    \eta_{\rm proj}\rho_{\rm int}(\mathcal D)+\Delta_{\rm proj},
\]
and
\[
    \|J_{\rm harm}\|_{Y_{\rm harm}\to Y_{\rm prs}}
    \|(I-\Pi_{\rm harm,M})p_{{\rm harm},\zeta_\ast}\|_{Y_{\rm harm}}
    \le
    \eta_{\rm harm}\rho_{\rm int}(\mathcal D)+\Delta_{\rm harm}.
\]
\end{assumption}

\begin{lemma}[Finite-amplitude same-gauge commutator compatibility]
\label{lem:finite-amplitude-same-gauge-commutator-compatibility}
Under the annular velocity part of
\Cref{ass:same-gauge-component-compatibility-datum},
\[
    2C_\eta
    \|u_{\zeta_\ast}\|_{L^3(I;L^3(A_{3/4,1}))}^2
    \le
    2C_\eta M_U C_U
    \Dist_{\loc,{\rm int}}(\mathcal D,\Gamma_\Lambda^{\rm int}).
\]
\end{lemma}

\begin{proof}
By the two annular velocity bounds in
\Cref{ass:same-gauge-component-compatibility-datum},
\[
    \|u_{\zeta_\ast}\|_{L^3(I;L^3(A_{3/4,1}))}^2
    \le
    M_U
    \|u_{\zeta_\ast}\|_{L^3(I;L^3(A_{3/4,1}))}
    \le
    M_U C_U
    \Dist_{\loc,{\rm int}}(\mathcal D,\Gamma_\Lambda^{\rm int}).
\]
Multiplying by \(2C_\eta\) proves the claim.
\end{proof}

\begin{proposition}[Same-gauge sufficient condition for source-size compatibility]
\label{prop:same-gauge-sufficient-source-size-compatibility}
Assume \Cref{ass:same-gauge-component-compatibility-datum}.  Then
\Cref{ass:intrinsic-source-size-budget-compatibility} holds with
\[
    \eta_{{\rm src},{\rm int}}
    =
    2C_\eta M_U C_U
    +
    \eta_{\rm act}
    +
    \eta_{\rm proj}
    +
    \eta_{\rm harm}
\]
and
\[
    \Delta_{{\rm src},{\rm int}}
    =
    \Delta_{\rm act}
    +
    \Delta_{\rm proj}
    +
    \Delta_{\rm harm}.
\]
\end{proposition}

\begin{proof}
By definition of \(\mathfrak S_{{\rm src},{\rm int}}^{\rm prs}\), its
infimum is bounded above by evaluating the source-size expression at the
single representative \(\zeta_\ast(\mathcal D)\).  The cutoff commutator
contribution is bounded by
\Cref{lem:finite-amplitude-same-gauge-commutator-compatibility}.  The active
covariance, projection, and harmonic contributions are bounded by the three
remaining estimates in
\Cref{ass:same-gauge-component-compatibility-datum}.  Summing these four
estimates gives
\[
    \mathfrak S_{{\rm src},{\rm int}}^{\rm prs}(\mathcal D)
    \le
    \eta_{{\rm src},{\rm int}}
    \Dist_{\loc,{\rm int}}(\mathcal D,\Gamma_\Lambda^{\rm int})
    +
    \Delta_{{\rm src},{\rm int}},
\]
with the displayed constants.
\end{proof}

\begin{corollary}[Same-gauge conditional pressure-source control]
\label{cor:same-gauge-conditional-pressure-source-control}
Under \Cref{ass:same-gauge-component-compatibility-datum},
\[
    \Err_{{\rm src},{\rm int}}^{\rm prs}(\mathcal D)
    \le
    \bigl(
        2C_\eta M_U C_U
        +
        \eta_{\rm act}
        +
        \eta_{\rm proj}
        +
        \eta_{\rm harm}
    \bigr)
    \Dist_{\loc,{\rm int}}(\mathcal D,\Gamma_\Lambda^{\rm int})
    +
    \Delta_{\rm act}
    +
    \Delta_{\rm proj}
    +
    \Delta_{\rm harm}.
\]
\end{corollary}

\begin{proof}
Combine
\Cref{prop:same-gauge-sufficient-source-size-compatibility} with
\Cref{cor:conditional-intrinsic-pressure-source-control}.
\end{proof}

\begin{remark}[What the criterion proves]
\Cref{prop:same-gauge-sufficient-source-size-compatibility} proves only a
finite-window implication.  The cutoff commutator contribution is handled by
finite-amplitude annular velocity control.  The active covariance, clean
projection, and harmonic tail entries are not proved from the intrinsic
distance here; they are named compatibility inputs.  No smallness,
scale-uniform control, decay in \(N\) or \(M\), or Navier--Stokes regularity is
claimed.
\end{remark}

\subsection{Active covariance coordinate convention}
\label{subsec:active-covariance-coordinate-convention}

We next isolate the active covariance component.  There are two logically
different choices.  One may treat the mismatch
\(\eta u_i u_j-(U_iU_j+R_{ij})\) as a residual to be estimated, or one may
choose the covariance coordinate so that it records this mismatch exactly.
For the finite-window model we first take the second, covariance-resolved
choice.  This does not prove that the convention is canonical for localized
Navier--Stokes packages; it only removes the active source term inside the
chosen model geometry.

\begin{definition}[Active covariance reconstruction residual]
\label{def:active-covariance-reconstruction-residual}
For a package \(D\), define
\[
    \mathcal R_{\rm cov}^{\rm act}(D)
    :=
    \eta u_i u_j-(U_iU_j+R_{ij})
    \in L^{3/2}(I;L^{3/2}(B_1))^{3\times3}.
\]
For a common gauge representative \(\zeta_\ast\), set
\[
    d_{\rm cov}^{\rm act}(D;\zeta_\ast)
    :=
    C_{\rm Riesz}
    \|\mathcal R_{\rm cov}^{\rm act}(D-\zeta_\ast)\|
    _{L^{3/2}(I;L^{3/2}(B_1))^{3\times3}}.
\]
\end{definition}

\begin{convention}[Covariance-resolved finite-window package]
\label{conv:covariance-resolved-finite-window-package}
In the covariance-resolved finite-window model, the Reynolds/covariance
coordinate is chosen on the same-gauge representative so that
\[
    R_{ij}
    =
    \eta u_i u_j-U_iU_j
\]
in \(L^{3/2}(I;L^{3/2}(B_1))^{3\times3}\).  Equivalently,
\[
    \mathcal R_{\rm cov}^{\rm act}(D-\zeta_\ast)=0
\]
for the representative used in the same-gauge source-size budget.
\end{convention}

\begin{lemma}[Active covariance compatibility in the resolved model]
\label{lem:active-covariance-compatibility-resolved-model}
Under \Cref{conv:covariance-resolved-finite-window-package}, the active
covariance entry in
\Cref{ass:same-gauge-component-compatibility-datum} holds with
\[
    \eta_{\rm act}=0,
    \qquad
    \Delta_{\rm act}=0.
\]
\end{lemma}

\begin{proof}
By \Cref{conv:covariance-resolved-finite-window-package},
\[
    M_{\zeta_\ast}^{\rm act}
    =
    \eta u_i u_j-(U_iU_j+R_{ij})
    =
    0
\]
for the common representative used in the source-size budget.  Therefore
\[
    C_{\rm Riesz}
    \|M_{\zeta_\ast}^{\rm act}\|
    _{L^{3/2}(I;L^{3/2}(B_1))^{3\times3}}
    =
    0,
\]
which is the active covariance compatibility bound with
\(\eta_{\rm act}=\Delta_{\rm act}=0\).
\end{proof}

\begin{remark}[Cost of resolving covariance]
The covariance-resolved convention makes the active source residual vanish in
the finite-window coordinate model.  The cost is that \(R\) is no longer an
arbitrary Reynolds coordinate; it must store the localized covariance defect
generated by \(\eta u_i u_j-U_iU_j\).  A later PDE theorem must justify that
this is the right localized package coordinate, or compare it with a more
intrinsic covariance variable.  No compactness, smallness, or scale-uniform
control is obtained from the convention alone.
\end{remark}

\subsection{Clean projection-tail compatibility}
\label{subsec:clean-projection-tail-compatibility}

The clean projection tail is the next unresolved component in the same-gauge
criterion.  There are two possible routes.  A genuine approximation route
would prove decay of the tail as \(N\to\infty\) from compactness or additional
regularity of the clean pressure source.  We do not assume such compactness
here.  Instead we record the finite-window enhanced-distance route, in which
the clean projection tail is built into the intrinsic quotient geometry.

\begin{definition}[Intrinsic clean projection-tail size]
\label{def:intrinsic-clean-projection-tail-size}
For an intrinsic package \(\mathcal D\) and a representative
\(\zeta\in\Gamma_\Lambda^{\rm int}\), define
\[
    \mathcal T_{\rm proj}(\mathcal D;\zeta)
    :=
    \bigl(
        1+\|P_{{\rm prs},N}^{\cl}\|_{Y_{\rm prs}\to Y_{\rm prs}}
    \bigr)
    \|R_iR_j(F_{\zeta,ij}^{\cl})\|_{Y_{\rm prs}}.
\]
The quotient projection-tail size is
\[
    \mathcal T_{{\rm proj},q}(\mathcal D)
    :=
    \inf_{\zeta\in\Gamma_\Lambda^{\rm int}}
    \mathcal T_{\rm proj}(\mathcal D;\zeta).
\]
\end{definition}

\begin{definition}[Projection-tail enhanced intrinsic distance]
\label{def:projection-tail-enhanced-intrinsic-distance}
Fix \(\alpha_{\rm proj}>0\).  Define
\[
\begin{aligned}
    \Dist_{\loc,{\rm int}}^{\sharp,{\rm proj}}
    (\mathcal D,\Gamma_\Lambda^{\rm int})
    :=
    \inf_{\zeta\in\Gamma_\Lambda^{\rm int}}
    \bigl(
        &\|\mathcal D-\zeta\|_{\loc,{\rm int}}\\
        &+
        \alpha_{\rm proj}
        \mathcal T_{\rm proj}(\mathcal D;\zeta)
    \bigr).
\end{aligned}
\]
\end{definition}

\begin{lemma}[Projection-tail observability in the enhanced geometry]
\label{lem:projection-tail-observability-enhanced-geometry}
For every intrinsic package \(\mathcal D\),
\[
    \mathcal T_{{\rm proj},q}(\mathcal D)
    \le
    \alpha_{\rm proj}^{-1}
    \Dist_{\loc,{\rm int}}^{\sharp,{\rm proj}}
    (\mathcal D,\Gamma_\Lambda^{\rm int}).
\]
\end{lemma}

\begin{proof}
For any \(\zeta\in\Gamma_\Lambda^{\rm int}\),
\[
    \mathcal T_{{\rm proj},q}(\mathcal D)
    \le
    \mathcal T_{\rm proj}(\mathcal D;\zeta)
    \le
    \alpha_{\rm proj}^{-1}
    \bigl(
        \|\mathcal D-\zeta\|_{\loc,{\rm int}}
        +
        \alpha_{\rm proj}
        \mathcal T_{\rm proj}(\mathcal D;\zeta)
    \bigr).
\]
Taking the infimum over \(\zeta\in\Gamma_\Lambda^{\rm int}\) proves the
claim.
\end{proof}

\begin{remark}[Projection-tail fork]
\Cref{lem:projection-tail-observability-enhanced-geometry} proves clean
projection-tail observability only after changing the quotient geometry.  It
does not prove that the original intrinsic distance
\(\Dist_{\loc,{\rm int}}\) controls the projection tail, and it does not prove
decay as \(N\to\infty\).  Such decay would require a separate compactness or
approximation theorem for the clean pressure source.
\end{remark}

\subsection{Harmonic-tail compatibility}
\label{subsec:harmonic-tail-compatibility}

The harmonic tail has the same structural status as the clean projection tail.
Boundedness of the harmonic projection and observation map makes the tail a
well-defined finite-window residual, but does not prove that it is controlled
by the original intrinsic quotient distance.  We therefore record the
enhanced-distance route and leave decay in the harmonic degree \(M\) to a
separate harmonic approximation theorem.

\begin{definition}[Intrinsic harmonic-tail size]
\label{def:intrinsic-harmonic-tail-size}
For an intrinsic package \(\mathcal D\) and a representative
\(\zeta\in\Gamma_\Lambda^{\rm int}\), define
\[
    \mathcal T_{\rm harm}(\mathcal D;\zeta)
    :=
    \|J_{\rm harm}\|_{Y_{\rm harm}\to Y_{\rm prs}}
    \|(I-\Pi_{\rm harm,M})p_{{\rm harm},\zeta}\|_{Y_{\rm harm}}.
\]
The quotient harmonic-tail size is
\[
    \mathcal T_{{\rm harm},q}(\mathcal D)
    :=
    \inf_{\zeta\in\Gamma_\Lambda^{\rm int}}
    \mathcal T_{\rm harm}(\mathcal D;\zeta).
\]
\end{definition}

\begin{definition}[Harmonic-tail enhanced intrinsic distance]
\label{def:harmonic-tail-enhanced-intrinsic-distance}
Fix \(\alpha_{\rm harm}>0\).  Define
\[
\begin{aligned}
    \Dist_{\loc,{\rm int}}^{\sharp,{\rm harm}}
    (\mathcal D,\Gamma_\Lambda^{\rm int})
    :=
    \inf_{\zeta\in\Gamma_\Lambda^{\rm int}}
    \bigl(
        &\|\mathcal D-\zeta\|_{\loc,{\rm int}}\\
        &+
        \alpha_{\rm harm}
        \mathcal T_{\rm harm}(\mathcal D;\zeta)
    \bigr).
\end{aligned}
\]
\end{definition}

\begin{lemma}[Harmonic-tail observability in the enhanced geometry]
\label{lem:harmonic-tail-observability-enhanced-geometry}
For every intrinsic package \(\mathcal D\),
\[
    \mathcal T_{{\rm harm},q}(\mathcal D)
    \le
    \alpha_{\rm harm}^{-1}
    \Dist_{\loc,{\rm int}}^{\sharp,{\rm harm}}
    (\mathcal D,\Gamma_\Lambda^{\rm int}).
\]
\end{lemma}

\begin{proof}
For any \(\zeta\in\Gamma_\Lambda^{\rm int}\),
\[
    \mathcal T_{{\rm harm},q}(\mathcal D)
    \le
    \mathcal T_{\rm harm}(\mathcal D;\zeta)
    \le
    \alpha_{\rm harm}^{-1}
    \bigl(
        \|\mathcal D-\zeta\|_{\loc,{\rm int}}
        +
        \alpha_{\rm harm}
        \mathcal T_{\rm harm}(\mathcal D;\zeta)
    \bigr).
\]
Taking the infimum over \(\zeta\in\Gamma_\Lambda^{\rm int}\) proves the
claim.
\end{proof}

\begin{remark}[Harmonic-tail fork]
\Cref{lem:harmonic-tail-observability-enhanced-geometry} proves harmonic-tail
observability only after changing the quotient geometry.  It does not prove
that the original intrinsic distance controls the harmonic tail, does not
prove decay as \(M\to\infty\), and does not justify compatibility with the
physical pressure gauge.  Those are separate theorem targets.
\end{remark}

\subsection{Combined pressure-tail enhanced geometry}
\label{subsec:combined-pressure-tail-enhanced-geometry}

The projection and harmonic tails should enter later transfer statements
through a single enhanced quotient geometry.  The following definition keeps
the two weights separate but takes one common quotient infimum.

\begin{definition}[Combined pressure-tail enhanced intrinsic distance]
\label{def:combined-pressure-tail-enhanced-intrinsic-distance}
Fix \(\alpha_{\rm proj},\alpha_{\rm harm}>0\).  Define
\[
\begin{aligned}
    \Dist_{\loc,{\rm int}}^{\sharp,{\rm tail}}
    (\mathcal D,\Gamma_\Lambda^{\rm int})
    :=
    \inf_{\zeta\in\Gamma_\Lambda^{\rm int}}
    \bigl(
        &\|\mathcal D-\zeta\|_{\loc,{\rm int}}\\
        &+
        \alpha_{\rm proj}\mathcal T_{\rm proj}(\mathcal D;\zeta)
        +
        \alpha_{\rm harm}\mathcal T_{\rm harm}(\mathcal D;\zeta)
    \bigr).
\end{aligned}
\]
\end{definition}

\begin{lemma}[Simultaneous pressure-tail observability]
\label{lem:simultaneous-pressure-tail-observability}
For every intrinsic package \(\mathcal D\),
\[
    \mathcal T_{{\rm proj},q}(\mathcal D)
    +
    \mathcal T_{{\rm harm},q}(\mathcal D)
    \le
    \max\{\alpha_{\rm proj}^{-1},\alpha_{\rm harm}^{-1}\}
    \Dist_{\loc,{\rm int}}^{\sharp,{\rm tail}}
    (\mathcal D,\Gamma_\Lambda^{\rm int}).
\]
In particular, each tail is separately controlled by the same right-hand side.
\end{lemma}

\begin{proof}
Fix \(\zeta\in\Gamma_\Lambda^{\rm int}\).  By definition of the two quotient
tail sizes,
\[
    \mathcal T_{{\rm proj},q}(\mathcal D)
    +
    \mathcal T_{{\rm harm},q}(\mathcal D)
    \le
    \mathcal T_{\rm proj}(\mathcal D;\zeta)
    +
    \mathcal T_{\rm harm}(\mathcal D;\zeta).
\]
Since
\[
    \mathcal T_{\rm proj}+\mathcal T_{\rm harm}
    \le
    \max\{\alpha_{\rm proj}^{-1},\alpha_{\rm harm}^{-1}\}
    \bigl(
        \alpha_{\rm proj}\mathcal T_{\rm proj}
        +
        \alpha_{\rm harm}\mathcal T_{\rm harm}
    \bigr),
\]
we obtain
\[
\begin{aligned}
    \mathcal T_{{\rm proj},q}(\mathcal D)
    +
    \mathcal T_{{\rm harm},q}(\mathcal D)
    \le&
    \max\{\alpha_{\rm proj}^{-1},\alpha_{\rm harm}^{-1}\}\\
    &\times
    \bigl(
        \|\mathcal D-\zeta\|_{\loc,{\rm int}}
        +
        \alpha_{\rm proj}\mathcal T_{\rm proj}(\mathcal D;\zeta)
        +
        \alpha_{\rm harm}\mathcal T_{\rm harm}(\mathcal D;\zeta)
    \bigr).
\end{aligned}
\]
Taking the infimum over \(\zeta\) proves the displayed estimate.  The
individual controls follow because both tail sizes are nonnegative.
\end{proof}

\begin{remark}[Status of the combined tail geometry]
\Cref{lem:simultaneous-pressure-tail-observability} combines the two
enhanced-distance bookkeeping steps into one finite-window quotient geometry.
It does not compare
\(\Dist_{\loc,{\rm int}}^{\sharp,{\rm tail}}\) with the original intrinsic
distance, and it does not prove decay in \(N\) or \(M\).  A later theorem must
justify this enhanced pressure-tail geometry or prove the corresponding
projection and harmonic approximation estimates.
\end{remark}

\subsection{Pressure-source transfer in the combined enhanced-tail geometry}
\label{subsec:pressure-source-transfer-combined-enhanced-tail}

We now state pressure-source control in the combined enhanced-tail geometry.
The localized defect distance is
\(\Dist_{\loc,{\rm int}}^{\sharp,{\rm tail}}\).  We do not compare it with
the original intrinsic distance, and we do not use projection-tail decay in
\(N\) or harmonic-tail decay in \(M\).

\begin{assumption}[Same-gauge enhanced-tail compatibility]
\label{ass:same-gauge-enhanced-tail-compatibility}
For every intrinsic package \(\mathcal D\), there exists a representative
\[
    \zeta_\ast=\zeta_\ast(\mathcal D)\in\Gamma_\Lambda^{\rm int}
\]
chosen in the combined enhanced-tail quotient, so that
\[
\begin{aligned}
    \Dist_{\loc,{\rm int}}^{\sharp,{\rm tail}}
    (\mathcal D,\Gamma_\Lambda^{\rm int})
    =
    &\|\mathcal D-\zeta_\ast\|_{\loc,{\rm int}}\\
    &+
    \alpha_{\rm proj}\mathcal T_{\rm proj}(\mathcal D;\zeta_\ast)
    +
    \alpha_{\rm harm}\mathcal T_{\rm harm}(\mathcal D;\zeta_\ast).
\end{aligned}
\]
Assume also that the cutoff commutator and active covariance core satisfy
\[
\begin{aligned}
    &2C_\eta
    \|u_{\zeta_\ast}\|_{L^3(I;L^3(A_{3/4,1}))}^2\\
    &\qquad+
    C_{\rm Riesz}
    \|M_{\zeta_\ast}^{\rm act}\|
    _{L^{3/2}(I;L^{3/2}(B_1))^{3\times3}}\\
    &\le
    \eta_{\rm core}
    \Dist_{\loc,{\rm int}}^{\sharp,{\rm tail}}
    (\mathcal D,\Gamma_\Lambda^{\rm int})
    +
    \Delta_{\rm core}.
\end{aligned}
\]
\end{assumption}

\begin{theorem}[Conditional pressure-source control in the enhanced-tail geometry]
\label{thm:pressure-source-control-enhanced-tail-geometry}
Assume \Cref{ass:same-gauge-enhanced-tail-compatibility}.  Then every
intrinsic package \(\mathcal D\) satisfies
\[
    \Err_{{\rm src},{\rm int}}^{\rm prs}(\mathcal D)
    \le
    \eta_{\rm src}^{\sharp,{\rm tail}}
    \Dist_{\loc,{\rm int}}^{\sharp,{\rm tail}}
    (\mathcal D,\Gamma_\Lambda^{\rm int})
    +
    \Delta_{\rm src}^{\sharp,{\rm tail}},
\]
where
\[
    \eta_{\rm src}^{\sharp,{\rm tail}}
    :=
    \eta_{\rm core}
    +
    \max\{\alpha_{\rm proj}^{-1},\alpha_{\rm harm}^{-1}\},
    \qquad
    \Delta_{\rm src}^{\sharp,{\rm tail}}
    :=
    \Delta_{\rm core}.
\]
\end{theorem}

\begin{proof}
By the component-budget estimate in
\Cref{prop:intrinsic-pressure-source-component-budget}, evaluated at the
common representative \(\zeta_\ast\),
\[
\begin{aligned}
    \Err_{{\rm src},{\rm int}}^{\rm prs}(\mathcal D)
    \le&
    2C_\eta
    \|u_{\zeta_\ast}\|_{L^3(I;L^3(A_{3/4,1}))}^2\\
    &+
    C_{\rm Riesz}
    \|M_{\zeta_\ast}^{\rm act}\|
    _{L^{3/2}(I;L^{3/2}(B_1))^{3\times3}}\\
    &+
    \mathcal T_{\rm proj}(\mathcal D;\zeta_\ast)
    +
    \mathcal T_{\rm harm}(\mathcal D;\zeta_\ast).
\end{aligned}
\]
The first two terms are bounded by
\Cref{ass:same-gauge-enhanced-tail-compatibility}.  For the two tail terms,
the defining property of \(\zeta_\ast\) gives
\[
\begin{aligned}
    &\mathcal T_{\rm proj}(\mathcal D;\zeta_\ast)
    +
    \mathcal T_{\rm harm}(\mathcal D;\zeta_\ast)\\
    &\le
    \max\{\alpha_{\rm proj}^{-1},\alpha_{\rm harm}^{-1}\}
    \Dist_{\loc,{\rm int}}^{\sharp,{\rm tail}}
    (\mathcal D,\Gamma_\Lambda^{\rm int}).
\end{aligned}
\]
This is the same enhanced-tail observability mechanism as
\Cref{lem:simultaneous-pressure-tail-observability}, now applied on the
common representative.  Combining the two estimates proves the theorem with
the displayed constants.
\end{proof}

\begin{assumption}[Enhanced-tail local-to-clean transfer comparison]
\label{ass:enhanced-tail-local-to-clean-transfer-comparison}
There are constants \(0\le\varepsilon_G<1\) and \(\delta_G\ge0\) such that
every intrinsic package \(\mathcal D\) satisfies
\[
    \Dist_{\cl}(\Theta_\Lambda\mathcal D,\Gamma_\Lambda^{\cl})
    \ge
    (1-\varepsilon_G)
    \Dist_{\loc,{\rm int}}^{\sharp,{\rm tail}}
    (\mathcal D,\Gamma_\Lambda^{\rm int})
    -
    \delta_G
\]
and
\[
    \calM_\Lambda^{\loc}(\mathcal D)
    +
    \Err_\Lambda^{\loc}(\mathcal D)
    \ge
    \calM_\Lambda^{\rm comp}(\Theta_\Lambda\mathcal D).
\]
\end{assumption}

\begin{assumption}[Enhanced-tail localized residual budget]
\label{ass:enhanced-tail-localized-residual-budget}
There are constants
\[
    \eta_\Lambda^{\sharp,{\rm tail}}\ge0,
    \qquad
    \Delta_\Lambda^{\sharp,{\rm tail}}\ge0
\]
such that every intrinsic package \(\mathcal D\) satisfies
\[
    \Err_\Lambda^{\loc}(\mathcal D)
    \le
    \eta_\Lambda^{\sharp,{\rm tail}}
    \Dist_{\loc,{\rm int}}^{\sharp,{\rm tail}}
    (\mathcal D,\Gamma_\Lambda^{\rm int})
    +
    \Delta_\Lambda^{\sharp,{\rm tail}}.
\]
\end{assumption}

\begin{theorem}[Conditional enhanced-tail localized transfer]
\label{thm:enhanced-tail-localized-transfer}
Assume \(\calQ_\Lambda^{\cl}\ne\{0\}\),
\Cref{ass:clean-detector-gauge-compatibility},
\Cref{ass:kernel-free-computational-detector},
\Cref{ass:enhanced-tail-local-to-clean-transfer-comparison}, and
\Cref{ass:enhanced-tail-localized-residual-budget}.  Then every intrinsic
package \(\mathcal D\) satisfies
\[
\begin{aligned}
    \calM_\Lambda^{\loc}(\mathcal D)
    \ge&
    \bigl(
        \mu_\Lambda^{\rm comp}(1-\varepsilon_G)
        -
        \eta_\Lambda^{\sharp,{\rm tail}}
    \bigr)
    \Dist_{\loc,{\rm int}}^{\sharp,{\rm tail}}
    (\mathcal D,\Gamma_\Lambda^{\rm int})\\
    &-
    \mu_\Lambda^{\rm comp}\delta_G
    -
    \Delta_\Lambda^{\sharp,{\rm tail}}.
\end{aligned}
\]
\end{theorem}

\begin{proof}
This is the same algebraic transfer argument as
\Cref{cor:localized-transfer-from-computational-gap}, with the localized
distance replaced by
\(\Dist_{\loc,{\rm int}}^{\sharp,{\rm tail}}\).  By
\Cref{thm:clean-computational-antiphantom},
\[
    \calM_\Lambda^{\rm comp}(\Theta_\Lambda\mathcal D)
    \ge
    \mu_\Lambda^{\rm comp}
    \Dist_{\cl}(\Theta_\Lambda\mathcal D,\Gamma_\Lambda^{\cl}).
\]
Using
\Cref{ass:enhanced-tail-local-to-clean-transfer-comparison}, we get
\[
    \calM_\Lambda^{\rm comp}(\Theta_\Lambda\mathcal D)
    \ge
    \mu_\Lambda^{\rm comp}
    \bigl(
        (1-\varepsilon_G)
        \Dist_{\loc,{\rm int}}^{\sharp,{\rm tail}}
        (\mathcal D,\Gamma_\Lambda^{\rm int})
        -
        \delta_G
    \bigr).
\]
The detector comparison in the same assumption gives
\[
    \calM_\Lambda^{\loc}(\mathcal D)
    \ge
    \mu_\Lambda^{\rm comp}(1-\varepsilon_G)
    \Dist_{\loc,{\rm int}}^{\sharp,{\rm tail}}
    (\mathcal D,\Gamma_\Lambda^{\rm int})
    -
    \mu_\Lambda^{\rm comp}\delta_G
    -
    \Err_\Lambda^{\loc}(\mathcal D).
\]
Applying \Cref{ass:enhanced-tail-localized-residual-budget} proves the
displayed estimate.
\end{proof}

\begin{remark}[Status of enhanced-tail transfer]
\Cref{thm:pressure-source-control-enhanced-tail-geometry} and
\Cref{thm:enhanced-tail-localized-transfer} prove pressure-source control and
localized transfer only in the combined enhanced-tail geometry.  They do not
claim
\[
    \Dist_{\loc,{\rm int}}^{\sharp,{\rm tail}}
    \lesssim
    \Dist_{\loc,{\rm int}},
\]
do not prove projection-tail decay as \(N\to\infty\), do not prove
harmonic-tail decay as \(M\to\infty\), and do not give scale-uniform control or
Navier--Stokes regularity.  Comparing this enhanced geometry back to the
original intrinsic geometry is a separate PDE-facing approximation theorem.
\end{remark}

\subsection{Subclaim decomposition for enhanced-tail transfer}
\label{subsec:enhanced-tail-transfer-subclaim-decomposition}

We now split the two enhanced-tail transfer assumptions into smaller
finite-window subclaims.  This does not prove the PDE estimates behind those
subclaims.  Its purpose is to identify the exact estimates that would replace
the single opaque assumptions
\Cref{ass:enhanced-tail-local-to-clean-transfer-comparison} and
\Cref{ass:enhanced-tail-localized-residual-budget}.

For an intrinsic package \(\mathcal D\) and a common representative
\(\zeta\in\Gamma_\Lambda^{\rm int}\), define the three enhanced-tail
visibility channels
\[
    X_0(\mathcal D;\zeta)
    :=
    \|\mathcal D-\zeta\|_{\loc,{\rm int}},
\]
\[
    X_{\rm proj}(\mathcal D;\zeta)
    :=
    \alpha_{\rm proj}\mathcal T_{\rm proj}(\mathcal D;\zeta),
    \qquad
    X_{\rm harm}(\mathcal D;\zeta)
    :=
    \alpha_{\rm harm}\mathcal T_{\rm harm}(\mathcal D;\zeta).
\]
Thus
\[
    \Dist_{\loc,{\rm int}}^{\sharp,{\rm tail}}
    (\mathcal D,\Gamma_\Lambda^{\rm int})
    \le
    X_0(\mathcal D;\zeta)
    +
    X_{\rm proj}(\mathcal D;\zeta)
    +
    X_{\rm harm}(\mathcal D;\zeta)
\]
for every common representative \(\zeta\).

\begin{assumption}[Three-channel enhanced-tail chart visibility]
\label{ass:three-channel-enhanced-tail-chart-visibility}
There is a same-gauge representative selection
\[
    \zeta_\ast(\mathcal D)\in\Gamma_\Lambda^{\rm int}
\]
and constants
\[
    \kappa_0,\kappa_{\rm proj},\kappa_{\rm harm}>0,
    \qquad
    \delta_0,\delta_{\rm proj},\delta_{\rm harm}\ge0,
\]
such that every intrinsic package \(\mathcal D\) satisfies
\[
    \Dist_{\cl}(\Theta_\Lambda\mathcal D,\Gamma_\Lambda^{\cl})
    \ge
    \kappa_0 X_0(\mathcal D;\zeta_\ast)-\delta_0,
\]
\[
    \Dist_{\cl}(\Theta_\Lambda\mathcal D,\Gamma_\Lambda^{\cl})
    \ge
    \kappa_{\rm proj}X_{\rm proj}(\mathcal D;\zeta_\ast)
    -\delta_{\rm proj},
\]
and
\[
    \Dist_{\cl}(\Theta_\Lambda\mathcal D,\Gamma_\Lambda^{\cl})
    \ge
    \kappa_{\rm harm}X_{\rm harm}(\mathcal D;\zeta_\ast)
    -\delta_{\rm harm}.
\]
\end{assumption}

\begin{proposition}[Enhanced-tail quotient comparison from channel visibility]
\label{prop:enhanced-tail-quotient-comparison-from-channels}
Assume
\Cref{ass:three-channel-enhanced-tail-chart-visibility}.  Set
\[
    \kappa_\ast
    :=
    \min\{\kappa_0,\kappa_{\rm proj},\kappa_{\rm harm}\},
    \qquad
    \lambda_G
    :=
    \min\left\{1,\frac{\kappa_\ast}{3}\right\},
\]
and
\[
    \delta_G^{\rm vis}
    :=
    \frac{\delta_0+\delta_{\rm proj}+\delta_{\rm harm}}{3}.
\]
Then
\[
    \Dist_{\cl}(\Theta_\Lambda\mathcal D,\Gamma_\Lambda^{\cl})
    \ge
    \lambda_G
    \Dist_{\loc,{\rm int}}^{\sharp,{\rm tail}}
    (\mathcal D,\Gamma_\Lambda^{\rm int})
    -
    \delta_G^{\rm vis}
\]
for every intrinsic package \(\mathcal D\).  Equivalently, the quotient
comparison part of
\Cref{ass:enhanced-tail-local-to-clean-transfer-comparison} holds with
\[
    1-\varepsilon_G=\lambda_G,
    \qquad
    \delta_G=\delta_G^{\rm vis}.
\]
\end{proposition}

\begin{proof}
Write
\[
    A(\mathcal D)
    :=
    \Dist_{\cl}(\Theta_\Lambda\mathcal D,\Gamma_\Lambda^{\cl}).
\]
Adding the three visibility inequalities in
\Cref{ass:three-channel-enhanced-tail-chart-visibility} gives
\[
\begin{aligned}
    3A(\mathcal D)
    \ge&
    \kappa_0X_0(\mathcal D;\zeta_\ast)
    +
    \kappa_{\rm proj}X_{\rm proj}(\mathcal D;\zeta_\ast)
    +
    \kappa_{\rm harm}X_{\rm harm}(\mathcal D;\zeta_\ast)\\
    &-
    \delta_0-\delta_{\rm proj}-\delta_{\rm harm}.
\end{aligned}
\]
Since each channel is nonnegative,
\[
\begin{aligned}
    A(\mathcal D)
    \ge&
    \frac{\kappa_\ast}{3}
    \bigl(
        X_0(\mathcal D;\zeta_\ast)
        +
        X_{\rm proj}(\mathcal D;\zeta_\ast)
        +
        X_{\rm harm}(\mathcal D;\zeta_\ast)
    \bigr)\\
    &-
    \delta_G^{\rm vis}.
\end{aligned}
\]
The sum inside the parentheses is bounded below by
\(\Dist_{\loc,{\rm int}}^{\sharp,{\rm tail}}
(\mathcal D,\Gamma_\Lambda^{\rm int})\).  Replacing
\(\kappa_\ast/3\) by the smaller coefficient \(\lambda_G\le\kappa_\ast/3\)
if necessary gives the displayed comparison.
\end{proof}

\begin{definition}[Decomposed enhanced-tail residual ledger]
\label{def:decomposed-enhanced-tail-residual-ledger}
For the enhanced-tail transfer branch, introduce nonnegative residual
functionals
\[
    \mathcal R_{\rm core},\quad
    \mathcal R_{\rm tail},\quad
    \mathcal R_{\rm chart},\quad
    \mathcal R_{\rm loc},\quad
    \mathcal R_{\rm rep},\quad
    \mathcal R_{\rm gate}.
\]
The first two are the pressure-source pieces already isolated:
\[
    \mathcal R_{\rm core}(\mathcal D)
    :=
    2C_\eta
    \|u_{\zeta_\ast}\|_{L^3(I;L^3(A_{3/4,1}))}^2
    +
    C_{\rm Riesz}
    \|M_{\zeta_\ast}^{\rm act}\|
    _{L^{3/2}(I;L^{3/2}(B_1))^{3\times3}},
\]
where \(\zeta_\ast\) is the selected same-gauge representative, and
\[
    \mathcal R_{\rm tail}(\mathcal D)
    :=
    \mathcal T_{{\rm proj},q}(\mathcal D)
    +
    \mathcal T_{{\rm harm},q}(\mathcal D).
\]
The remaining residuals denote, respectively, chart-commutation loss,
localization leakage, reproduction drift, and detector or gate-tax mismatch.
They are finite-window residual functionals; no smallness or scale-uniformity
is attached to their definition.
\end{definition}

\begin{assumption}[Enhanced-tail detector and residual sub-budgets]
\label{ass:enhanced-tail-detector-residual-subbudgets}
The localized detector comparison is mediated by the residual ledger in
\Cref{def:decomposed-enhanced-tail-residual-ledger}:
\[
\begin{aligned}
    \calM_\Lambda^{\loc}(\mathcal D)
    &+
    \mathcal R_{\rm core}(\mathcal D)
    +
    \mathcal R_{\rm tail}(\mathcal D)
    +
    \mathcal R_{\rm chart}(\mathcal D)\\
    &+
    \mathcal R_{\rm loc}(\mathcal D)
    +
    \mathcal R_{\rm rep}(\mathcal D)
    +
    \mathcal R_{\rm gate}(\mathcal D)
    \ge
    \calM_\Lambda^{\rm comp}(\Theta_\Lambda\mathcal D).
\end{aligned}
\]
In addition, for
\[
    a\in\{{\rm chart},{\rm loc},{\rm rep},{\rm gate}\},
\]
there are constants \(\eta_a,\Delta_a\ge0\) such that
\[
    \mathcal R_a(\mathcal D)
    \le
    \eta_a
    \Dist_{\loc,{\rm int}}^{\sharp,{\rm tail}}
    (\mathcal D,\Gamma_\Lambda^{\rm int})
    +
    \Delta_a.
\]
\end{assumption}

\begin{proposition}[Enhanced-tail residual budget from sub-budgets]
\label{prop:enhanced-tail-residual-budget-from-subbudgets}
Assume \Cref{ass:same-gauge-enhanced-tail-compatibility} and
\Cref{ass:enhanced-tail-detector-residual-subbudgets}.  Define
\[
\begin{aligned}
    \Err_{\Lambda,{\rm dec}}^{\loc}(\mathcal D)
    :=
    &\mathcal R_{\rm core}(\mathcal D)
    +
    \mathcal R_{\rm tail}(\mathcal D)
    +
    \mathcal R_{\rm chart}(\mathcal D)\\
    &+
    \mathcal R_{\rm loc}(\mathcal D)
    +
    \mathcal R_{\rm rep}(\mathcal D)
    +
    \mathcal R_{\rm gate}(\mathcal D).
\end{aligned}
\]
Then
\[
    \calM_\Lambda^{\loc}(\mathcal D)
    +
    \Err_{\Lambda,{\rm dec}}^{\loc}(\mathcal D)
    \ge
    \calM_\Lambda^{\rm comp}(\Theta_\Lambda\mathcal D)
\]
and
\[
    \Err_{\Lambda,{\rm dec}}^{\loc}(\mathcal D)
    \le
    \eta_{\Lambda,{\rm dec}}^{\sharp,{\rm tail}}
    \Dist_{\loc,{\rm int}}^{\sharp,{\rm tail}}
    (\mathcal D,\Gamma_\Lambda^{\rm int})
    +
    \Delta_{\Lambda,{\rm dec}}^{\sharp,{\rm tail}},
\]
where
\[
    \eta_{\Lambda,{\rm dec}}^{\sharp,{\rm tail}}
    =
    \eta_{\rm core}
    +
    \max\{\alpha_{\rm proj}^{-1},\alpha_{\rm harm}^{-1}\}
    +
    \eta_{\rm chart}
    +
    \eta_{\rm loc}
    +
    \eta_{\rm rep}
    +
    \eta_{\rm gate}
\]
and
\[
    \Delta_{\Lambda,{\rm dec}}^{\sharp,{\rm tail}}
    =
    \Delta_{\rm core}
    +
    \Delta_{\rm chart}
    +
    \Delta_{\rm loc}
    +
    \Delta_{\rm rep}
    +
    \Delta_{\rm gate}.
\]
\end{proposition}

\begin{proof}
The detector comparison is exactly the first assertion in
\Cref{ass:enhanced-tail-detector-residual-subbudgets} after substituting the
definition of \(\Err_{\Lambda,{\rm dec}}^{\loc}\).

The core residual is bounded by
\Cref{ass:same-gauge-enhanced-tail-compatibility}.  The tail residual is
bounded by \Cref{lem:simultaneous-pressure-tail-observability}:
\[
    \mathcal R_{\rm tail}(\mathcal D)
    \le
    \max\{\alpha_{\rm proj}^{-1},\alpha_{\rm harm}^{-1}\}
    \Dist_{\loc,{\rm int}}^{\sharp,{\rm tail}}
    (\mathcal D,\Gamma_\Lambda^{\rm int}).
\]
The chart, localization, reproduction, and gate residuals are bounded by the
sub-budget estimates in
\Cref{ass:enhanced-tail-detector-residual-subbudgets}.  Summing these six
estimates gives the displayed constants.
\end{proof}

\begin{corollary}[Subclaim criterion for enhanced-tail localized transfer]
\label{cor:subclaim-criterion-enhanced-tail-localized-transfer}
Assume
\Cref{ass:three-channel-enhanced-tail-chart-visibility},
\Cref{ass:same-gauge-enhanced-tail-compatibility},
\Cref{ass:enhanced-tail-detector-residual-subbudgets},
\(\calQ_\Lambda^{\cl}\ne\{0\}\),
\Cref{ass:clean-detector-gauge-compatibility}, and
\Cref{ass:kernel-free-computational-detector}.  Then the enhanced-tail
localized transfer estimate holds with
\[
    1-\varepsilon_G
    =
    \lambda_G,
    \qquad
    \delta_G=\delta_G^{\rm vis},
\]
and with
\[
    \eta_\Lambda^{\sharp,{\rm tail}}
    =
    \eta_{\Lambda,{\rm dec}}^{\sharp,{\rm tail}},
    \qquad
    \Delta_\Lambda^{\sharp,{\rm tail}}
    =
    \Delta_{\Lambda,{\rm dec}}^{\sharp,{\rm tail}}.
\]
\end{corollary}

\begin{proof}
\Cref{prop:enhanced-tail-quotient-comparison-from-channels} gives the
quotient-distance part of the enhanced-tail local-to-clean comparison.
\Cref{prop:enhanced-tail-residual-budget-from-subbudgets} supplies both the
detector comparison and the residual budget, with the decomposed residual
\(\Err_{\Lambda,{\rm dec}}^{\loc}\).  Applying
\Cref{thm:enhanced-tail-localized-transfer} with these constants proves the
claim.
\end{proof}

\begin{remark}[Status of the subclaim decomposition]
\Cref{cor:subclaim-criterion-enhanced-tail-localized-transfer} is a
finite-window reduction of the enhanced-tail transfer assumptions to smaller
subclaims.  It does not prove the three chart-visibility estimates, the
chart-commutation residual bound, the localization leakage bound, the
reproduction drift bound, or the gate-tax mismatch bound.  It also does not
compare
\(\Dist_{\loc,{\rm int}}^{\sharp,{\rm tail}}\) with
\(\Dist_{\loc,{\rm int}}\), prove decay in \(N\) or \(M\), or claim
scale-uniformity or Navier--Stokes regularity.
\end{remark}

\subsection{Intrinsic-core chart visibility from a quotient left inverse}
\label{subsec:intrinsic-core-chart-visibility-left-inverse}

We next isolate the first chart-visibility channel in
\Cref{ass:three-channel-enhanced-tail-chart-visibility}.  The result below is
purely finite-dimensional.  It says that intrinsic-core visibility follows if
the clean quotient chart admits a bounded left inverse up to an explicit
additive reconstruction defect and if the chosen same-gauge representative is
not much larger than the intrinsic quotient distance.

\begin{definition}[Intrinsic and clean quotient classes]
\label{def:intrinsic-clean-quotient-classes}
Write
\[
    [\mathcal D]_{\rm int}
    \in
    \mathfrak I_\Lambda^{\loc}/\Gamma_\Lambda^{\rm int},
    \qquad
    [\Theta_\Lambda\mathcal D]_{\cl}
    \in
    \calK_\Lambda^{\cl}/\Gamma_\Lambda^{\cl}
\]
for the intrinsic and clean quotient classes.  Their quotient norms are
\[
    \|[\mathcal D]_{\rm int}\|_{\loc,{\rm int}/\Gamma}
    =
    \Dist_{\loc,{\rm int}}(\mathcal D,\Gamma_\Lambda^{\rm int})
\]
and
\[
    \|[\Theta_\Lambda\mathcal D]_{\cl}\|_{\cl/\Gamma}
    =
    \Dist_{\cl}(\Theta_\Lambda\mathcal D,\Gamma_\Lambda^{\cl}).
\]
\end{definition}

\begin{assumption}[Bounded intrinsic-core quotient left inverse]
\label{ass:bounded-intrinsic-core-quotient-left-inverse}
There is a linear map
\[
    L_{\rm core}:
    \calK_\Lambda^{\cl}/\Gamma_\Lambda^{\cl}
    \to
    \mathfrak I_\Lambda^{\loc}/\Gamma_\Lambda^{\rm int}
\]
and constants \(C_{\rm core}<\infty\) and
\(\delta_{\rm rec}\ge0\) such that
\[
    \|L_{\rm core}q\|_{\loc,{\rm int}/\Gamma}
    \le
    C_{\rm core}\|q\|_{\cl/\Gamma}
\]
for every clean quotient class \(q\), and
\[
    \|[\mathcal D]_{\rm int}
    -
    L_{\rm core}[\Theta_\Lambda\mathcal D]_{\cl}\|
    _{\loc,{\rm int}/\Gamma}
    \le
    \delta_{\rm rec}
\]
for every intrinsic package \(\mathcal D\).
\end{assumption}

\begin{assumption}[Controlled same-gauge core representative]
\label{ass:controlled-same-gauge-core-representative}
The same-gauge representative selection \(\zeta_\ast(\mathcal D)\) used in
\Cref{ass:three-channel-enhanced-tail-chart-visibility} satisfies
\[
    X_0(\mathcal D;\zeta_\ast)
    \le
    C_{\rm sel}
    \Dist_{\loc,{\rm int}}(\mathcal D,\Gamma_\Lambda^{\rm int})
    +
    \delta_{\rm sel}
\]
for constants \(C_{\rm sel}<\infty\) and \(\delta_{\rm sel}\ge0\).
\end{assumption}

\begin{lemma}[Intrinsic-core chart visibility from a quotient left inverse]
\label{lem:intrinsic-core-chart-visibility-left-inverse}
Assume \Cref{ass:bounded-intrinsic-core-quotient-left-inverse} and
\Cref{ass:controlled-same-gauge-core-representative}.  Then the intrinsic-core
visibility channel
\[
    \Dist_{\cl}(\Theta_\Lambda\mathcal D,\Gamma_\Lambda^{\cl})
    \ge
    \kappa_0X_0(\mathcal D;\zeta_\ast)-\delta_0
\]
holds with
\[
    \kappa_0
    =
    \frac{1}{C_{\rm core}C_{\rm sel}},
    \qquad
    \delta_0
    =
    \frac{\delta_{\rm sel}}{C_{\rm core}C_{\rm sel}}
    +
    \frac{\delta_{\rm rec}}{C_{\rm core}}.
\]
\end{lemma}

\begin{proof}
By the reconstruction property and the boundedness of \(L_{\rm core}\),
\[
\begin{aligned}
    \Dist_{\loc,{\rm int}}(\mathcal D,\Gamma_\Lambda^{\rm int})
    &=
    \|[\mathcal D]_{\rm int}\|_{\loc,{\rm int}/\Gamma}\\
    &\le
    \|L_{\rm core}[\Theta_\Lambda\mathcal D]_{\cl}\|
    _{\loc,{\rm int}/\Gamma}
    +
    \delta_{\rm rec}\\
    &\le
    C_{\rm core}
    \Dist_{\cl}(\Theta_\Lambda\mathcal D,\Gamma_\Lambda^{\cl})
    +
    \delta_{\rm rec}.
\end{aligned}
\]
Therefore
\[
    \Dist_{\cl}(\Theta_\Lambda\mathcal D,\Gamma_\Lambda^{\cl})
    \ge
    \frac{1}{C_{\rm core}}
    \Dist_{\loc,{\rm int}}(\mathcal D,\Gamma_\Lambda^{\rm int})
    -
    \frac{\delta_{\rm rec}}{C_{\rm core}}.
\]
The controlled-representative assumption gives
\[
    \Dist_{\loc,{\rm int}}(\mathcal D,\Gamma_\Lambda^{\rm int})
    \ge
    \frac{1}{C_{\rm sel}}X_0(\mathcal D;\zeta_\ast)
    -
    \frac{\delta_{\rm sel}}{C_{\rm sel}}.
\]
Substituting this lower bound into the previous inequality yields the stated
constants.
\end{proof}

\begin{remark}[Status of intrinsic-core visibility]
The preceding lemma proves only a finite-dimensional sufficient condition for
the intrinsic-core visibility channel.  The existence of \(L_{\rm core}\), the
size of the reconstruction defect \(\delta_{\rm rec}\), and the controlled
same-gauge representative property are not derived from Navier--Stokes data
here.  The lemma does not address projection-tail visibility or harmonic-tail
visibility.
\end{remark}

\subsection{Pressure-tail chart visibility from clean tail recovery}
\label{subsec:pressure-tail-chart-visibility-recovery}

We now record the analogous finite-window mechanism for the two pressure-tail
visibility channels.  The point is not that clean quotient distance must see
these tails automatically.  Rather, if the clean quotient carries bounded
recovery maps for the selected projection and harmonic tail data, then the
remaining two visibility inequalities in
\Cref{ass:three-channel-enhanced-tail-chart-visibility} follow.

\begin{definition}[Projection and harmonic tail data]
\label{def:projection-harmonic-tail-data}
Let \(\mathcal Y_{\rm proj}^{\rm tail}\) and
\(\mathcal Y_{\rm harm}^{\rm tail}\) be finite-window normed spaces.  For an
intrinsic package \(\mathcal D\) and a common representative
\(\zeta\in\Gamma_\Lambda^{\rm int}\), let
\[
    \tau_{\rm proj}(\mathcal D;\zeta)
    \in
    \mathcal Y_{\rm proj}^{\rm tail},
    \qquad
    \tau_{\rm harm}(\mathcal D;\zeta)
    \in
    \mathcal Y_{\rm harm}^{\rm tail}
\]
be tail data satisfying
\[
    \|\tau_{\rm proj}(\mathcal D;\zeta)\|
    _{\mathcal Y_{\rm proj}^{\rm tail}}
    =
    \mathcal T_{\rm proj}(\mathcal D;\zeta),
    \qquad
    \|\tau_{\rm harm}(\mathcal D;\zeta)\|
    _{\mathcal Y_{\rm harm}^{\rm tail}}
    =
    \mathcal T_{\rm harm}(\mathcal D;\zeta).
\]
\end{definition}

\begin{assumption}[Clean quotient recovery of pressure tails]
\label{ass:clean-quotient-pressure-tail-recovery}
There are linear maps
\[
    L_{\rm proj}^{\rm tail}:
    \calK_\Lambda^{\cl}/\Gamma_\Lambda^{\cl}
    \to
    \mathcal Y_{\rm proj}^{\rm tail},
    \qquad
    L_{\rm harm}^{\rm tail}:
    \calK_\Lambda^{\cl}/\Gamma_\Lambda^{\cl}
    \to
    \mathcal Y_{\rm harm}^{\rm tail},
\]
constants \(C_{\rm proj},C_{\rm harm}<\infty\), and reconstruction defects
\(\delta_{\rm proj}^{\rm rec},\delta_{\rm harm}^{\rm rec}\ge0\), such that
\[
    \|L_{\rm proj}^{\rm tail}q\|_{\mathcal Y_{\rm proj}^{\rm tail}}
    \le
    C_{\rm proj}\|q\|_{\cl/\Gamma},
    \qquad
    \|L_{\rm harm}^{\rm tail}q\|_{\mathcal Y_{\rm harm}^{\rm tail}}
    \le
    C_{\rm harm}\|q\|_{\cl/\Gamma}
\]
for every clean quotient class \(q\).  For the selected same-gauge
representative \(\zeta_\ast(\mathcal D)\), assume
\[
    \left\|
        \tau_{\rm proj}(\mathcal D;\zeta_\ast)
        -
        L_{\rm proj}^{\rm tail}
        [\Theta_\Lambda\mathcal D]_{\cl}
    \right\|_{\mathcal Y_{\rm proj}^{\rm tail}}
    \le
    \delta_{\rm proj}^{\rm rec}
\]
and
\[
    \left\|
        \tau_{\rm harm}(\mathcal D;\zeta_\ast)
        -
        L_{\rm harm}^{\rm tail}
        [\Theta_\Lambda\mathcal D]_{\cl}
    \right\|_{\mathcal Y_{\rm harm}^{\rm tail}}
    \le
    \delta_{\rm harm}^{\rm rec}.
\]
\end{assumption}

\begin{lemma}[Pressure-tail chart visibility from recovery]
\label{lem:pressure-tail-chart-visibility-recovery}
Assume \Cref{ass:clean-quotient-pressure-tail-recovery}.  Then
\[
    \Dist_{\cl}(\Theta_\Lambda\mathcal D,\Gamma_\Lambda^{\cl})
    \ge
    \kappa_{\rm proj}
    X_{\rm proj}(\mathcal D;\zeta_\ast)
    -
    \delta_{\rm proj}
\]
with
\[
    \kappa_{\rm proj}
    =
    \frac{1}{\alpha_{\rm proj}C_{\rm proj}},
    \qquad
    \delta_{\rm proj}
    =
    \frac{\delta_{\rm proj}^{\rm rec}}{C_{\rm proj}},
\]
and
\[
    \Dist_{\cl}(\Theta_\Lambda\mathcal D,\Gamma_\Lambda^{\cl})
    \ge
    \kappa_{\rm harm}
    X_{\rm harm}(\mathcal D;\zeta_\ast)
    -
    \delta_{\rm harm}
\]
with
\[
    \kappa_{\rm harm}
    =
    \frac{1}{\alpha_{\rm harm}C_{\rm harm}},
    \qquad
    \delta_{\rm harm}
    =
    \frac{\delta_{\rm harm}^{\rm rec}}{C_{\rm harm}}.
\]
\end{lemma}

\begin{proof}
We prove the projection-tail estimate; the harmonic-tail estimate is
identical.  By the recovery defect bound and boundedness of
\(L_{\rm proj}^{\rm tail}\),
\[
\begin{aligned}
    \mathcal T_{\rm proj}(\mathcal D;\zeta_\ast)
    &=
    \|\tau_{\rm proj}(\mathcal D;\zeta_\ast)\|
    _{\mathcal Y_{\rm proj}^{\rm tail}}\\
    &\le
    \|L_{\rm proj}^{\rm tail}
    [\Theta_\Lambda\mathcal D]_{\cl}\|
    _{\mathcal Y_{\rm proj}^{\rm tail}}
    +
    \delta_{\rm proj}^{\rm rec}\\
    &\le
    C_{\rm proj}
    \Dist_{\cl}(\Theta_\Lambda\mathcal D,\Gamma_\Lambda^{\cl})
    +
    \delta_{\rm proj}^{\rm rec}.
\end{aligned}
\]
Since \(X_{\rm proj}=\alpha_{\rm proj}\mathcal T_{\rm proj}\), this gives
\[
    \Dist_{\cl}(\Theta_\Lambda\mathcal D,\Gamma_\Lambda^{\cl})
    \ge
    \frac{1}{\alpha_{\rm proj}C_{\rm proj}}
    X_{\rm proj}(\mathcal D;\zeta_\ast)
    -
    \frac{\delta_{\rm proj}^{\rm rec}}{C_{\rm proj}}.
\]
The same argument with \(L_{\rm harm}^{\rm tail}\) proves the harmonic
estimate.
\end{proof}

\begin{corollary}[Finite-window sufficient condition for three-channel visibility]
\label{cor:finite-window-three-channel-visibility-sufficient-condition}
Assume
\Cref{ass:bounded-intrinsic-core-quotient-left-inverse},
\Cref{ass:controlled-same-gauge-core-representative}, and
\Cref{ass:clean-quotient-pressure-tail-recovery}.  Then
\Cref{ass:three-channel-enhanced-tail-chart-visibility} holds with the
constants from
\Cref{lem:intrinsic-core-chart-visibility-left-inverse} and
\Cref{lem:pressure-tail-chart-visibility-recovery}.
\end{corollary}

\begin{proof}
The intrinsic-core visibility inequality follows from
\Cref{lem:intrinsic-core-chart-visibility-left-inverse}.  The projection-tail
and harmonic-tail visibility inequalities follow from
\Cref{lem:pressure-tail-chart-visibility-recovery}.
\end{proof}

\begin{remark}[Status of pressure-tail recovery]
\Cref{lem:pressure-tail-chart-visibility-recovery} is a finite-window
observability statement conditional on the clean quotient carrying tail
recovery maps.  It does not prove projection-tail decay as \(N\to\infty\), does
not prove harmonic-tail decay as \(M\to\infty\), and does not compare the
enhanced-tail distance with the original intrinsic distance.
\end{remark}

\subsection{Chart residual sub-budget in the enhanced-tail geometry}
\label{subsec:chart-residual-subbudget-enhanced-tail}

We now replace the placeholder chart residual
\(\mathcal R_{\rm chart}\) in
\Cref{def:decomposed-enhanced-tail-residual-ledger} by explicit
finite-window chart mismatch components.  This subsection is a chart
compatibility module.  It does not introduce a new Navier--Stokes pressure
estimate.

\begin{assumption}[Enhanced-tail same-gauge selector]
\label{ass:enhanced-tail-same-gauge-selector}
For every intrinsic package \(\mathcal D\), fix a same-gauge representative
\[
    \zeta_\ast(\mathcal D)\in\Gamma_\Lambda^{\rm int}
\]
and write
\[
    \mathcal D_\ast:=\mathcal D-\zeta_\ast(\mathcal D).
\]
There is a selector defect \(\delta_{\rm sel}^{\sharp}\ge0\) such that
\[
\begin{aligned}
    &\|\mathcal D_\ast\|_{\loc,{\rm int}}
    +
    \alpha_{\rm proj}
    \mathcal T_{\rm proj}(\mathcal D;\zeta_\ast)
    +
    \alpha_{\rm harm}
    \mathcal T_{\rm harm}(\mathcal D;\zeta_\ast)\\
    &\le
    \Dist_{\loc,{\rm int}}^{\sharp,{\rm tail}}
    (\mathcal D,\Gamma_\Lambda^{\rm int})
    +
    \delta_{\rm sel}^{\sharp}.
\end{aligned}
\]
In the exact minimizer case one may take
\(\delta_{\rm sel}^{\sharp}=0\).
\end{assumption}

\begin{definition}[Finite-window chart residual operators]
\label{def:finite-window-chart-residual-operators}
Let
\[
    Y_{\rm chart}^{O},\qquad
    Y_{\rm chart}^{\rm Rep},\qquad
    Y_{\rm chart}^{\rm Tax},\qquad
    Y_{\rm chart}^{\rm Tail}
\]
be finite-dimensional normed spaces.  Fix bounded finite-window operators
\[
    K_{\rm chart}^{O}:
    \mathfrak I_\Lambda^{\loc}\to Y_{\rm chart}^{O},
    \qquad
    K_{\rm chart}^{\rm Rep}:
    \mathfrak I_\Lambda^{\loc}\to Y_{\rm chart}^{\rm Rep},
\]
\[
    K_{\rm chart}^{\rm Tax}:
    \mathfrak I_\Lambda^{\loc}\to Y_{\rm chart}^{\rm Tax},
    \qquad
    K_{\rm chart}^{\rm Tail}:
    \mathfrak I_\Lambda^{\loc}\to Y_{\rm chart}^{\rm Tail}.
\]
The operator \(K_{\rm chart}^{O}\) records observation-coordinate chart
mismatch, namely the mismatch between localized observation channels and
clean observation channels after applying the local-to-clean chart.  The
operator \(K_{\rm chart}^{\rm Rep}\) records reproduction-coordinate chart
mismatch between localized reproduction residuals and clean adjacent-scale
reproduction residuals after charting.  The operator
\(K_{\rm chart}^{\rm Tax}\) records tax or gate-coordinate chart mismatch.
The operator \(K_{\rm chart}^{\rm Tail}\) records pressure-tail chart
mismatch between projection/harmonic tail data measured in the enhanced-tail
intrinsic geometry and the corresponding clean tail coordinates recovered in
the clean quotient.

These are finite-window model-level residual operators.  Their boundedness is
part of the chart datum, not a PDE estimate.
\end{definition}

\begin{definition}[Component chart residuals]
\label{def:component-chart-residuals}
For an intrinsic package \(\mathcal D\), define
\[
    \mathcal R_{\rm chart}^{O}(\mathcal D)
    :=
    \|K_{\rm chart}^{O}\mathcal D_\ast\|_{Y_{\rm chart}^{O}},
\]
\[
    \mathcal R_{\rm chart}^{\rm Rep}(\mathcal D)
    :=
    \|K_{\rm chart}^{\rm Rep}\mathcal D_\ast\|
    _{Y_{\rm chart}^{\rm Rep}},
\]
\[
    \mathcal R_{\rm chart}^{\rm Tax}(\mathcal D)
    :=
    \|K_{\rm chart}^{\rm Tax}\mathcal D_\ast\|
    _{Y_{\rm chart}^{\rm Tax}},
\]
and
\[
    \mathcal R_{\rm chart}^{\rm Tail}(\mathcal D)
    :=
    \|K_{\rm chart}^{\rm Tail}\mathcal D_\ast\|
    _{Y_{\rm chart}^{\rm Tail}}.
\]
The total chart residual is
\[
    \mathcal R_{\rm chart}(\mathcal D)
    :=
    \mathcal R_{\rm chart}^{O}(\mathcal D)
    +
    \mathcal R_{\rm chart}^{\rm Rep}(\mathcal D)
    +
    \mathcal R_{\rm chart}^{\rm Tax}(\mathcal D)
    +
    \mathcal R_{\rm chart}^{\rm Tail}(\mathcal D).
\]
\end{definition}

\begin{lemma}[Componentwise chart residual bounds]
\label{lem:componentwise-chart-residual-bounds}
Set
\[
    C_{\rm chart}^{O}
    :=
    \|K_{\rm chart}^{O}\|_{\mathfrak I_\Lambda^{\loc}
    \to Y_{\rm chart}^{O}},
    \qquad
    C_{\rm chart}^{\rm Rep}
    :=
    \|K_{\rm chart}^{\rm Rep}\|_{\mathfrak I_\Lambda^{\loc}
    \to Y_{\rm chart}^{\rm Rep}},
\]
\[
    C_{\rm chart}^{\rm Tax}
    :=
    \|K_{\rm chart}^{\rm Tax}\|_{\mathfrak I_\Lambda^{\loc}
    \to Y_{\rm chart}^{\rm Tax}},
    \qquad
    C_{\rm chart}^{\rm Tail}
    :=
    \|K_{\rm chart}^{\rm Tail}\|_{\mathfrak I_\Lambda^{\loc}
    \to Y_{\rm chart}^{\rm Tail}}.
\]
Then
\[
    \mathcal R_{\rm chart}^{O}(\mathcal D)
    \le
    C_{\rm chart}^{O}\|\mathcal D_\ast\|_{\loc,{\rm int}},
\]
\[
    \mathcal R_{\rm chart}^{\rm Rep}(\mathcal D)
    \le
    C_{\rm chart}^{\rm Rep}\|\mathcal D_\ast\|_{\loc,{\rm int}},
\]
\[
    \mathcal R_{\rm chart}^{\rm Tax}(\mathcal D)
    \le
    C_{\rm chart}^{\rm Tax}\|\mathcal D_\ast\|_{\loc,{\rm int}},
\]
and
\[
    \mathcal R_{\rm chart}^{\rm Tail}(\mathcal D)
    \le
    C_{\rm chart}^{\rm Tail}\|\mathcal D_\ast\|_{\loc,{\rm int}}.
\]
\end{lemma}

\begin{proof}
Each estimate is the operator-norm bound for the corresponding finite-window
chart residual operator, evaluated at \(\mathcal D_\ast\).
\end{proof}

\begin{lemma}[Enhanced-tail selector bound for the core norm]
\label{lem:enhanced-tail-selector-core-bound}
Under \Cref{ass:enhanced-tail-same-gauge-selector},
\[
    \|\mathcal D_\ast\|_{\loc,{\rm int}}
    \le
    \Dist_{\loc,{\rm int}}^{\sharp,{\rm tail}}
    (\mathcal D,\Gamma_\Lambda^{\rm int})
    +
    \delta_{\rm sel}^{\sharp}.
\]
\end{lemma}

\begin{proof}
The defining inequality in
\Cref{ass:enhanced-tail-same-gauge-selector} is a sum of three nonnegative
terms on the left-hand side.  Dropping the two tail terms gives the stated
bound.
\end{proof}

\begin{proposition}[Chart residual sub-budget]
\label{prop:chart-residual-subbudget-enhanced-tail}
Assume the enhanced-tail same-gauge selector and the chart operators above.
Set
\[
    C_{\rm chart}
    :=
    C_{\rm chart}^{O}
    +
    C_{\rm chart}^{\rm Rep}
    +
    C_{\rm chart}^{\rm Tax}
    +
    C_{\rm chart}^{\rm Tail}.
\]
Then
\[
    \mathcal R_{\rm chart}(\mathcal D)
    \le
    C_{\rm chart}
    \Dist_{\loc,{\rm int}}^{\sharp,{\rm tail}}
    (\mathcal D,\Gamma_\Lambda^{\rm int})
    +
    C_{\rm chart}\delta_{\rm sel}^{\sharp}.
\]
Equivalently, the chart residual satisfies
\[
    \mathcal R_{\rm chart}(\mathcal D)
    \le
    \eta_{\rm chart}
    \Dist_{\loc,{\rm int}}^{\sharp,{\rm tail}}
    (\mathcal D,\Gamma_\Lambda^{\rm int})
    +
    \Delta_{\rm chart},
\]
with
\[
    \eta_{\rm chart}=C_{\rm chart},
    \qquad
    \Delta_{\rm chart}=C_{\rm chart}\delta_{\rm sel}^{\sharp}.
\]
In the exact minimizer case, \(\delta_{\rm sel}^{\sharp}=0\), and hence
\(\Delta_{\rm chart}=0\).
\end{proposition}

\begin{proof}
Summing the four estimates in
\Cref{lem:componentwise-chart-residual-bounds} gives
\[
    \mathcal R_{\rm chart}(\mathcal D)
    \le
    C_{\rm chart}\|\mathcal D_\ast\|_{\loc,{\rm int}}.
\]
Apply \Cref{lem:enhanced-tail-selector-core-bound}.
\end{proof}

\begin{proposition}[Affine chart residual variant]
\label{prop:affine-chart-residual-variant}
Suppose that for each
\[
    a\in\{O,{\rm Rep},{\rm Tax},{\rm Tail}\}
\]
one has an affine finite-window estimate
\[
    \mathcal R_{\rm chart}^{a}(\mathcal D)
    \le
    C_{\rm chart}^{a}\|\mathcal D_\ast\|_{\loc,{\rm int}}
    +
    \Delta_{\rm chart}^{a}
\]
with \(C_{\rm chart}^{a},\Delta_{\rm chart}^{a}\ge0\).  Then
\[
    \mathcal R_{\rm chart}(\mathcal D)
    \le
    \eta_{\rm chart}
    \Dist_{\loc,{\rm int}}^{\sharp,{\rm tail}}
    (\mathcal D,\Gamma_\Lambda^{\rm int})
    +
    \Delta_{\rm chart},
\]
where
\[
    \eta_{\rm chart}
    =
    C_{\rm chart}^{O}
    +
    C_{\rm chart}^{\rm Rep}
    +
    C_{\rm chart}^{\rm Tax}
    +
    C_{\rm chart}^{\rm Tail},
\]
and
\[
    \Delta_{\rm chart}
    =
    \eta_{\rm chart}\delta_{\rm sel}^{\sharp}
    +
    \Delta_{\rm chart}^{O}
    +
    \Delta_{\rm chart}^{\rm Rep}
    +
    \Delta_{\rm chart}^{\rm Tax}
    +
    \Delta_{\rm chart}^{\rm Tail}.
\]
\end{proposition}

\begin{proof}
Sum the four affine component estimates and then use
\Cref{lem:enhanced-tail-selector-core-bound}.
\end{proof}

\begin{corollary}[Chart contribution to the enhanced-tail residual ledger]
\label{cor:chart-contribution-enhanced-tail-ledger}
In the decomposed enhanced-tail residual budget from
\Cref{prop:enhanced-tail-residual-budget-from-subbudgets}, the chart
contribution may be taken to be
\[
    \eta_{\rm chart}
    =
    C_{\rm chart}^{O}
    +
    C_{\rm chart}^{\rm Rep}
    +
    C_{\rm chart}^{\rm Tax}
    +
    C_{\rm chart}^{\rm Tail}
\]
and
\[
    \Delta_{\rm chart}
    =
    \eta_{\rm chart}\delta_{\rm sel}^{\sharp}
    +
    \Delta_{\rm chart}^{O}
    +
    \Delta_{\rm chart}^{\rm Rep}
    +
    \Delta_{\rm chart}^{\rm Tax}
    +
    \Delta_{\rm chart}^{\rm Tail},
\]
with all \(\Delta_{\rm chart}^{a}=0\) in the purely linear operator-norm
case.  Consequently
\[
\begin{aligned}
    \eta_{\Lambda,{\rm dec}}^{\sharp,{\rm tail}}
    =
    &\eta_{\rm core}
    +
    \max\{\alpha_{\rm proj}^{-1},\alpha_{\rm harm}^{-1}\}
    +
    \eta_{\rm chart}\\
    &+
    \eta_{\rm loc}
    +
    \eta_{\rm rep}
    +
    \eta_{\rm gate},
\end{aligned}
\]
and
\[
    \Delta_{\Lambda,{\rm dec}}^{\sharp,{\rm tail}}
    =
    \Delta_{\rm core}
    +
    \Delta_{\rm chart}
    +
    \Delta_{\rm loc}
    +
    \Delta_{\rm rep}
    +
    \Delta_{\rm gate}.
\]
\end{corollary}

\begin{proof}
This is exactly
\Cref{prop:enhanced-tail-residual-budget-from-subbudgets} with the chart
sub-budget supplied by
\Cref{prop:chart-residual-subbudget-enhanced-tail} or by the affine variant
\Cref{prop:affine-chart-residual-variant}.
\end{proof}

\begin{remark}[Status of the chart residual sub-budget]
This subsection proves only a finite-window chart residual sub-budget.  Once
the chart mismatch operators and same-gauge representative selection are
fixed, the chart residual is controlled by the enhanced-tail quotient
geometry.  The result does not prove local-to-clean quotient lifting, the
chart-visibility inequalities, scale-uniformity, smallness, or
Navier--Stokes regularity.  It also does not compare
\[
    \Dist_{\loc,{\rm int}}^{\sharp,{\rm tail}}
\]
with the original intrinsic distance
\[
    \Dist_{\loc,{\rm int}}.
\]
That comparison remains a separate approximation or geometry theorem.
\end{remark}

\subsection{Localization leakage sub-budget in the enhanced-tail geometry}
\label{subsec:localization-leakage-subbudget-enhanced-tail}

We next replace the placeholder localization leakage residual
\(\mathcal R_{\rm loc}\) in
\Cref{def:decomposed-enhanced-tail-residual-ledger}.  The convention in this
finite-window bookkeeping step is to separate energy, flux, pressure, and
momentum leakage.  These are localization-channel residuals, not new
Navier--Stokes estimates.

\begin{definition}[Finite-window localization leakage operators]
\label{def:finite-window-localization-leakage-operators}
Let
\[
    Y_{\loc}^{\rm en},\qquad
    Y_{\loc}^{\rm flux},\qquad
    Y_{\loc}^{\rm prs},\qquad
    Y_{\loc}^{\rm mom}
\]
be finite-dimensional normed spaces.  Fix bounded finite-window operators
\[
    K_{\loc}^{\rm en}:
    \mathfrak I_\Lambda^{\loc}\to Y_{\loc}^{\rm en},
    \qquad
    K_{\loc}^{\rm flux}:
    \mathfrak I_\Lambda^{\loc}\to Y_{\loc}^{\rm flux},
\]
\[
    K_{\loc}^{\rm prs}:
    \mathfrak I_\Lambda^{\loc}\to Y_{\loc}^{\rm prs},
    \qquad
    K_{\loc}^{\rm mom}:
    \mathfrak I_\Lambda^{\loc}\to Y_{\loc}^{\rm mom}.
\]
The operator \(K_{\loc}^{\rm en}\) records localization leakage in the
energy or trace ledger.  The operator \(K_{\loc}^{\rm flux}\) records leakage
from flux terms generated by a cutoff.  The operator \(K_{\loc}^{\rm prs}\)
records the pressure contribution to the localized energy or flux budget.
The operator \(K_{\loc}^{\rm mom}\) records the mismatch in localized
momentum residuals.  Boundedness of these operators is a finite-window
localization datum.
\end{definition}

\begin{definition}[Localization leakage residuals]
\label{def:localization-leakage-residuals}
For an intrinsic package \(\mathcal D\), define
\[
    \mathcal R_{\loc}^{\rm en}(\mathcal D)
    :=
    \|K_{\loc}^{\rm en}\mathcal D_\ast\|_{Y_{\loc}^{\rm en}},
\]
\[
    \mathcal R_{\loc}^{\rm flux}(\mathcal D)
    :=
    \|K_{\loc}^{\rm flux}\mathcal D_\ast\|_{Y_{\loc}^{\rm flux}},
\]
\[
    \mathcal R_{\loc}^{\rm prs}(\mathcal D)
    :=
    \|K_{\loc}^{\rm prs}\mathcal D_\ast\|_{Y_{\loc}^{\rm prs}},
\]
and
\[
    \mathcal R_{\loc}^{\rm mom}(\mathcal D)
    :=
    \|K_{\loc}^{\rm mom}\mathcal D_\ast\|_{Y_{\loc}^{\rm mom}}.
\]
The total localization leakage residual is
\[
    \mathcal R_{\rm loc}(\mathcal D)
    :=
    \mathcal R_{\loc}^{\rm en}(\mathcal D)
    +
    \mathcal R_{\loc}^{\rm flux}(\mathcal D)
    +
    \mathcal R_{\loc}^{\rm prs}(\mathcal D)
    +
    \mathcal R_{\loc}^{\rm mom}(\mathcal D).
\]
\end{definition}

\begin{lemma}[Componentwise localization leakage bounds]
\label{lem:componentwise-localization-leakage-bounds}
Set
\[
    C_{\loc}^{\rm en}
    :=
    \|K_{\loc}^{\rm en}\|_{\mathfrak I_\Lambda^{\loc}
    \to Y_{\loc}^{\rm en}},
    \qquad
    C_{\loc}^{\rm flux}
    :=
    \|K_{\loc}^{\rm flux}\|_{\mathfrak I_\Lambda^{\loc}
    \to Y_{\loc}^{\rm flux}},
\]
\[
    C_{\loc}^{\rm prs}
    :=
    \|K_{\loc}^{\rm prs}\|_{\mathfrak I_\Lambda^{\loc}
    \to Y_{\loc}^{\rm prs}},
    \qquad
    C_{\loc}^{\rm mom}
    :=
    \|K_{\loc}^{\rm mom}\|_{\mathfrak I_\Lambda^{\loc}
    \to Y_{\loc}^{\rm mom}}.
\]
Then
\[
    \mathcal R_{\loc}^{\rm en}(\mathcal D)
    \le
    C_{\loc}^{\rm en}\|\mathcal D_\ast\|_{\loc,{\rm int}},
\]
\[
    \mathcal R_{\loc}^{\rm flux}(\mathcal D)
    \le
    C_{\loc}^{\rm flux}\|\mathcal D_\ast\|_{\loc,{\rm int}},
\]
\[
    \mathcal R_{\loc}^{\rm prs}(\mathcal D)
    \le
    C_{\loc}^{\rm prs}\|\mathcal D_\ast\|_{\loc,{\rm int}},
\]
and
\[
    \mathcal R_{\loc}^{\rm mom}(\mathcal D)
    \le
    C_{\loc}^{\rm mom}\|\mathcal D_\ast\|_{\loc,{\rm int}}.
\]
\end{lemma}

\begin{proof}
Each inequality is the boundedness estimate for the corresponding
finite-window localization leakage operator, evaluated on the selected
representative \(\mathcal D_\ast\).
\end{proof}

\begin{proposition}[Localization leakage sub-budget]
\label{prop:localization-leakage-subbudget-enhanced-tail}
Assume the enhanced-tail same-gauge selector from
\Cref{ass:enhanced-tail-same-gauge-selector} and the finite-window
localization leakage operators above.  Set
\[
    C_{\loc}
    :=
    C_{\loc}^{\rm en}
    +
    C_{\loc}^{\rm flux}
    +
    C_{\loc}^{\rm prs}
    +
    C_{\loc}^{\rm mom}.
\]
Then
\[
    \mathcal R_{\rm loc}(\mathcal D)
    \le
    C_{\loc}
    \Dist_{\loc,{\rm int}}^{\sharp,{\rm tail}}
    (\mathcal D,\Gamma_\Lambda^{\rm int})
    +
    C_{\loc}\delta_{\rm sel}^{\sharp}.
\]
Equivalently, the localization leakage residual satisfies
\[
    \mathcal R_{\rm loc}(\mathcal D)
    \le
    \eta_{\loc}
    \Dist_{\loc,{\rm int}}^{\sharp,{\rm tail}}
    (\mathcal D,\Gamma_\Lambda^{\rm int})
    +
    \Delta_{\loc},
\]
with
\[
    \eta_{\loc}=C_{\loc},
    \qquad
    \Delta_{\loc}=C_{\loc}\delta_{\rm sel}^{\sharp}.
\]
If the selector is an exact enhanced-tail minimizer, then
\(\delta_{\rm sel}^{\sharp}=0\) and \(\Delta_{\loc}=0\).
\end{proposition}

\begin{proof}
Summing the componentwise estimates in
\Cref{lem:componentwise-localization-leakage-bounds} gives
\[
    \mathcal R_{\rm loc}(\mathcal D)
    \le
    C_{\loc}\|\mathcal D_\ast\|_{\loc,{\rm int}}.
\]
The selector bound in
\Cref{lem:enhanced-tail-selector-core-bound} gives the stated estimate.
\end{proof}

\begin{proposition}[Affine localization leakage variant]
\label{prop:affine-localization-leakage-variant}
Suppose that each localization leakage component satisfies an affine
finite-window estimate
\[
    \mathcal R_{\loc}^{a}(\mathcal D)
    \le
    C_{\loc}^{a}\|\mathcal D_\ast\|_{\loc,{\rm int}}
    +
    \Delta_{\loc}^{a},
    \qquad
    a\in\{{\rm en},{\rm flux},{\rm prs},{\rm mom}\}.
\]
Then
\[
    \mathcal R_{\rm loc}(\mathcal D)
    \le
    \eta_{\loc}
    \Dist_{\loc,{\rm int}}^{\sharp,{\rm tail}}
    (\mathcal D,\Gamma_\Lambda^{\rm int})
    +
    \Delta_{\loc},
\]
where
\[
    \eta_{\loc}
    =
    C_{\loc}^{\rm en}
    +
    C_{\loc}^{\rm flux}
    +
    C_{\loc}^{\rm prs}
    +
    C_{\loc}^{\rm mom},
\]
and
\[
    \Delta_{\loc}
    =
    \eta_{\loc}\delta_{\rm sel}^{\sharp}
    +
    \Delta_{\loc}^{\rm en}
    +
    \Delta_{\loc}^{\rm flux}
    +
    \Delta_{\loc}^{\rm prs}
    +
    \Delta_{\loc}^{\rm mom}.
\]
\end{proposition}

\begin{proof}
Sum the four affine estimates and then use
\Cref{lem:enhanced-tail-selector-core-bound}.
\end{proof}

\begin{corollary}[Localization contribution to the enhanced-tail residual ledger]
\label{cor:localization-contribution-enhanced-tail-ledger}
In the decomposed enhanced-tail residual budget from
\Cref{prop:enhanced-tail-residual-budget-from-subbudgets}, the localization
contribution may be taken to be
\[
    \eta_{\loc}
    =
    C_{\loc}^{\rm en}
    +
    C_{\loc}^{\rm flux}
    +
    C_{\loc}^{\rm prs}
    +
    C_{\loc}^{\rm mom}
\]
and
\[
    \Delta_{\loc}
    =
    \eta_{\loc}\delta_{\rm sel}^{\sharp}
    +
    \Delta_{\loc}^{\rm en}
    +
    \Delta_{\loc}^{\rm flux}
    +
    \Delta_{\loc}^{\rm prs}
    +
    \Delta_{\loc}^{\rm mom},
\]
with all \(\Delta_{\loc}^{a}=0\) in the purely linear operator-norm case.
Thus the decomposed residual-budget constants remain
\[
\begin{aligned}
    \eta_{\Lambda,{\rm dec}}^{\sharp,{\rm tail}}
    =
    &\eta_{\rm core}
    +
    \max\{\alpha_{\rm proj}^{-1},\alpha_{\rm harm}^{-1}\}
    +
    \eta_{\rm chart}\\
    &+
    \eta_{\loc}
    +
    \eta_{\rm rep}
    +
    \eta_{\rm gate},
\end{aligned}
\]
and
\[
    \Delta_{\Lambda,{\rm dec}}^{\sharp,{\rm tail}}
    =
    \Delta_{\rm core}
    +
    \Delta_{\rm chart}
    +
    \Delta_{\loc}
    +
    \Delta_{\rm rep}
    +
    \Delta_{\rm gate}.
\]
\end{corollary}

\begin{proof}
Use
\Cref{prop:localization-leakage-subbudget-enhanced-tail} or
\Cref{prop:affine-localization-leakage-variant} as the localization
sub-budget in
\Cref{prop:enhanced-tail-residual-budget-from-subbudgets}.
\end{proof}

\begin{remark}[Status of the localization leakage sub-budget]
This subsection proves only a finite-window localization leakage sub-budget.
It says that, once the localization leakage operators and the same-gauge
representative selection are fixed, the localization residual is controlled
by the enhanced-tail quotient geometry.  It does not prove a local energy
inequality estimate, a pressure localization estimate, a momentum residual
estimate, smallness, scale-uniform localization control, or
Navier--Stokes regularity.  It also does not compare
\[
    \Dist_{\loc,{\rm int}}^{\sharp,{\rm tail}}
\]
with
\[
    \Dist_{\loc,{\rm int}}.
\]
\end{remark}

\subsection{Reproduction drift sub-budget in the enhanced-tail geometry}
\label{subsec:reproduction-drift-subbudget-enhanced-tail}

We next replace the placeholder reproduction residual
\(\mathcal R_{\rm rep}\) in
\Cref{def:decomposed-enhanced-tail-residual-ledger}.  This is a
finite-window bookkeeping module.  It records the drift between adjacent
scale packages after a reproduction model has been fixed; it does not prove
that Navier--Stokes dynamically generates the reproduction maps, and it does
not prove scale-uniform reproduction.

\begin{definition}[Finite-window reproduction drift operators]
\label{def:finite-window-reproduction-drift-operators}
Let
\[
    \Lambda_{\rm adj}
    :=
    \{k\in\Lambda:\ k+1\in\Lambda\}
\]
be the adjacent-scale index set.  For each
\(k\in\Lambda_{\rm adj}\), let \(Y_k^{\rm rep}\) be a
finite-dimensional normed space and fix a bounded finite-window operator
\[
    K_k^{\rm rep}:
    \mathfrak I_\Lambda^{\loc}\to Y_k^{\rm rep}.
\]
The quantity \(K_k^{\rm rep}\mathcal D_\ast\) records the model-level drift
between the selected intrinsic package at scale \(k+1\) and the package
predicted from scale \(k\) by the chosen adjacent-scale reproduction map.
Schematically, it represents a coordinate-extracted or projected term of the
form
\[
    \mathcal D_{\ast,k+1}
    -
    \mathcal R_k^{\rm int}\mathcal D_{\ast,k},
\]
possibly after clean/local charting.  The map
\(\mathcal R_k^{\rm int}\) is fixed finite-window reproduction data here; no
Navier--Stokes generation of this map is asserted.
\end{definition}

\begin{definition}[Reproduction drift residuals]
\label{def:reproduction-drift-residuals}
For an intrinsic package \(\mathcal D\), define the component reproduction
drift residuals by
\[
    \mathcal R_k^{\rm rep}(\mathcal D)
    :=
    \|K_k^{\rm rep}\mathcal D_\ast\|_{Y_k^{\rm rep}},
    \qquad
    k\in\Lambda_{\rm adj}.
\]
The total reproduction drift residual is
\[
    \mathcal R_{\rm rep}(\mathcal D)
    :=
    \sum_{k\in\Lambda_{\rm adj}}
    \mathcal R_k^{\rm rep}(\mathcal D).
\]
\end{definition}

\begin{lemma}[Componentwise reproduction drift bounds]
\label{lem:componentwise-reproduction-drift-bounds}
For each \(k\in\Lambda_{\rm adj}\), set
\[
    C_k^{\rm rep}
    :=
    \|K_k^{\rm rep}\|_{\mathfrak I_\Lambda^{\loc}\to Y_k^{\rm rep}},
    \qquad
    C_{\rm rep}
    :=
    \sum_{k\in\Lambda_{\rm adj}}C_k^{\rm rep}.
\]
Then
\[
    \mathcal R_k^{\rm rep}(\mathcal D)
    \le
    C_k^{\rm rep}\|\mathcal D_\ast\|_{\loc,{\rm int}},
    \qquad
    k\in\Lambda_{\rm adj},
\]
and hence
\[
    \mathcal R_{\rm rep}(\mathcal D)
    \le
    C_{\rm rep}\|\mathcal D_\ast\|_{\loc,{\rm int}}.
\]
\end{lemma}

\begin{proof}
The component estimate is the operator-norm bound for
\(K_k^{\rm rep}\), evaluated at the selected same-gauge representative
\(\mathcal D_\ast\).  Summing over \(k\in\Lambda_{\rm adj}\) gives the
second estimate.
\end{proof}

\begin{proposition}[Reproduction drift sub-budget]
\label{prop:reproduction-drift-subbudget-enhanced-tail}
Assume the enhanced-tail same-gauge selector from
\Cref{ass:enhanced-tail-same-gauge-selector} and the finite-window
reproduction drift operators above.  Then
\[
    \mathcal R_{\rm rep}(\mathcal D)
    \le
    C_{\rm rep}
    \Dist_{\loc,{\rm int}}^{\sharp,{\rm tail}}
    (\mathcal D,\Gamma_\Lambda^{\rm int})
    +
    C_{\rm rep}\delta_{\rm sel}^{\sharp}.
\]
Equivalently, the reproduction drift residual satisfies
\[
    \mathcal R_{\rm rep}(\mathcal D)
    \le
    \eta_{\rm rep}
    \Dist_{\loc,{\rm int}}^{\sharp,{\rm tail}}
    (\mathcal D,\Gamma_\Lambda^{\rm int})
    +
    \Delta_{\rm rep},
\]
with
\[
    \eta_{\rm rep}=C_{\rm rep},
    \qquad
    \Delta_{\rm rep}=C_{\rm rep}\delta_{\rm sel}^{\sharp}.
\]
If the selector is an exact enhanced-tail minimizer, then
\(\delta_{\rm sel}^{\sharp}=0\) and \(\Delta_{\rm rep}=0\).
\end{proposition}

\begin{proof}
By \Cref{lem:componentwise-reproduction-drift-bounds},
\[
    \mathcal R_{\rm rep}(\mathcal D)
    \le
    C_{\rm rep}\|\mathcal D_\ast\|_{\loc,{\rm int}}.
\]
The enhanced-tail selector bound in
\Cref{lem:enhanced-tail-selector-core-bound} gives the displayed
sub-budget.
\end{proof}

\begin{proposition}[Affine reproduction drift variant]
\label{prop:affine-reproduction-drift-variant}
Suppose that each adjacent-scale reproduction component satisfies an affine
finite-window estimate
\[
    \mathcal R_k^{\rm rep}(\mathcal D)
    \le
    C_k^{\rm rep}\|\mathcal D_\ast\|_{\loc,{\rm int}}
    +
    \Delta_k^{\rm rep},
    \qquad
    k\in\Lambda_{\rm adj}.
\]
Then
\[
    \mathcal R_{\rm rep}(\mathcal D)
    \le
    \eta_{\rm rep}
    \Dist_{\loc,{\rm int}}^{\sharp,{\rm tail}}
    (\mathcal D,\Gamma_\Lambda^{\rm int})
    +
    \Delta_{\rm rep},
\]
where
\[
    \eta_{\rm rep}
    =
    \sum_{k\in\Lambda_{\rm adj}}C_k^{\rm rep},
\]
and
\[
    \Delta_{\rm rep}
    =
    \eta_{\rm rep}\delta_{\rm sel}^{\sharp}
    +
    \sum_{k\in\Lambda_{\rm adj}}\Delta_k^{\rm rep}.
\]
\end{proposition}

\begin{proof}
Sum the affine component estimates and use
\Cref{lem:enhanced-tail-selector-core-bound}.
\end{proof}

\begin{corollary}[Reproduction contribution to the enhanced-tail residual ledger]
\label{cor:reproduction-contribution-enhanced-tail-ledger}
In the decomposed enhanced-tail residual budget from
\Cref{prop:enhanced-tail-residual-budget-from-subbudgets}, the reproduction
contribution may be taken to be
\[
    \eta_{\rm rep}
    =
    \sum_{k\in\Lambda_{\rm adj}}C_k^{\rm rep},
\]
and
\[
    \Delta_{\rm rep}
    =
    \eta_{\rm rep}\delta_{\rm sel}^{\sharp}
    +
    \sum_{k\in\Lambda_{\rm adj}}\Delta_k^{\rm rep},
\]
with all \(\Delta_k^{\rm rep}=0\) in the purely linear operator-norm case.
Therefore the decomposed residual-budget constants remain
\[
\begin{aligned}
    \eta_{\Lambda,{\rm dec}}^{\sharp,{\rm tail}}
    =
    &\eta_{\rm core}
    +
    \max\{\alpha_{\rm proj}^{-1},\alpha_{\rm harm}^{-1}\}
    +
    \eta_{\rm chart}\\
    &+
    \eta_{\loc}
    +
    \eta_{\rm rep}
    +
    \eta_{\rm gate},
\end{aligned}
\]
and
\[
    \Delta_{\Lambda,{\rm dec}}^{\sharp,{\rm tail}}
    =
    \Delta_{\rm core}
    +
    \Delta_{\rm chart}
    +
    \Delta_{\loc}
    +
    \Delta_{\rm rep}
    +
    \Delta_{\rm gate}.
\]
\end{corollary}

\begin{proof}
Insert either reproduction sub-budget, the linear one from
\Cref{prop:reproduction-drift-subbudget-enhanced-tail} or the affine one from
\Cref{prop:affine-reproduction-drift-variant}, into the decomposed ledger
\Cref{prop:enhanced-tail-residual-budget-from-subbudgets}.
\end{proof}

\begin{remark}[Status of the reproduction drift sub-budget]
This subsection proves only a finite-window reproduction drift sub-budget.
It says that, once the adjacent-scale reproduction drift operators and the
same-gauge representative selection are fixed, the reproduction drift
residual is controlled by the enhanced-tail quotient geometry.  It does not
prove that Navier--Stokes generates the reproduction maps, exact
reproduction, scale-uniformity, smallness, or Navier--Stokes regularity.  It
also does not compare
\[
    \Dist_{\loc,{\rm int}}^{\sharp,{\rm tail}}
\]
with
\[
    \Dist_{\loc,{\rm int}}.
\]
\end{remark}

\subsection{Gate/tax mismatch sub-budget in the enhanced-tail geometry}
\label{subsec:gate-tax-mismatch-subbudget-enhanced-tail}

We finally replace the placeholder gate residual
\(\mathcal R_{\rm gate}\) in
\Cref{def:decomposed-enhanced-tail-residual-ledger}.  This subsection
records finite-window mismatch between localized gate data, clean detector
gates, and tax or ledger-slack coordinates.  It is not a lower bound for a
concrete Navier--Stokes pressure, flux, or dissipation tax.

\begin{definition}[Finite-window gate/tax mismatch operators]
\label{def:finite-window-gate-tax-mismatch-operators}
Let
\[
    Y_{\rm gate}^{\rm det},\qquad
    Y_{\rm gate}^{\rm tax},\qquad
    Y_{\rm gate}^{\rm slack}
\]
be finite-dimensional normed spaces.  Fix bounded finite-window operators
\[
    K_{\rm gate}^{\rm det}:
    \mathfrak I_\Lambda^{\loc}\to Y_{\rm gate}^{\rm det},
    \qquad
    K_{\rm gate}^{\rm tax}:
    \mathfrak I_\Lambda^{\loc}\to Y_{\rm gate}^{\rm tax},
\]
and
\[
    K_{\rm gate}^{\rm slack}:
    \mathfrak I_\Lambda^{\loc}\to Y_{\rm gate}^{\rm slack}.
\]
The operator \(K_{\rm gate}^{\rm det}\) records detector-gate mismatch,
namely the mismatch between localized gate coordinates and the clean
detector gate after applying the local-to-clean chart.  The operator
\(K_{\rm gate}^{\rm tax}\) records tax-functional mismatch, namely the
finite-window discrepancy between localized tax data and the clean tax
functional \(\Tax_\Lambda^{\cl}\).  The operator
\(K_{\rm gate}^{\rm slack}\) records ledger-slack mismatch.  Their
boundedness is part of the finite-window gate/tax datum.
\end{definition}

\begin{definition}[Gate/tax mismatch residuals]
\label{def:gate-tax-mismatch-residuals}
For an intrinsic package \(\mathcal D\), define
\[
    \mathcal R_{\rm gate}^{\rm det}(\mathcal D)
    :=
    \|K_{\rm gate}^{\rm det}\mathcal D_\ast\|_{Y_{\rm gate}^{\rm det}},
\]
\[
    \mathcal R_{\rm gate}^{\rm tax}(\mathcal D)
    :=
    \|K_{\rm gate}^{\rm tax}\mathcal D_\ast\|_{Y_{\rm gate}^{\rm tax}},
\]
and
\[
    \mathcal R_{\rm gate}^{\rm slack}(\mathcal D)
    :=
    \|K_{\rm gate}^{\rm slack}\mathcal D_\ast\|
    _{Y_{\rm gate}^{\rm slack}}.
\]
The total gate/tax mismatch residual is
\[
    \mathcal R_{\rm gate}(\mathcal D)
    :=
    \mathcal R_{\rm gate}^{\rm det}(\mathcal D)
    +
    \mathcal R_{\rm gate}^{\rm tax}(\mathcal D)
    +
    \mathcal R_{\rm gate}^{\rm slack}(\mathcal D).
\]
\end{definition}

\begin{lemma}[Componentwise gate/tax mismatch bounds]
\label{lem:componentwise-gate-tax-mismatch-bounds}
Set
\[
    C_{\rm gate}^{\rm det}
    :=
    \|K_{\rm gate}^{\rm det}\|_{\mathfrak I_\Lambda^{\loc}
    \to Y_{\rm gate}^{\rm det}},
    \qquad
    C_{\rm gate}^{\rm tax}
    :=
    \|K_{\rm gate}^{\rm tax}\|_{\mathfrak I_\Lambda^{\loc}
    \to Y_{\rm gate}^{\rm tax}},
\]
and
\[
    C_{\rm gate}^{\rm slack}
    :=
    \|K_{\rm gate}^{\rm slack}\|_{\mathfrak I_\Lambda^{\loc}
    \to Y_{\rm gate}^{\rm slack}}.
\]
Then
\[
    \mathcal R_{\rm gate}^{\rm det}(\mathcal D)
    \le
    C_{\rm gate}^{\rm det}\|\mathcal D_\ast\|_{\loc,{\rm int}},
\]
\[
    \mathcal R_{\rm gate}^{\rm tax}(\mathcal D)
    \le
    C_{\rm gate}^{\rm tax}\|\mathcal D_\ast\|_{\loc,{\rm int}},
\]
and
\[
    \mathcal R_{\rm gate}^{\rm slack}(\mathcal D)
    \le
    C_{\rm gate}^{\rm slack}\|\mathcal D_\ast\|_{\loc,{\rm int}}.
\]
\end{lemma}

\begin{proof}
Each estimate is the operator-norm bound for the corresponding operator,
evaluated at \(\mathcal D_\ast\).
\end{proof}

\begin{proposition}[Gate/tax mismatch sub-budget]
\label{prop:gate-tax-mismatch-subbudget-enhanced-tail}
Assume the enhanced-tail same-gauge selector from
\Cref{ass:enhanced-tail-same-gauge-selector} and the finite-window gate/tax
mismatch operators above.  Set
\[
    C_{\rm gate}
    :=
    C_{\rm gate}^{\rm det}
    +
    C_{\rm gate}^{\rm tax}
    +
    C_{\rm gate}^{\rm slack}.
\]
Then
\[
    \mathcal R_{\rm gate}(\mathcal D)
    \le
    C_{\rm gate}
    \Dist_{\loc,{\rm int}}^{\sharp,{\rm tail}}
    (\mathcal D,\Gamma_\Lambda^{\rm int})
    +
    C_{\rm gate}\delta_{\rm sel}^{\sharp}.
\]
Equivalently, the gate/tax mismatch residual satisfies
\[
    \mathcal R_{\rm gate}(\mathcal D)
    \le
    \eta_{\rm gate}
    \Dist_{\loc,{\rm int}}^{\sharp,{\rm tail}}
    (\mathcal D,\Gamma_\Lambda^{\rm int})
    +
    \Delta_{\rm gate},
\]
with
\[
    \eta_{\rm gate}=C_{\rm gate},
    \qquad
    \Delta_{\rm gate}=C_{\rm gate}\delta_{\rm sel}^{\sharp}.
\]
If the selector is an exact enhanced-tail minimizer, then
\(\delta_{\rm sel}^{\sharp}=0\) and \(\Delta_{\rm gate}=0\).
\end{proposition}

\begin{proof}
Summing the componentwise estimates in
\Cref{lem:componentwise-gate-tax-mismatch-bounds} gives
\[
    \mathcal R_{\rm gate}(\mathcal D)
    \le
    C_{\rm gate}\|\mathcal D_\ast\|_{\loc,{\rm int}}.
\]
The enhanced-tail selector bound in
\Cref{lem:enhanced-tail-selector-core-bound} gives the displayed estimate.
\end{proof}

\begin{proposition}[Affine gate/tax mismatch variant]
\label{prop:affine-gate-tax-mismatch-variant}
Suppose that each gate/tax mismatch component satisfies an affine
finite-window estimate
\[
    \mathcal R_{\rm gate}^{a}(\mathcal D)
    \le
    C_{\rm gate}^{a}\|\mathcal D_\ast\|_{\loc,{\rm int}}
    +
    \Delta_{\rm gate}^{a},
    \qquad
    a\in\{{\rm det},{\rm tax},{\rm slack}\}.
\]
Then
\[
    \mathcal R_{\rm gate}(\mathcal D)
    \le
    \eta_{\rm gate}
    \Dist_{\loc,{\rm int}}^{\sharp,{\rm tail}}
    (\mathcal D,\Gamma_\Lambda^{\rm int})
    +
    \Delta_{\rm gate},
\]
where
\[
    \eta_{\rm gate}
    =
    C_{\rm gate}^{\rm det}
    +
    C_{\rm gate}^{\rm tax}
    +
    C_{\rm gate}^{\rm slack},
\]
and
\[
    \Delta_{\rm gate}
    =
    \eta_{\rm gate}\delta_{\rm sel}^{\sharp}
    +
    \Delta_{\rm gate}^{\rm det}
    +
    \Delta_{\rm gate}^{\rm tax}
    +
    \Delta_{\rm gate}^{\rm slack}.
\]
\end{proposition}

\begin{proof}
Sum the three affine component estimates and use
\Cref{lem:enhanced-tail-selector-core-bound}.
\end{proof}

\begin{corollary}[Gate/tax contribution to the enhanced-tail residual ledger]
\label{cor:gate-tax-contribution-enhanced-tail-ledger}
In the decomposed enhanced-tail residual budget from
\Cref{prop:enhanced-tail-residual-budget-from-subbudgets}, the gate/tax
contribution may be taken to be
\[
    \eta_{\rm gate}
    =
    C_{\rm gate}^{\rm det}
    +
    C_{\rm gate}^{\rm tax}
    +
    C_{\rm gate}^{\rm slack},
\]
and
\[
    \Delta_{\rm gate}
    =
    \eta_{\rm gate}\delta_{\rm sel}^{\sharp}
    +
    \Delta_{\rm gate}^{\rm det}
    +
    \Delta_{\rm gate}^{\rm tax}
    +
    \Delta_{\rm gate}^{\rm slack},
\]
with all \(\Delta_{\rm gate}^{a}=0\) in the purely linear operator-norm case.
Therefore all four finite-window residual sub-budgets in the decomposed
enhanced-tail ledger have explicit coefficients, and
\[
\begin{aligned}
    \eta_{\Lambda,{\rm dec}}^{\sharp,{\rm tail}}
    =
    &\eta_{\rm core}
    +
    \max\{\alpha_{\rm proj}^{-1},\alpha_{\rm harm}^{-1}\}
    +
    \eta_{\rm chart}\\
    &+
    \eta_{\loc}
    +
    \eta_{\rm rep}
    +
    \eta_{\rm gate},
\end{aligned}
\]
while
\[
    \Delta_{\Lambda,{\rm dec}}^{\sharp,{\rm tail}}
    =
    \Delta_{\rm core}
    +
    \Delta_{\rm chart}
    +
    \Delta_{\loc}
    +
    \Delta_{\rm rep}
    +
    \Delta_{\rm gate}.
\]
\end{corollary}

\begin{proof}
Insert either the linear gate/tax sub-budget from
\Cref{prop:gate-tax-mismatch-subbudget-enhanced-tail} or the affine variant
from \Cref{prop:affine-gate-tax-mismatch-variant} into the decomposed ledger
\Cref{prop:enhanced-tail-residual-budget-from-subbudgets}.  The chart
contribution is supplied by
\Cref{cor:chart-contribution-enhanced-tail-ledger}.  The localization and
reproduction contributions are supplied by
\Cref{cor:localization-contribution-enhanced-tail-ledger} and
\Cref{cor:reproduction-contribution-enhanced-tail-ledger}.
\end{proof}

\begin{remark}[Status of the gate/tax mismatch sub-budget]
This subsection proves only a finite-window gate/tax mismatch sub-budget.  It
says that, once the detector-gate, tax-functional, and ledger-slack mismatch
operators and the same-gauge representative selection are fixed, the
gate/tax residual is controlled by the enhanced-tail quotient geometry.  It
does not prove a positive pressure tax, a positive flux tax, coercivity of
\(\Tax_\Lambda^{\cl}\), scale-uniformity, smallness, or Navier--Stokes
regularity.  It also does not compare
\[
    \Dist_{\loc,{\rm int}}^{\sharp,{\rm tail}}
\]
with
\[
    \Dist_{\loc,{\rm int}}.
\]
\end{remark}

\subsection{Assembled enhanced-tail residual-budget criterion}
\label{subsec:assembled-enhanced-tail-residual-budget-criterion}

We now collect the four finite-window sub-budgets recorded above.  This is an
assembly step only: it does not derive the enhanced-tail comparison
hypothesis, does not compare the enhanced-tail distance with the original
intrinsic distance, and does not prove any scale-uniform estimate.

\begin{corollary}[Assembled finite-window enhanced-tail residual budget]
\label{cor:assembled-finite-window-enhanced-tail-residual-budget}
Assume the core and tail estimates in
\Cref{ass:same-gauge-enhanced-tail-compatibility} and
\Cref{lem:simultaneous-pressure-tail-observability}.  Assume also the
finite-window chart, localization, reproduction, and gate/tax sub-budgets in
\Cref{prop:chart-residual-subbudget-enhanced-tail},
\Cref{prop:localization-leakage-subbudget-enhanced-tail},
\Cref{prop:reproduction-drift-subbudget-enhanced-tail}, and
\Cref{prop:gate-tax-mismatch-subbudget-enhanced-tail}.  Then the decomposed
enhanced-tail residual satisfies
\[
    \Err_{\Lambda,{\rm dec}}^{\loc}(\mathcal D)
    \le
    \eta_{\Lambda,{\rm dec}}^{\sharp,{\rm tail}}
    \Dist_{\loc,{\rm int}}^{\sharp,{\rm tail}}
    (\mathcal D,\Gamma_\Lambda^{\rm int})
    +
    \Delta_{\Lambda,{\rm dec}}^{\sharp,{\rm tail}},
\]
where
\[
\begin{aligned}
    \eta_{\Lambda,{\rm dec}}^{\sharp,{\rm tail}}
    =
    &\eta_{\rm core}
    +
    \max\{\alpha_{\rm proj}^{-1},\alpha_{\rm harm}^{-1}\}
    +
    \eta_{\rm chart}\\
    &+
    \eta_{\loc}
    +
    \eta_{\rm rep}
    +
    \eta_{\rm gate},
\end{aligned}
\]
and
\[
    \Delta_{\Lambda,{\rm dec}}^{\sharp,{\rm tail}}
    =
    \Delta_{\rm core}
    +
    \Delta_{\rm chart}
    +
    \Delta_{\loc}
    +
    \Delta_{\rm rep}
    +
    \Delta_{\rm gate}.
\]
Consequently, if the detector comparison in
\Cref{ass:enhanced-tail-detector-residual-subbudgets} holds with the same
decomposed residual, then the enhanced-tail residual-budget hypothesis is
realized with the displayed constants.
\end{corollary}

\begin{proof}
The definition of \(\Err_{\Lambda,{\rm dec}}^{\loc}\) in
\Cref{prop:enhanced-tail-residual-budget-from-subbudgets} is the sum of the
core, tail, chart, localization, reproduction, and gate/tax residuals.  The
core estimate is
\Cref{ass:same-gauge-enhanced-tail-compatibility}.  The pressure-tail
estimate is \Cref{lem:simultaneous-pressure-tail-observability}.  The four
finite-window sub-budgets are supplied by the four propositions listed in
the statement.  Adding the six estimates gives the displayed constants.
\end{proof}

\begin{remark}[Status of the assembled criterion]
\Cref{cor:assembled-finite-window-enhanced-tail-residual-budget} is a
bookkeeping consequence of the finite-window sub-budgets.  It does not prove
that the enhanced-tail detector comparison holds for localized
Navier--Stokes packages, does not prove the local-to-clean comparison, and
does not compare \(\Dist_{\loc,{\rm int}}^{\sharp,{\rm tail}}\) with
\(\Dist_{\loc,{\rm int}}\).  Those are the next PDE-facing problems.
\end{remark}

\subsection{Conditional comparison between enhanced-tail and intrinsic geometries}
\label{subsec:conditional-enhanced-tail-intrinsic-comparison}

We now compare the combined enhanced-tail geometry with the original
intrinsic geometry.  This is the first PDE-facing step after the residual
ledger has been assembled.  The comparison is conditional: the projection and
harmonic tails must be controlled on a common intrinsic representative.

\begin{lemma}[One-sided comparison with the intrinsic distance]
\label{lem:enhanced-tail-dominates-intrinsic-distance}
For every intrinsic package \(\mathcal D\),
\[
    \Dist_{\loc,{\rm int}}(\mathcal D,\Gamma_\Lambda^{\rm int})
    \le
    \Dist_{\loc,{\rm int}}^{\sharp,{\rm tail}}
    (\mathcal D,\Gamma_\Lambda^{\rm int}).
\]
\end{lemma}

\begin{proof}
For every \(\zeta\in\Gamma_\Lambda^{\rm int}\), the tail terms in
\Cref{def:combined-pressure-tail-enhanced-intrinsic-distance} are
nonnegative, so
\[
    \|\mathcal D-\zeta\|_{\loc,{\rm int}}
    \le
    \|\mathcal D-\zeta\|_{\loc,{\rm int}}
    +
    \alpha_{\rm proj}\mathcal T_{\rm proj}(\mathcal D;\zeta)
    +
    \alpha_{\rm harm}\mathcal T_{\rm harm}(\mathcal D;\zeta).
\]
Taking the infimum over \(\zeta\) gives the claim.
\end{proof}

\begin{assumption}[Common intrinsic representative for tail approximation]
\label{ass:common-intrinsic-representative-tail-approximation}
For every intrinsic package \(\mathcal D\), there is a representative
\[
    \zeta_{\rm int}(\mathcal D)\in\Gamma_\Lambda^{\rm int}
\]
and a representative-selection error \(\delta_{\rm int}\ge0\) such that
\[
    \|\mathcal D-\zeta_{\rm int}\|_{\loc,{\rm int}}
    \le
    \Dist_{\loc,{\rm int}}(\mathcal D,\Gamma_\Lambda^{\rm int})
    +
    \delta_{\rm int}.
\]
On this same representative, assume that the projection and harmonic tails
satisfy
\[
    \mathcal T_{\rm proj}(\mathcal D;\zeta_{\rm int})
    \le
    C_{\rm proj}^{\rm app}
    \Dist_{\loc,{\rm int}}(\mathcal D,\Gamma_\Lambda^{\rm int})
    +
    \Delta_{{\rm proj},N},
\]
and
\[
    \mathcal T_{\rm harm}(\mathcal D;\zeta_{\rm int})
    \le
    C_{\rm harm}^{\rm app}
    \Dist_{\loc,{\rm int}}(\mathcal D,\Gamma_\Lambda^{\rm int})
    +
    \Delta_{{\rm harm},M}.
\]
Here \(C_{\rm proj}^{\rm app},C_{\rm harm}^{\rm app}<\infty\), while
\(\Delta_{{\rm proj},N},\Delta_{{\rm harm},M}\ge0\) are finite-window
approximation errors.  No decay of these errors as \(N\to\infty\) or
\(M\to\infty\) is assumed here.
\end{assumption}

\begin{theorem}[Conditional enhanced-tail/intrinsic comparison]
\label{thm:conditional-enhanced-tail-intrinsic-comparison}
Assume
\Cref{ass:common-intrinsic-representative-tail-approximation}.  Then, for
every intrinsic package \(\mathcal D\),
\[
\begin{aligned}
    &\Dist_{\loc,{\rm int}}^{\sharp,{\rm tail}}
    (\mathcal D,\Gamma_\Lambda^{\rm int})\\
    &\le
    \bigl(
        1
        +
        \alpha_{\rm proj}C_{\rm proj}^{\rm app}
        +
        \alpha_{\rm harm}C_{\rm harm}^{\rm app}
    \bigr)
    \Dist_{\loc,{\rm int}}(\mathcal D,\Gamma_\Lambda^{\rm int})\\
    &\quad+
    \delta_{\rm int}
    +
    \alpha_{\rm proj}\Delta_{{\rm proj},N}
    +
    \alpha_{\rm harm}\Delta_{{\rm harm},M}.
\end{aligned}
\]
\end{theorem}

\begin{proof}
Use the representative \(\zeta_{\rm int}(\mathcal D)\) from
\Cref{ass:common-intrinsic-representative-tail-approximation} as a competitor
in the infimum defining the enhanced-tail distance.  Then
\[
\begin{aligned}
    &\Dist_{\loc,{\rm int}}^{\sharp,{\rm tail}}
    (\mathcal D,\Gamma_\Lambda^{\rm int})\\
    &\le
    \|\mathcal D-\zeta_{\rm int}\|_{\loc,{\rm int}}
    +
    \alpha_{\rm proj}\mathcal T_{\rm proj}
    (\mathcal D;\zeta_{\rm int})
    +
    \alpha_{\rm harm}\mathcal T_{\rm harm}
    (\mathcal D;\zeta_{\rm int}).
\end{aligned}
\]
Substitute the three bounds in
\Cref{ass:common-intrinsic-representative-tail-approximation} and collect
terms.
\end{proof}

\begin{corollary}[Two-sided conditional geometry comparison]
\label{cor:two-sided-conditional-tail-intrinsic-geometry}
Under
\Cref{ass:common-intrinsic-representative-tail-approximation}, one has
\[
    \Dist_{\loc,{\rm int}}
    (\mathcal D,\Gamma_\Lambda^{\rm int})
    \le
    \Dist_{\loc,{\rm int}}^{\sharp,{\rm tail}}
    (\mathcal D,\Gamma_\Lambda^{\rm int})
\]
and
\[
    \Dist_{\loc,{\rm int}}^{\sharp,{\rm tail}}
    (\mathcal D,\Gamma_\Lambda^{\rm int})
    \le
    C_{{\rm tail}/{\rm int}}
    \Dist_{\loc,{\rm int}}
    (\mathcal D,\Gamma_\Lambda^{\rm int})
    +
    \Delta_{{\rm tail}/{\rm int}},
\]
where
\[
    C_{{\rm tail}/{\rm int}}
    :=
    1
    +
    \alpha_{\rm proj}C_{\rm proj}^{\rm app}
    +
    \alpha_{\rm harm}C_{\rm harm}^{\rm app},
\]
and
\[
    \Delta_{{\rm tail}/{\rm int}}
    :=
    \delta_{\rm int}
    +
    \alpha_{\rm proj}\Delta_{{\rm proj},N}
    +
    \alpha_{\rm harm}\Delta_{{\rm harm},M}.
\]
\end{corollary}

\begin{proof}
Combine
\Cref{lem:enhanced-tail-dominates-intrinsic-distance} and
\Cref{thm:conditional-enhanced-tail-intrinsic-comparison}.
\end{proof}

\begin{corollary}[Intrinsic lower bound inherited from enhanced-tail transfer]
\label{cor:intrinsic-lower-bound-from-enhanced-tail-transfer}
Assume that \(c_\Lambda^{\sharp,{\rm tail}}\ge0\) and that an enhanced-tail
localized transfer estimate has the form
\[
    \calM_\Lambda^{\loc}(\mathcal D)
    \ge
    c_\Lambda^{\sharp,{\rm tail}}
    \Dist_{\loc,{\rm int}}^{\sharp,{\rm tail}}
    (\mathcal D,\Gamma_\Lambda^{\rm int})
    -
    \Delta_\Lambda^{\sharp,{\rm tail}}.
\]
Then
\[
    \calM_\Lambda^{\loc}(\mathcal D)
    \ge
    c_\Lambda^{\sharp,{\rm tail}}
    \Dist_{\loc,{\rm int}}
    (\mathcal D,\Gamma_\Lambda^{\rm int})
    -
    \Delta_\Lambda^{\sharp,{\rm tail}}.
\]
\end{corollary}

\begin{proof}
By \Cref{lem:enhanced-tail-dominates-intrinsic-distance}, the enhanced-tail
distance is at least the original intrinsic distance.  Multiplication by the
nonnegative constant \(c_\Lambda^{\sharp,{\rm tail}}\) preserves the
inequality, and substitution into the assumed enhanced-tail transfer estimate
gives the result.
\end{proof}

\begin{remark}[Meaning of the reverse comparison]
The lower-bound corollary uses only the immediate comparison
\(\Dist_{\loc,{\rm int}}\le
\Dist_{\loc,{\rm int}}^{\sharp,{\rm tail}}\).  The reverse comparison in
\Cref{thm:conditional-enhanced-tail-intrinsic-comparison} is needed for a
different reason: it justifies that the enhanced-tail geometry is not adding
an uncontrolled new pressure-tail coordinate.  That justification is
conditional on the projection and harmonic tail approximation assumptions in
\Cref{ass:common-intrinsic-representative-tail-approximation}.
\end{remark}

\begin{remark}[Status of the conditional comparison]
This subsection does not prove tail approximation, pressure/tax coercivity,
scale-uniformity, or Navier--Stokes regularity.  It proves only a conditional
geometric implication.  If the projection and harmonic tails are controlled
on a common intrinsic representative, then the enhanced-tail distance is
comparable to the original intrinsic distance up to explicit additive
approximation errors.
\end{remark}

\section{Scope of the Results and Outlook}
\label{sec:roadmap}

The preceding sections establish a finite-window reduction rather than a
complete regularity theorem.  In relation to the preceding defect-cascade
formulation \cite{YuInvisible2026}, the present paper isolates the clean
finite-window compactness and transfer bookkeeping that such a cascade analysis
would require.  The purpose of this final section is to make the logical status
of the results explicit: which statements have been proved inside the
finite-window framework, which inputs remain conditional, and which PDE
estimates would turn the framework into a stronger Navier--Stokes result.

\subsection{Established finite-window consequences}

Within the definitions fixed in this paper, the main conclusions are the
following.

\begin{enumerate}[label=(\roman*),leftmargin=2em]
    \item \emph{Clean compactness gap.}  Under the clean detector gauge
    compatibility assumption in
    \Cref{ass:clean-detector-gauge-compatibility} and the kernel-free
    detector condition in
    \Cref{ass:kernel-free-computational-detector}, the finite-dimensional
    clean compactness argument gives a positive computational anti-phantom
    gap; see \Cref{thm:clean-computational-antiphantom}.

    \item \emph{Finite sector reduction.}  The detector family decomposes the
    possible non-gauge clean defect into finitely many observable sectors.
    Equivalently, a normalized clean defect cannot remain simultaneously
    invisible to all pressure, flux, energy, trace, reproduction, and tax
    channels; see \Cref{cor:sector-cover}.

    \item \emph{Conditional localized transfer.}  A clean finite-window gap
    yields a localized lower bound once the local-to-clean comparison and the
    localized residual budget are assumed; see
    \Cref{ass:local-to-clean-transfer-comparison,ass:localized-residual-budget}
    and \Cref{cor:localized-transfer-from-computational-gap}.

    \item \emph{Enhanced pressure-tail bookkeeping.}  The later sections
    refine the transfer mechanism by separating the intrinsic core, the clean
    projection tail, and the harmonic pressure tail.  The resulting
    enhanced-tail transfer theorem and its subclaim criterion isolate the
    chart-visibility and residual-budget inputs needed for a pressure-tail
    version of the localized lower bound; see
    \Cref{thm:enhanced-tail-localized-transfer,cor:subclaim-criterion-enhanced-tail-localized-transfer}.

    \item \emph{Conditional comparison with the intrinsic geometry.}  If the
    clean projection tail and the harmonic tail are controlled on a common
    intrinsic representative, then the enhanced-tail distance is comparable
    with the original intrinsic distance up to explicit additive errors; see
    \Cref{thm:conditional-enhanced-tail-intrinsic-comparison}.
\end{enumerate}

Thus the paper proves a precise finite-window implication: persistent
localized badness must be visible either through a clean detector sector or
through one of the explicitly recorded residual channels, provided the stated
comparison and budget hypotheses are available.

\subsection{Conditional inputs not proved here}

The framework deliberately separates algebraic compactness from the PDE
estimates needed to realize that compactness inside the Navier--Stokes local
geometry.  In particular, the present paper does not prove the following
inputs unconditionally:
\begin{enumerate}[label=(\roman*),leftmargin=2em]
    \item the construction of canonical clean defect spaces, gauge spaces, and
    gauge-compatible detector families directly from suitable weak solutions;
    \item the local-to-clean quotient lifting and the corresponding transfer
    comparison in \Cref{ass:local-to-clean-transfer-comparison};
    \item the localized residual budget in
    \Cref{ass:localized-residual-budget}, and its enhanced-tail analogue in
    \Cref{ass:enhanced-tail-localized-residual-budget};
    \item the enhanced-tail comparison in
    \Cref{ass:enhanced-tail-local-to-clean-transfer-comparison};
    \item projection-tail and harmonic-tail approximation estimates strong
    enough to remove the additive errors in
    \Cref{thm:conditional-enhanced-tail-intrinsic-comparison};
    \item a coercive lower bound for a concrete pressure, flux, or gate/tax
    functional in the original PDE variables.
\end{enumerate}

These are not cosmetic omissions.  They mark the points at which the abstract
finite-window architecture must be connected to genuine Navier--Stokes
estimates, especially pressure decomposition, cutoff commutators, quotient
selection, and same-gauge comparison.

\subsection{PDE-facing next steps}

The most immediate continuation is to prove pressure-tail approximation in a
fixed localized intrinsic package.  Concretely, one should estimate the clean
projection error and the harmonic residual on a common representative and show
that
\[
    \Delta_{{\rm proj},N}\to0,
    \qquad
    \Delta_{{\rm harm},M}\to0,
\]
under hypotheses that are natural for suitable weak solutions.  This would
turn the conditional comparison in
\Cref{thm:conditional-enhanced-tail-intrinsic-comparison} into a usable bridge
between the enhanced-tail geometry and the original intrinsic quotient
distance.

A second direction is to justify the enhanced-tail local-to-clean comparison
and residual-budget assumptions in a concrete intrinsic model.  The
sub-budget decomposition developed above reduces this task to three
chart-visibility channels and four finite-window residual budgets: chart
mismatch, localization leakage, reproduction drift, and gate/tax mismatch.
The role of these sub-budgets is to identify exactly where a PDE estimate is
needed rather than to hide it inside a single global assumption.

A third, more coercive direction is to develop a normalized pressure/tax model
that gives an actual positive lower bound for a detector sector along
expanding finite windows.  This is the step that would move the framework from
a finite-window accounting theorem toward a mechanism capable of excluding
specific persistent defect scenarios.

\subsection{Limitations}

No result in this paper proves Navier--Stokes global regularity, rules out all
possible singularity mechanisms, or constructs a singular solution.  The
paper also does not claim a symbolic-dynamics, universal-computation, or
undecidability theorem.  Its contribution is narrower and more structural: it
turns a persistent finite-window defect into an explicit accounting problem,
where the possible escape routes are cleanly separated into detector sectors,
pressure-tail errors, and localized residual budgets.  The usefulness of the
framework therefore depends on whether the conditional inputs isolated above
can be proved in intrinsic Navier--Stokes geometries.


\begin{thebibliography}{99}

\bibitem{Leray1934}
J.~Leray,
\newblock Sur le mouvement d'un liquide visqueux emplissant l'espace,
\newblock \emph{Acta Mathematica} \textbf{63} (1934), 193--248.
\newblock DOI: \url{https://doi.org/10.1007/BF02547354}.

\bibitem{Hopf1951}
E.~Hopf,
\newblock {\"U}ber die Anfangswertaufgabe f{\"u}r die hydrodynamischen Grundgleichungen,
\newblock \emph{Mathematische Nachrichten} \textbf{4} (1950/51), no.~1--6, 213--231.
\newblock DOI: \url{https://doi.org/10.1002/mana.3210040121}.

\bibitem{Scheffer1976}
V.~Scheffer,
\newblock Partial regularity of solutions to the Navier--Stokes equations,
\newblock \emph{Pacific Journal of Mathematics} \textbf{66} (1976), no.~2, 535--552.
\newblock DOI: \url{https://doi.org/10.2140/pjm.1976.66.535}.

\bibitem{Scheffer1977}
V.~Scheffer,
\newblock Hausdorff measure and the Navier--Stokes equations,
\newblock \emph{Communications in Mathematical Physics} \textbf{55} (1977), no.~2, 97--112.
\newblock DOI: \url{https://doi.org/10.1007/BF01626512}.

\bibitem{CKN1982}
L.~Caffarelli, R.~Kohn, and L.~Nirenberg,
\newblock Partial regularity of suitable weak solutions of the Navier--Stokes equations,
\newblock \emph{Communications on Pure and Applied Mathematics} \textbf{35} (1982), no.~6, 771--831.
\newblock DOI: \url{https://doi.org/10.1002/cpa.3160350604}.

\bibitem{SohrWahl1986}
H.~Sohr and W.~von Wahl,
\newblock On the regularity of the pressure of weak solutions of Navier--Stokes equations,
\newblock \emph{Archiv der Mathematik} \textbf{46} (1986), 428--439.
\newblock DOI: \url{https://doi.org/10.1007/BF01210782}.

\bibitem{Lin1998}
F.-H.~Lin,
\newblock A new proof of the Caffarelli--Kohn--Nirenberg theorem,
\newblock \emph{Communications on Pure and Applied Mathematics} \textbf{51} (1998), no.~3, 241--257.
\newblock DOI: \url{https://doi.org/10.1002/(SICI)1097-0312(199803)51:3<241::AID-CPA2>3.0.CO;2-A}.

\bibitem{ConstantinETiti1994}
P.~Constantin, W.~E, and E.~S. Titi,
\newblock Onsager's conjecture on the energy conservation for solutions of Euler's equation,
\newblock \emph{Communications in Mathematical Physics} \textbf{165} (1994), no.~1, 207--209.
\newblock DOI: \url{https://doi.org/10.1007/BF02099744}.

\bibitem{Eyink1994}
G.~L. Eyink,
\newblock Energy dissipation without viscosity in ideal hydrodynamics. I. Fourier analysis and local energy transfer,
\newblock \emph{Physica D} \textbf{78} (1994), no.~3--4, 222--240.
\newblock DOI: \url{https://doi.org/10.1016/0167-2789(94)90117-1}.

\bibitem{DuchonRobert2000}
J.~Duchon and R.~Robert,
\newblock Inertial energy dissipation for weak solutions of incompressible Euler and Navier--Stokes equations,
\newblock \emph{Nonlinearity} \textbf{13} (2000), no.~1, 249--255.
\newblock DOI: \url{https://doi.org/10.1088/0951-7715/13/1/312}.

\bibitem{SereginSverak2002}
G.~A. Seregin and V.~\v{S}ver\'ak,
\newblock Navier--Stokes equations with lower bounds on the pressure,
\newblock \emph{Archive for Rational Mechanics and Analysis} \textbf{163} (2002), no.~1, 65--86.
\newblock DOI: \url{https://doi.org/10.1007/s002050200199}.

\bibitem{ESS2003}
L.~Escauriaza, G.~Seregin, and V.~\v{S}ver\'ak,
\newblock $L_{3,\infty}$-solutions of Navier--Stokes equations and backward uniqueness,
\newblock \emph{Russian Mathematical Surveys} \textbf{58} (2003), no.~2, 211--250.
\newblock DOI: \url{https://doi.org/10.1070/RM2003v058n02ABEH000609}.

\bibitem{KukavicaZiane2006}
I.~Kukavica and M.~Ziane,
\newblock One component regularity for the Navier--Stokes equations,
\newblock \emph{Nonlinearity} \textbf{19} (2006), no.~2, 453--469.
\newblock DOI: \url{https://doi.org/10.1088/0951-7715/19/2/012}.

\bibitem{CaoTiti2011}
C.~Cao and E.~S. Titi,
\newblock Global regularity criterion for the 3D Navier--Stokes equations involving one entry of the velocity gradient tensor,
\newblock \emph{Archive for Rational Mechanics and Analysis} \textbf{202} (2011), no.~3, 919--932.
\newblock DOI: \url{https://doi.org/10.1007/s00205-011-0439-6}.

\bibitem{SereginLectureNotes}
G.~A. Seregin,
\newblock \emph{Lecture Notes on Regularity Theory for the Navier--Stokes Equations},
\newblock World Scientific, Singapore, 2014.
\newblock DOI: \url{https://doi.org/10.1142/9314}.
\bibitem{CheminZhang2016}
J.-Y.~Chemin and P.~Zhang,
\newblock On the critical one component regularity for 3-D Navier--Stokes system,
\newblock \emph{Annales scientifiques de l'\'{E}cole Normale Sup\'{e}rieure, S\'{e}rie 4} \textbf{49} (2016), no.~1, 131--167.
\newblock DOI: \url{https://doi.org/10.24033/asens.2278}.
\bibitem{CheminZhangZhang2017}
J.-Y.~Chemin, P.~Zhang, and Z.~Zhang,
\newblock On the critical one component regularity for 3-D Navier--Stokes system: general case,
\newblock \emph{Archive for Rational Mechanics and Analysis} \textbf{224} (2017), no.~3, 871--905.
\newblock DOI: \url{https://doi.org/10.1007/s00205-017-1089-0}.

\bibitem{HanLeiLiZhao2019}
B.~Han, Z.~Lei, D.~Li, and N.~Zhao,
\newblock Sharp one component regularity for Navier--Stokes,
\newblock \emph{Archive for Rational Mechanics and Analysis} \textbf{231} (2019), no.~2, 939--970.
\newblock DOI: \url{https://doi.org/10.1007/s00205-018-1292-7}.

\bibitem{BarkerPrange2021}
T.~Barker and C.~Prange,
\newblock Quantitative regularity for the Navier--Stokes equations via spatial concentration,
\newblock \emph{Communications in Mathematical Physics} \textbf{385} (2021), no.~2, 717--792.
\newblock DOI: \url{https://doi.org/10.1007/s00220-021-04122-x}.

\bibitem{KangNguyen2023}
K.~Kang and D.~D. Nguyen,
\newblock Local regularity criteria in terms of one velocity component for the Navier--Stokes equations,
\newblock \emph{Journal of Mathematical Fluid Mechanics} \textbf{25} (2023), no.~1, article no.~10, 15 pp.
\newblock DOI: \url{https://doi.org/10.1007/s00021-022-00754-8}.
\bibitem{AlbrittonBarkerPrange2023}
D.~Albritton, T.~Barker, and C.~Prange,
\newblock Epsilon regularity for the Navier--Stokes equations via weak--strong uniqueness,
\newblock \emph{Journal of Mathematical Fluid Mechanics} \textbf{25} (2023), no.~3, article no.~49, 12 pp.
\newblock DOI: \url{https://doi.org/10.1007/s00021-023-00780-0}.
\bibitem{YuOneComponent2026}
R.~Yu,
\newblock Finite-Scale One-Component Regularity via Harmonic Pressure for the 3D Navier--Stokes Equations,
\newblock arXiv:2606.08352 [math.AP], 2026.
\newblock DOI: \url{https://doi.org/10.48550/arXiv.2606.08352}.
\bibitem{YuStrict2026}
R.~Yu,
\newblock Strict 2.5D Shadows for One-Component Navier--Stokes Regularity,
\newblock arXiv:2606.11720 [math.AP], 2026.
\newblock DOI: \url{https://doi.org/10.48550/arXiv.2606.11720}.
\bibitem{YuSchur2026}
R.~Yu,
\newblock Schur Visibility and Anti-Phantom Reduction in One-Component Navier--Stokes Degeneration,
\newblock arXiv:2606.12267 [math.AP], 2026.
\newblock DOI: \url{https://doi.org/10.48550/arXiv.2606.12267}.
\bibitem{YuInvisible2026}
R.~Yu,
\newblock Invisible Defect Cascades for Navier--Stokes Regularity,
\newblock arXiv:2606.12756 [math.AP], 2026.
\newblock DOI: \url{https://doi.org/10.48550/arXiv.2606.12756}.
\end{thebibliography}
\end{document}